\documentclass{article}
\usepackage{latexsym,amsfonts,amsmath,amsthm,amssymb,makeidx}
\usepackage[title]{appendix}
\usepackage{CJK,CJKnumb,CJKulem,times,ifthen,mathrsfs,latexsym,amsfonts, color}
\usepackage{amsmath,amsthm,makeidx,fontenc,amssymb,bm,graphicx,psfrag,listings, curves,extarrows}
\usepackage[driver=pdftex,lmargin=3.2cm,rmargin=3.2cm,heightrounded=true,centering]{geometry}
\usepackage{authblk}
\let\oldbibliography\thebibliography
\renewcommand{\thebibliography}[1]{
\oldbibliography{#1}
\setlength{\itemsep}{0pt}
}
\makeindex
\newtheorem{definition}{Definition}[section]
\newtheorem{theorem}{Theorem}[section]
\newtheorem{lemma}{Lemma}[section]

\newtheorem{prop}{Proposition}[section]
\newtheorem{remark}{Remark}[section]

\newcommand{\bt}{\begin{theorem}}
\newcommand{\et}{\end{theorem}}
\newcommand{\bl}{\begin{lemma}}
\newcommand{\el}{\end{lemma}}
\newcommand{\bd}{\begin{definition}}
\newcommand{\ed}{\end{definition}}
\newcommand{\bc}{\begin{corollary}}
\newcommand{\ec}{\end{corollary}}
\newcommand{\bp}{\begin{proof}}
\newcommand{\ep}{\end{proof}}
\newcommand{\bx}{\begin{example}}
\newcommand{\ex}{\end{example}}
\newcommand{\bi}{\begin{exercise}}
\newcommand{\ei}{\end{exercise}}
\newcommand{\bo}{\begin{prop}}
\newcommand{\eo}{\end{prop}}
\newcommand{\br}{\begin{remark}}
\newcommand{\er}{\end{remark}}
\newcommand{\be}{\begin{equation}}
\newcommand{\ee}{\end{equation}}
\newcommand{\ba}{\begin{align}}
\newcommand{\ea}{\end{align}}
\newcommand{\bn}{\begin{enumerate}}
\newcommand{\en}{\end{enumerate}}
\newcommand{\bg}{\begin{align*}}
\newcommand{\bcs}{\begin{cases}}
\newcommand{\ecs}{\end{cases}}

\newcommand{\bean}{\begin{eqnarray*}}
\newcommand{\eean}{\end{eqnarray*}}

\numberwithin{equation}{section}
\begin{document}
\title{{\bf Infinite time blow-up for half-harmonic map flow from $\mathbb{R}$ into $\mathbb{S}^1$}
\footnote{J. Wei is partially supported by NSERC of Canada, Y. Zheng is partially supported by NSF of China
(11301374) and China Scholarship Council (CSC).}}

\author[1]{Yannick Sire\thanks{sire@math.jhu.edu}}
\author[2]{Juncheng Wei\thanks{jcwei@math.ubc.ca}}
\author[3]{Youquan Zheng\thanks{zhengyq@tju.edu.cn}}
\affil[1]{Department of Mathematics, Johns Hopkins University, 404 Krieger Hall, 3400 N. Charles Street,
Baltimore, MD 21218, USA}
\affil[2]{Department of Mathematics, University of British Columbia, Vancouver V6T 1Z2, Canada}
\affil[3]{School of Mathematics, Tianjin University, Tianjin 300072, P. R. China}
\date{}
\maketitle
\begin{center}
\begin{minipage}{120mm}
\begin{center}{\bf Abstract}\end{center}
We study infinite time blow-up phenomenon for the half-harmonic map flow
\begin{equation}\label{e:main00}
\left\{\begin{array}{ll}
        u_t = -(-\Delta)^{\frac{1}{2}}u + \left(\frac{1}{2\pi}\int_{\mathbb{R}}\frac{|u(x)-u(s)|^2}{|x-s|^2}ds\right)u\quad\text{ in }\mathbb{R}\times (0, \infty),\\
        u(\cdot, 0) = u_0\quad\text{ in }\mathbb{R},
       \end{array}
\right.
\end{equation}
with a function $u:\mathbb{R}\times [0, \infty)\to \mathbb{S}^1$. Let $q_1,\cdots, q_k$ be distinct points in $\mathbb{R}$, there exist an initial datum $u_0$ and smooth functions $\xi_j(t)\to q_j$, $0<\mu_j(t)\to 0$, as $t\to +\infty$, $j = 1, \cdots, k$, such that the solution $u_q$ of Problem (\ref{e:main00}) has the form
\begin{equation*}
u_q =\omega_\infty +\sum_{j= 1}^k \left(\omega (\frac{x-\xi_j(t)}{\mu_j(t)} )-\omega_\infty \right)+\theta(x, t),
\end{equation*}
where $\omega$ is the canonical least energy half-harmonic map, $\omega_\infty=\begin{pmatrix}
     0\\
     1
\end{pmatrix} $, $\theta(x, t)\to 0$ as $t\to +\infty$, uniformly away from the points $q_j$. In addition, the parameter functions $\mu_j(t)$ decay to $0$ exponentially.
\vskip0.10in
\end{minipage}
\end{center}
\vskip0.10in
\section{Introduction}
Harmonic maps $u: \Omega \to \mathbb S^\ell$ into the sphere are critical points of the Dirichlet energy
$$
\mathcal L(u)=\frac12 \int_\Omega |\nabla u|^2\,dx.
$$
They play a crucial role in physics and geometry. When the domain $\Omega$ is a subset of $\mathbb R^2$ the Lagrangian $\mathcal L(u)$ is conformally invariant and this plays a crucial role in the regularity theory of such maps (see Helen \cite{helein}, Riviere \cite{riviere} and references therein). The theory has been generalized to even-dimensional domains whose critical points are called poly-harmonic maps.

In the recent years, many authors are interested in the analog of Dirichlet energy in odd-dimensional cases, for instance, Da Lio \cite{DaLioCVPDE2013, DaLioAIHAN2015}, Da Lio-Riviere \cite{DaLioAdvMath2011, DaLioAnalPDE2011}, Millot-Sire \cite{MillotSireARMA2015}, Schikorra \cite{SchikorraJDE2012} and the references therein. In these works, a special but quite interesting case is the so-called half-harmonic maps from $\mathbb{R}$ into $\mathbb{S}^{1}$ which can be defined as critical points of the following line energy
\begin{equation}\label{e:linearenergy}
\mathcal{L}(u) = \frac{1}{2}\int_{\mathbb{R}}|(-\Delta_{\mathbb{R}})^{\frac{1}{4}}u|^2dx.
\end{equation}
The functional $\mathcal{L}$ is invariant under the M\"{o}bius group which is the
trace of conformal maps keeping invariant the half-space $\mathbb{R}^2_+$. Half-harmonic maps have deep connections to minimal surfaces with free boundary, see \cite{FraserSchoenInventMath2016}, \cite{jost2016qualitative}, \cite{laurain2017regularity} and \cite{MillotSireARMA2015}. Computing the associated Euler-Lagrange equation for (\ref{e:linearenergy}), it is easy to see that if $u:\mathbb{R}\to \mathbb{S}^1$ is a half-harmonic map, then $u$ satisfies
\begin{equation}\label{e:halfharmonicequation}
(-\Delta)^{\frac{1}{2}}u(x) = \left(\frac{1}{2\pi}\int_{\mathbb{R}}\frac{|u(x)-u(s)|^2}{|x-s|^2}ds\right)u(x)\text{ in }\mathbb{R}.
\end{equation}
In this paper , we investigate the half-harmonic map flow
\begin{equation}\label{e:main}
\left\{\begin{array}{ll}
        u_t = -(-\Delta)^{\frac{1}{2}}u + \left(\frac{1}{2\pi}\int_{\mathbb{R}}\frac{|u(x)-u(s)|^2}{|x-s|^2}ds\right)u\quad\text{ in }\mathbb{R}\times (0, \infty),\\
        u(\cdot, 0) = u_0\quad\text{ in }\mathbb{R}
       \end{array}
\right.
\end{equation}
for a function $u:\mathbb{R}\times [0, \infty)\to \mathbb{S}^1$ and $u_0:\mathbb{R}\to \mathbb{S}^1$ is a given smooth map.

A special role in our construction is played by a canonical example of half-harmonic map and it was proved in \cite{MillotSireARMA2015} that
if $u\in \dot{H}^{1/2}(\mathbb{R},\mathbb{S}^1)$ is a non-constant entire half-harmonic map into $\mathbb{S}^1$ and $u^e$ is its harmonic extension to $\mathbb{R}^2_+$, then there exist $d\in\mathbb{N}$, $\vartheta\in \mathbb{R}$, $\{\lambda_k\}_{k=1}^d\subset (0, \infty)$ and $\{a_k\}_{k=1}^d\subset \mathbb{R}$ such that $u^e(z)$ or its complex conjugate has the form
$$
e^{i\vartheta}\prod_{k=1}^d\frac{\lambda_k(z-a_k)-i}{\lambda_k(z-a_k)+i}.
$$
Moreover,
$$
\mathcal{E}(u,\mathbb{R}) = [u]^2_{H^{1/2}(\mathbb{R})}=\frac{1}{2}\int_{\mathbb{R}^2_+}|\nabla u^e|^2dz = \pi d.
$$
This result shows that the function $\omega:\mathbb{R}\to\mathbb{S}^1$
\begin{equation}\label{e:halfharmonicmap}
x\to \begin{pmatrix}
\frac{2x}{x^2+1}\\
\frac{x^2-1}{x^2+1}
\end{pmatrix}
\end{equation}
is a half-harmonic map which corresponds to the case $\vartheta = 0$, $d=1$, $\lambda_1=1$ and $a_1=0$. We denote
$$ \omega_\infty = \begin{pmatrix}
0\\
1
\end{pmatrix}.
$$
Note that $\omega$ is invariant under dilation, translation and rotation, i.e., for $Q_\alpha=\begin{pmatrix}
     \cos\alpha & -\sin\alpha \\
     \sin\alpha & \cos\alpha
\end{pmatrix}\in O(2)$, $q\in \mathbb{R}$ and $\lambda\in\mathbb{R}^+$, the function
\begin{equation*}
U = Q_\alpha\omega\left(\frac{x-q}{\lambda}\right) = \begin{pmatrix}
     \cos\alpha & -\sin\alpha \\
     \sin\alpha & \cos\alpha
\end{pmatrix}\omega\left(\frac{x-q}{\lambda}\right)
\end{equation*}
satisfies the half-harmonic map equation (\ref{e:halfharmonicequation}).
Differentiating with $\alpha$, $q$ and $\lambda$ respectively and set $\alpha = 0$, $q=0$, $\lambda = 1$, we obtain that the functions
\begin{equation}\label{e:kernel}
Z_1(x) =
\begin{pmatrix}
     \frac{1-x^2}{x^2+1}\\
     \frac{2x}{x^2+1}
\end{pmatrix},\quad
Z_2(x) = \begin{pmatrix}
     \frac{2(x^2-1)}{(x^2+1)^2}\\
     \frac{-4x}{(x^2+1)^2}
\end{pmatrix},\quad
Z_3(x) =
\begin{pmatrix}
     \frac{2x(x^2-1)}{(x^2+1)^2}\\
     \frac{-4x^2}{(x^2+1)^2}
\end{pmatrix}
\end{equation}
satisfy the linearized equation of (\ref{e:halfharmonicequation}) at the solution $\omega$ defined by
\begin{eqnarray}\label{e:linearized1}
\nonumber (-\Delta)^{\frac{1}{2}}v(x) &=& \left(\frac{1}{2\pi}\int_{\mathbb{R}}\frac{|\omega(x)-\omega(s)|^2}{|x-s|^2}ds\right)v(x)\\
\quad\quad\quad &&+ \left(\frac{1}{\pi}\int_{\mathbb{R}}\frac{(\omega(x)-\omega(s))\cdot(v(x) -v(s))}{|x-s|^2}ds\right)\omega(x)
\end{eqnarray}
for $v:\mathbb{R}\to T_U\mathbb{S}^1$.
Our main result is
\begin{theorem}\label{t:main}
Let $q_1,\cdots, q_k$ be distinct points in $\mathbb{R}$, there exist an initial datum $u_0$ and smooth functions $\xi_j(t)\to q_j$, $0<\mu_j(t)\to 0$, as $t\to +\infty$, $j = 1, \cdots, k$, such that the solution $u_q$ of Problem (\ref{e:main}) has the form
\begin{equation*}
u_q =\omega_\infty + \sum_{j= 1}^k \left( \omega (\frac{x-\xi_j(t)}{\mu_j(t)} )- \omega_\infty\right) +\theta(x, t),
\end{equation*}
where $\theta(x, t)\to 0$ as $t\to +\infty$, uniformly away from the points $q_j$. In addition, the parameter functions $\mu_j(t)$ decay to $0$ exponentially.
\end{theorem}

This result is in striking contrast with what happens for the classical harmonic map flow in the conformal dimension as investigated in \cite{Chang-Ding-Ye, chen-struwe, Lin, Lin-Wang1, Topping} for instance. Indeed, in this latter case, the flow exhibits blow up in finite time. For finite-time blow up for classical harmonic map flows there have been intensive research in recent years. We refer to Davila-del Pino-Wei \cite{davila2017singularity}, Raphael-Schweyer \cite{rs1, rs2} and the nice book Lin-Wang \cite{Lin-Wang2} and the references therein.  On the other hand for half-harmonic map flows, our computation and result  suggests  that the blow-up occurs only at infinity. We conjecture that finite time blow up does not occur for half-harmonic map flows.  This is a purely non-local phenomenon. Similar phenomena has been observed for higher-degree co-rotational harmonic map flows into ${\mathbb S}^2$ (see Guan-Gustafson-Tsai \cite{Guan-Gustafson-Tsai}).

\medskip

In order to prove our theorem, we develop  a fractional {\it inner-outer gluing scheme} for fractional evolution problems. It is known that the inner-outer gluing scheme is useful in constructing finite or infinite dimensional concentrating solutions in  nonlinear elliptic problems. See for example,  del Pino-Kowalczyk-Wei \cite{delpinoKWeiCPAM2007, delwei2011Degiorgi, delkowalczykweijdg2013entire}. Recently  this method has also been applied in many parabolic problems, for example, the infinite time blow-up of critical nonlinear heat equation Cortazar-del Pino-Musso  \cite{cortazar2016green}, del Pino-Musso-Wei \cite{del2017infinite}, the singularity formation for two-dimensional harmonic map flow \cite{davila2017singularity}, type II ancient solutions for the Yamabe flow del Pino-Daskalopoulos-Sesum \cite{del2012type} and ancient solutions for the Allen-Cahn flow del Pino-Gkikas \cite{del2017ancient, del2017ancientAHNL}.  For fractional elliptic problems finite dimensional gluing scheme has been developed in Davila-del Pino-Wei\cite{ddw-JDE}. This paper takes the first step in developing the gluing scheme for fractional evolution problems. One of the key ingredients of the gluing methods is Liouville type theorem and a priori estimates.  Besides the nondegeneracy of the solutions proved in \cite{sire2017nondegeneracy}, we also need many new  Liouville type theorems for fractional evolution problem. One Liouville theorem is on the classification of  ancient solutions of linear fractional heat equations, which may be of independent interest\footnote{ We thank Professor Dong Li for stimulating discussions on the proof}:

\begin{theorem}
(Lemma \ref{l4.3}.) Suppose $u = u(x, t):\mathbb{R}\times (-\infty, 0] \to \mathbb{R}$ satisfies $\sup_{(x, t)\in \mathbb{R}\times (-\infty, 0]}|x|^a|u(x, t)|\leq C$
for some $C > 0$ and $0 < a < 1$.
If $u$ is an ancient solution to the equation
\begin{equation*}
\partial_t u = -\Lambda^\alpha u, \ \  (x, t) \in \mathbb{R}\times (-\infty, 0],
\end{equation*}
where $\Lambda =  (-\Delta)^{\frac{1}{2}}$ and $0 < \alpha \leq 2$, then $u \equiv 0$.
\end{theorem}

\medskip

Finally we mention  several works in the literature on blow-up problems for fractional or non-fractional dispersive (NLS or wave) equations (\cite{BHL, KLR, LS, KS, MRR1, MRR2, RR}) in which dispersive estimates are heavily employed. Our proof here is more parabolic in nature.

\section{Sketch of the Proof}

The proof of Theorem \ref{t:main} is quite long. We divide it into the following six steps: (For simplicity we consider the case $k = 1$).

\medskip

\noindent
{\bf Step 1. Construction of approximation}. Given a point $q\in\mathbb{R}$, we are looking for a solution $u(x,t)$ of form
\begin{equation*}
u(x, t) \approx U(x, t): = U_{\mu, \xi}(x) = \omega\left(\frac{x-\xi(t)}{\mu(t)}\right)
\end{equation*}
with
\begin{equation*}
\lim_{t\to +\infty}\xi(t) = q, \quad \lim_{t\to +\infty}\mu(t) = 0.
\end{equation*}
For convenience, we write $u = U  + \Pi_{U^\perp}\varphi + a(\Pi_{U^\perp}\varphi)U$ for a free function $\varphi:\mathbb{R}\times [t_0,\infty)\to \mathbb{R}^2$ where
$$
\Pi_{U^\perp}\varphi:=\varphi-(\varphi\cdot U)U,\quad a(\Pi_{U^\perp}\varphi):=\sqrt{1+(\varphi\cdot U)^2-|\varphi|^2}-1 = \sqrt{1-|\Pi_{U^\perp}\varphi|^2}-1.
$$
Since $u$ is a solution of (\ref{e:main}),
$$
0 = S(U+\Pi_{U^\perp}\varphi + a(\Pi_{U^\perp}\varphi)U):=\left(-u_t -(-\Delta)^{\frac{1}{2}}u + \left(\frac{1}{2\pi}\int_{\mathbb{R}}\frac{|u(x)-u(s)|^2}{|x-s|^2}ds\right)u\right)\Bigg|_{u = U+\Pi_{U^\perp}\varphi + a(\Pi_{U^\perp}\varphi)U},
$$
the main problem is equivalent to the following equation for $\varphi$,
\begin{equation}\label{e1:01}
0=-U_t -\partial_t\Pi_{U^\perp}\varphi + L_U(\Pi_{U^\perp}\varphi)+N_U(\Pi_{U^\perp}\varphi)+b(\Pi_{U^\perp}\varphi)U
\end{equation}
where
\begin{equation*}
\begin{aligned}
L_U(\Pi_{U^\perp}\varphi) &= -(-\Delta)^{\frac{1}{2}}\Pi_{U^\perp}\varphi + \left(\frac{1}{2\pi}\int_{\mathbb{R}}\frac{|U(x)-U(s)|^2}{|x-s|^2}ds\right)\Pi_{U^\perp}\varphi\\
\quad\quad\quad\quad &\quad + \left(\frac{1}{\pi}\int_{\mathbb{R}}\frac{(U(x)-U(s))\cdot(\Pi_{U^\perp}\varphi(x) -\Pi_{U^\perp}\varphi(s))}{|x-s|^2}ds\right)U(x)
\end{aligned}
\end{equation*}
and $N_U(\Pi_{U^\perp}\varphi)$ is a high order nonlinear term.

The term $U_t$ in (\ref{e1:01}) can be expressed as
\begin{equation*}
\begin{aligned}
-U_t = -\frac{\dot{\mu}}{\mu}Z_3(y)- \frac{\dot{\xi}}{\mu}Z_2(y):= \mathcal{E}_0(x, t) + \mathcal{E}_1(x, t),
\end{aligned}
\end{equation*}
with $\mathcal{E}_0(x, t) = -\frac{\dot{\mu}}{\mu}Z_3(y)$, $\mathcal{E}_1(x, t) = - \frac{\dot{\xi}}{\mu}Z_2(y)$ and $y = \frac{x-\xi(t)}{\mu(t)}$.
To get a better approximation away from $q$, we shall add a small term $\Phi^*:= \Phi^0[\mu,\xi](x, t) + Z^*(x, t)$ to $U$, where
$$
\Phi^0[\mu,\xi](x, t) = \begin{pmatrix}
\psi^0\\
0
\end{pmatrix},
$$
$\psi^0$ is a nonlocal (integral) term and $Z^*$ is a function independent of the parameter functions $\mu$, $\xi$ but just as a solution of the homogeneous half heat equation. The aim of $\Phi^0$ is to improve the decay rate of the error $-U_t$ and $Z^*$ is useful in computations of the dynamics.
Then the new linear error is
\begin{equation*}
\mathcal{E}^* = \mathcal{E}_{0}+\mathcal{E}_1 + (-\partial_t + L_U)[\Pi_{U^\perp}\Phi^*].
\end{equation*}

To get an approximation with error globally smaller  than $-U_t$, we use the self-similar transformation
$$
\Phi(x,t) = \phi(y, t),\quad y = \frac{x-\xi}{\mu},\quad \rho = |y|,
$$
and (\ref{e1:01}) becomes to
\begin{equation*}
\begin{aligned}
0= -\mu\partial_t\Phi + L_\omega(\phi) + \mu\mathcal{E}^* + b(x,t)\omega,\quad \phi\cdot \omega = 0.
\end{aligned}
\end{equation*}
An improvement of the approximation can be obtained if we solve the following time independent equation
\begin{equation}\label{e2:02}
\begin{aligned}
0 = L_\omega(\phi) + \mu\mathcal{E}^*, \quad\phi\cdot \omega = 0,\quad \lim_{|y|\to \infty}\phi(y, t) = 0\quad\text{in}\quad\mathbb{R}.
\end{aligned}
\end{equation}
The decay condition is added in order to not essentially modify the size of the error far away. From the nondegeneracy result of \cite{sire2017nondegeneracy}, a necessary condition is that $\mu\mathcal{E}^*$ is orthogonal to $Z_2(y)$ and $Z_3(y)$ in the $L^2$ sense. Now testing equation (\ref{e2:02}) with $Z_3(y)$ and integrating by parts, we obtain
\begin{equation*}
\mu\int_{\mathbb{R}}\mathcal{E}^*\cdot Z_3dy = 0.
\end{equation*}
We can achieve this approximately by setting $\mu_0(t) = e^{-\kappa_0 t}$ for a suitable positive constant $\kappa_0$.
Similarly, test the error $\mu\mathcal{E}^*$ against the function $Z_2(y)$, we have
\begin{equation*}
\mu\int_{\mathbb{R}}\mathcal{E}^*\cdot Z_2dy = 0.
\end{equation*}
From this condition, we get $\dot{\xi}(t) = 0$ and it is natural to set $\xi(t_0) = q$, then the first order approximation of $\xi(t)$ should be
\begin{equation*}
\xi_0(t) = q.
\end{equation*}

Now fix the parameter functions $\mu_0(t)$, $\xi_0(t)$, we write
$$
\mu(t) = \mu_0(t) + \lambda(t),\quad \xi(t) = \xi_0(t) +\xi_1(t) = q + \xi_1(t).
$$
We are looking for a small solution $\varphi$ of
\begin{equation}\label{e2:03}
\mathcal{E}^* -\partial_t\Pi_{U^\perp}\varphi + L_U(\Pi_{U^\perp}\varphi)+N_U(\Pi_{U^\perp}[\Phi^0 + Z^*+\varphi])+\tilde{b}(x,t)U = 0.
\end{equation}
In other words, let $t_0 > 0$, we want the function
$$
u = U+\Pi_{U^\perp}[\Phi^0 + Z^*+\varphi] + a(\Pi_{U^\perp}[\Phi^0 + Z^*+\varphi])U
$$
to solve the problem
\begin{equation*}
\left\{\begin{array}{ll}
        u_t = -(-\Delta)^{\frac{1}{2}}u + \left(\frac{1}{2\pi}\int_{\mathbb{R}}\frac{|u(x)-u(s)|^2}{|x-s|^2}ds\right)u\quad\text{ in }\mathbb{R}\times [t_0, \infty),\\
        u(\cdot, t_0) = u_0\quad\text{ in }\mathbb{R}
       \end{array}
\right.
\end{equation*}
when $t_0$ is sufficiently large. This provides a solution $u(x, t) = u(x, t-t_0)$ to the main problem (\ref{e:main}).
We will show the details in Section 3.

\medskip

\noindent
{\bf Step 2. The inner-outer gluing procedure}.
Let $\eta_0(s)$ be a smooth cut-off function with $\eta_0(s) = 1$ for $s < 1$ and $=0$ for $s > 2$. Take a sufficiently large constant
\begin{equation*}
R = e^{\rho t_0}
\end{equation*}
for $\rho\in (0, 1)$ sufficiently small and $t_0$ is the initial time.
We define
$$
\eta(x, t):=\eta_0\left(\frac{|x-\xi(t)|}{R\mu_0(t)}\right)
$$
and consider a function $\varphi(x, t)$ with form
\begin{equation}\label{e2:04}
\varphi(x, t) = \eta \tilde{\phi}(x, t) + \psi(x, t)
\end{equation}
for a function $\phi(y, t_0) = 0$, $\phi(y, t)\cdot \omega(y) \equiv 0$ and $\tilde{\phi}(x, t) = \phi\left(\frac{x-\xi(t)}{\mu_0(t)}, t\right)$, $y := \frac{x-\xi}{\mu_0}$.
Then $\varphi$ defined by (\ref{e2:04}) solves (\ref{e2:03}) if the pair $(\phi, \psi)$ satisfies the following evolution equations
\begin{equation}\label{e2:05}
\left\{
\begin{aligned}
&\mu_0(t)\partial_t \phi = L_\omega[\phi](y)+ \mu_0\Pi_{U^\perp}\mathcal{E}^*(\xi+\mu_0y, t) +\frac{2\frac{\mu_0}{\mu}}{1+\left|\frac{\mu_0}{\mu}y\right|^2}\Pi_{\omega^\perp}\psi +\cdots,\\
&\phi = 0\quad\text{in }B_{2R}(0)\times \{t_0\}
\end{aligned}
\right.
\end{equation}
and
\begin{equation}\label{e2:06}
\left\{
\begin{aligned}
&\partial_t \psi = (-\Delta)^{\frac{1}{2}}\psi+ (1-\eta)\frac{2}{1+\left|\frac{x-\xi}{\mu}\right|^2}\psi + (1-\eta)\Pi_{U^\perp}\mathcal{E}^*+\cdots\\
&\text{ in }\mathbb{R}\times [t_0, +\infty).
\end{aligned}
\right.
\end{equation}

(\ref{e2:05}) is the so-called inner problem and (\ref{e2:06}) is the outer problem. The strategy we use to solve this system is:
for a fixed function $\phi(y, t)$ in a suitable class and $\lambda$, $\xi$, we solve (\ref{e2:06}) for $\psi$ as an operator $\Psi = \Psi[\lambda, \xi, \dot{\lambda}, \dot{\xi}, \phi]$. Inserting this $\Psi$ into equation (\ref{e2:05}) and after change of variable, we obtain
\begin{equation}\label{e2:07}
\left\{
\begin{array}{ll}
\partial_\tau\phi = L_\omega[\phi](y) + H[\lambda,\xi,\dot{\lambda},\dot{\xi},\phi](y,t(\tau))\quad \text{in }B_{2R}(0)\times [\tau_0, +\infty),\\
\phi(\cdot, \tau_0) = 0\quad\text{in }B_{2R}(0).
\end{array}
\right.
\end{equation}
Then (\ref{e2:07}) is solved by the Contraction Mapping Theorem involving an inverse for the following linear problem
\begin{equation}\label{e2:08}
\left\{
\begin{array}{ll}
\partial_\tau\phi = L_\omega[\phi] + h(y, \tau)\quad \text{in }B_{2R}(0)\times [\tau_0, +\infty),\\
\phi(\cdot, \tau_0) = 0\quad\text{in }B_{2R}(0).
\end{array}
\right.
\end{equation}
Providing certain orthogonality conditions hold, we will find a solution $\phi$ of (\ref{e2:08}) which defined a linear operator of $h$ with good $L^\infty$-weighted estimates. See Section 4 for details.

\medskip

\noindent
{\bf Step 3. The outer problem}. We solve the outer problem (\ref{e2:06}) for $\psi$. Suppose
\begin{equation}\label{e0:decay}
(1+|y|)|\nabla\phi|\chi_{\{|y|\leq 2R\}} + |\phi|\leq e^{-t_0\varepsilon}\frac{\mu_0^\sigma(t)}{1+|y|^\alpha}
\end{equation}
holds for a given $0 < \alpha < 1$, small constant $\sigma >0$ and $\varepsilon > 0$. Using the heat kernel integral, we solve (\ref{e2:06})
and obtain the existence of a solution $\psi = \Psi[\lambda, \xi, \dot{\lambda}, \dot{\xi}, \phi]$ satisfying
\begin{equation*}
|\psi(x, t)|\lesssim e^{-\varepsilon t_0}\frac{\mu_0^{\sigma}(t)\log(1+|y|)}{1+|y|^\alpha}
\end{equation*}
and
\begin{equation*}
[\psi(x, t)]_{\eta, B_{3\mu_0(t)R}(\xi)}\lesssim e^{-\varepsilon t_0}\frac{\mu_0^{\sigma-\eta}(t)\log(1+|y|)}{1+|y|^{\alpha+\eta}}\quad\text{ for }|y| = \left|\frac{x-\xi(t)}{\mu_0(t)}\right|\leq 3R.
\end{equation*}
where $y=\frac{x-\xi}{\mu_{0}}$.
To ensure the solvability of (\ref{e2:06}), we need the term $(1-\eta)\Pi_{U^\perp}\mathcal{E}^*$ decay faster than $1/r$, this is the reason we add the nonlocal term $\Phi^0$ in the approximation step. See Section 5 for details.

After substituting $\psi = \Psi[\lambda,\xi,\dot{\lambda},\dot{\xi},\phi]$ into the inner problem (\ref{e2:05}) and using the change of variables $\frac{dt}{d\tau}=\mu_0(t)$, the full problem is reduced to the solvability of (\ref{e2:07}).

\medskip

\noindent
{\bf Step 4. Linear theory for (\ref{e2:07})}. To solve the nonlinear problem (\ref{e2:07}), we first consider the following linear parabolic problem
\begin{equation}\label{e2:10}
\left\{
\begin{aligned}
&\partial_\tau \phi = L_\omega[\phi](y) + h(y,\tau)\text{ in }B_{2R}(0)\times [\tau_0, \infty),\\
&\phi(\cdot, \tau_0) = 0\quad\text{in }B_{2R}(0),\\
&\phi(y,\tau)\cdot \omega(y) = 0 \quad\text{in }B_{2R}(0)\times [\tau_0, \infty)
\end{aligned}
\right.
\end{equation}
Assuming $\|h\|_{1+a,\nu, \eta} <+\infty$ and
\begin{equation*}
\int_{B_{2R}(0)}h(y,\tau)Z_j(y)dy = 0\quad\text{for all}\quad\tau\in [\tau_0,\infty),\quad j = 2,3,
\end{equation*}
we prove the existence of $\phi = \phi[h](y, \tau)$ defined on $\mathbb{R}\times [\tau_0,+\infty)$ satisfying (\ref{e2:10}) and
\begin{equation*}
(1+|y|)|\nabla_y \phi(y,\tau)|\chi_{\{|y|\leq R\}}+ |\phi(y,\tau)|\lesssim \tau^{-\nu}(1+|y|)^{-a}\|h\|_{1+a,\nu, \eta}, \tau\in [\tau_0,+\infty),y\in \mathbb{R}.
\end{equation*}
Note that in (\ref{e2:10}), we do not add boundary conditions, and the solution we find are defined in the whole space, so the nonlocal term
$-(-\Delta)^{\frac{1}{2}}\phi$ and other integrals on $\mathbb{R}$ is well defined.  To prove existence, we  use the blow-up argument developed in \cite{davila2017singularity} and prove several Liouville theorems. This is done in Section 6.1.

\medskip

\noindent
{\bf Step 5. The solvability condition for (\ref{e2:07})}.
From Step 4, we know that (\ref{e2:07}) is solvable within the set of functions $\phi$ satisfying (\ref{e0:decay}), provided $\xi$ and $\lambda$ are chosen such that
\begin{equation*}
\int_{B_{2R}}H[\lambda,\xi,\dot{\lambda},\dot{\xi},\phi](y,t(\tau))Z_l(y)dy = 0\text{ for all }\tau\geq \tau_0, l = 2,3.
\end{equation*}
These are achieved by adjusting $\lambda$ and $\xi$ in Section 6.2.

\medskip

\noindent
{\bf Step 6. The inner problem: gluing}. We solve the nonlinear problem (\ref{e2:07}) based on the linear theory for (\ref{e2:10}) and the Contraction Mapping Theorem in Section 6.3.

\section{Construction of a first approximate solution}
\subsection{Setting up the main problem}
To prove Theorem \ref{t:main}, we need a sufficiently good approximation to the exact solution of equation (\ref{e:main}). To keep notation to minimum, we will do this in the case $k=1$ and later indicate the necessary changes for the general case. Given a point $q\in\mathbb{R}$, we are looking for a solution $u(x,t)$ of the following equation
\begin{equation}\label{e:error}
S(u) = -u_t -(-\Delta)^{\frac{1}{2}}u + \left(\frac{1}{2\pi}\int_{\mathbb{R}}\frac{|u(x)-u(s)|^2}{|x-s|^2}ds\right)u = 0, \quad |u| = 1\quad\text{in }\mathbb{R}\times (t_0, \infty)
\end{equation}
satisfying
\begin{equation*}
u(x, t) \approx U(x, t): = U_{\mu, \xi}(x) = \omega\left(\frac{x-\xi(t)}{\mu(t)}\right),
\end{equation*}
where $\omega$ is the canonical least energy half-harmonic map defined in (\ref{e:halfharmonicmap}) and $t_0$ is a sufficiently large positive constant. We will find functions $\xi(t)$ and $\mu(t)$ of class $C^1[t_0,\infty)$ satisfying
\begin{equation*}
\lim_{t\to\infty}\xi(t) = q, \quad \lim_{t\to\infty}\mu(t) = 0
\end{equation*}
in such a way that there exists a solution $u(x, t)$ which blows up at $t = \infty$ and the point $q$ with a profile given at main order by $U$.

For a fixed $t > t_0$, $U$ is a half-harmonic map:
\begin{equation*}
-(-\Delta)^{\frac{1}{2}}U + \left(\frac{1}{2\pi}\int_{\mathbb{R}}\frac{|U(x)-U(s)|^2}{|x-s|^2}ds\right)U = 0,
\end{equation*}
hence $S(U) = -U_t$. We want to find a solution $u(x,t)$ of (\ref{e:error}) which is a small perturbation of $U$, $U+p$ and $|U+p|=1$ for suitable choices of the parameter functions $\mu$ and $\xi$. For convenience, we parameterize the admissible perturbation $p(x,t)$ in terms of a free function $\varphi:\mathbb{R}\times [t_0,\infty)\to \mathbb{R}^2$ with form
\begin{equation*}
p(\varphi):=\Pi_{U^\perp}\varphi + a(\Pi_{U^\perp}\varphi)U,
\end{equation*}
where
$$
\Pi_{U^\perp}\varphi:=\varphi-(\varphi\cdot U)U,\quad a(\Pi_{U^\perp}\varphi):=\sqrt{1+(\varphi\cdot U)^2-|\varphi|^2}-1 = \sqrt{1-|\Pi_{U^\perp}\varphi|^2}-1
$$
so that
$$
|U+p(\varphi)|^2=1
$$
holds. Thus, we need to find a small function $\varphi$ with values in $\mathbb{R}^2$ such that
$$
u=U+\Pi_{U^\perp}\varphi + a(\Pi_{U^\perp}\varphi)U
$$
satisfies (\ref{e:error}). By direct computations, we have
$$
S(U+\Pi_{U^\perp}\varphi + a(\Pi_{U^\perp}\varphi)U) = -U_t -\partial_t\Pi_{U^\perp}\varphi + L_U(\Pi_{U^\perp}\varphi)+N_U(\Pi_{U^\perp}\varphi)+b(\Pi_{U^\perp}\varphi)U
$$
where
\begin{equation*}
\begin{aligned}
L_U(\Pi_{U^\perp}\varphi) &= -(-\Delta)^{\frac{1}{2}}\Pi_{U^\perp}\varphi + \left(\frac{1}{2\pi}\int_{\mathbb{R}}\frac{|U(x)-U(s)|^2}{|x-s|^2}ds\right)\Pi_{U^\perp}\varphi\\
\quad\quad\quad\quad &\quad + \left(\frac{1}{\pi}\int_{\mathbb{R}}\frac{(U(x)-U(s))\cdot(\Pi_{U^\perp}\varphi(x) -\Pi_{U^\perp}\varphi(s))}{|x-s|^2}ds\right)U(x),
\end{aligned}
\end{equation*}
\begin{equation*}
\begin{aligned}
&N_U(\Pi_{U^\perp}\varphi)\\ &= \left(\frac{1}{\pi}\int_{\mathbb{R}}\frac{(a(x)U(x)-a(s)U(s))\cdot(U(x)+\Pi_{U^\perp}\varphi(x)-U(s) -\Pi_{U^\perp}\varphi(s))}{|x-s|^2}ds\right.\\ &\quad\quad\quad\quad\quad\quad\quad\left.+\frac{1}{\pi}\int_{\mathbb{R}}\frac{(U(x)-U(s))\cdot(\Pi_{U^\perp}\varphi(x) -\Pi_{U^\perp}\varphi(s))}{|x-s|^2}ds\right.\\ &\quad\quad\quad\quad\quad\quad\quad\left. +\frac{1}{2\pi}\int_{\mathbb{R}}\frac{(\Pi_{U^\perp}\varphi(x)-\Pi_{U^\perp}\varphi(s))\cdot(\Pi_{U^\perp}\varphi(x) -\Pi_{U^\perp}\varphi(s))}{|x-s|^2}ds\right.\\
&\quad\quad\quad\quad\quad\quad\quad +\left.\frac{1}{2\pi}\int_{\mathbb{R}}\frac{(a(x)U(x)-a(s)U(s))\cdot(a(x)U(x)-a(s)U(s))}{|x-s|^2}ds\right)\Pi_{U^\perp}\varphi\\
&\quad\quad\quad\quad\quad\quad\quad -aU_t - \frac{1}{\pi}\int_{\mathbb{R}}\frac{(a(x)-a(s))\cdot(U(x) - U(s))}{|x-s|^2}ds
\end{aligned}
\end{equation*}
and
\begin{equation*}
\begin{aligned}
& b(\Pi_{U^\perp}\varphi)\\ &= -(-\Delta)^{\frac{1}{2}}a - a_t\\
&\quad + \left(\frac{1}{2\pi}\int_{\mathbb{R}}\frac{|(U+\Pi_{U^\perp}\varphi + a(\Pi_{U^\perp}\varphi)U)(x)-(U+\Pi_{U^\perp}\varphi + a(\Pi_{U^\perp}\varphi)U)(s)|^2}{|x-s|^2}ds\right.\\ &\quad\quad\quad\quad\quad\quad\quad\quad\quad\quad\quad\quad\quad\quad\quad\quad \left. - \frac{1}{2\pi}\int_{\mathbb{R}}\frac{|U(x)-U(s)|^2}{|x-s|^2}ds\right)(1+a)\\
\quad\quad\quad &\quad - \left(\frac{1}{\pi}\int_{\mathbb{R}}\frac{(U(x)-U(s))\cdot(\Pi_{U^\perp}\varphi(x) -\Pi_{U^\perp}\varphi(s))}{|x-s|^2}ds\right).
\end{aligned}
\end{equation*}

A useful observation is that if $\varphi$ solves an equation of the form
\begin{equation*}\label{e:balanceequation}
-U_t -\partial_t\Pi_{U^\perp}\varphi + L_U(\Pi_{U^\perp}\varphi)+N_U(\Pi_{U^\perp}\varphi)+\tilde{b}(x,t)U = 0
\end{equation*}
for some scalar function $\tilde{b}(x,t)$ and $|\varphi| \leq \frac{1}{2}$ holds, then $u=U+\Pi_{U^\perp}\varphi + a(\Pi_{U^\perp}\varphi)U$ solves\ (\ref{e:error}). Indeed, $u$ satisfies the equation
\begin{equation*}
S(u) + b_0U = 0
\end{equation*}
with $b_0 = \tilde{b} - b(\Pi_{U^\perp}\varphi)$. Using the fact that $|u| = 1$, we have
\begin{equation*}
\begin{aligned}
-b_0(x,t)U\cdot u = S(u)\cdot u = -\frac{1}{2}\frac{d}{dt}|u|^2 &- \frac{1}{\pi}\int_{\mathbb{R}}\frac{(u(x)-u(s))}{|x-s|^2}ds\cdot u(x)\\
&+\left(\frac{1}{2\pi}\int_{\mathbb{R}}\frac{|u(x)-u(s)|^2}{|x-s|^2}ds\right)u\cdot u = 0.
\end{aligned}
\end{equation*}
On the other hand, since $U\cdot u = 1 + a(\Pi_{U^\perp}\varphi)$ and $|\varphi|\leq \frac{1}{2}$, we easily check that $|a(\Pi_{U^\perp}\varphi)|\leq \frac{1}{4}$. Hence $U\cdot u > 0$ and $b_0\equiv 0$.

Inserting $U$ into equation (\ref{e:error}), we compute the error of approximation as
\begin{equation*}
\begin{aligned}
S(U) &= -U_t = -\dot{\mu}\partial_\mu U_{\mu,\xi} - \dot{\xi}\partial_{\xi}U_{\mu, \xi} = \frac{\dot{\mu}}{\mu}\nabla\omega(y)\cdot y + \nabla\omega(y)\cdot\frac{\dot{\xi}}{\mu}\\
& = \frac{\dot{\mu}}{\mu}\left(- Z_3(y)\right)+\frac{\dot{\xi}}{\mu}\left(- Z_2(y)\right)= -\frac{\dot{\mu}}{\mu}Z_3(y)- \frac{\dot{\xi}}{\mu}Z_2(y)\\
& := \mathcal{E}_0(x, t) + \mathcal{E}_1(x, t).
\end{aligned}
\end{equation*}
Here $\mathcal{E}_0(x, t) = -\frac{\dot{\mu}}{\mu}Z_3(y)$, $\mathcal{E}_1(x, t) = - \frac{\dot{\xi}}{\mu}Z_2(y)$ and $y = \frac{x-\xi(t)}{\mu(t)}$.

\subsection{Some useful computations}
For a vector function $\varphi:\mathbb{R}\to \mathbb{R}^2$, the following result provides a convenient expression for the quantity $(-\partial_t + L_U)[\Pi_{U^\perp}\varphi]$.
\begin{lemma}\label{l:lemma2.1}
For a function $\varphi:\mathbb{R}\to \mathbb{R}^2$, we have
\begin{eqnarray*}
\nonumber L_U[\Pi_{U^\perp}\varphi] = \Pi_{U^\perp}\left[-(-\Delta)^{\frac{1}{2}}\varphi\right] + \tilde{L}_U[\varphi]
\end{eqnarray*}
where
\begin{eqnarray*}
&&\tilde{L}_U[\varphi]=\\
&&\quad \left(\frac{1}{2\pi}\int_{\mathbb{R}}\frac{|U(x)-U(s)|^2}{|x-s|^2}ds\right)\Pi_{U^\perp}\varphi\\
&&\quad -\frac{1}{\pi}\int_{\mathbb{R}}\frac{\left[(\varphi\cdot U)(x)-(\varphi\cdot U)(s)\right](U(x)-U(s))}{|x-s|^2}ds \\
&&\quad + \left(\frac{1}{\pi}\int_{\mathbb{R}}\frac{(U(x)-U(s))\left[(\varphi\cdot U)(x)-(\varphi\cdot U)(s)\right]}{|x-s|^2}ds\cdot U(x)\right)U(x).
\end{eqnarray*}
\end{lemma}
\begin{proof}
We have
\begin{equation*}
\begin{aligned}
& -(-\Delta)^{\frac{1}{2}}\Pi_{U^\perp}\varphi\\
& = -(-\Delta)^{\frac{1}{2}}\left[\varphi-(\varphi\cdot U)U\right] \\
& = -(-\Delta)^{\frac{1}{2}}\varphi +  (-\Delta)^{\frac{1}{2}}\left[(\varphi\cdot U)U\right]\\
& = -(-\Delta)^{\frac{1}{2}}\varphi +  (-\Delta)^{\frac{1}{2}}\left[(\varphi\cdot U)\right]U+\left[(\varphi\cdot U)\right](-\Delta)^{\frac{1}{2}}U\\
&\quad -\frac{1}{\pi}\int_{\mathbb{R}}\frac{\left[(\varphi\cdot U)(x)-(\varphi\cdot U)(s)\right](U(x)-U(s))}{|x-s|^2}ds\\
& = -(-\Delta)^{\frac{1}{2}}\varphi +  (-\Delta)^{\frac{1}{2}}\left[(\varphi\cdot U)\right]U+\left[(\varphi\cdot U)\right]\left[\left(\frac{1}{2\pi}\int_{\mathbb{R}}\frac{|U(x)-U(s)|^2}{|x-s|^2}ds\right)U\right]\\
&\quad -\frac{1}{\pi}\int_{\mathbb{R}}\frac{\left[(\varphi\cdot U)(x)-(\varphi\cdot U)(s)\right](U(x)-U(s))}{|x-s|^2}ds.
\end{aligned}
\end{equation*}
Furthermore, it holds that
\begin{equation*}
\begin{aligned}
& (-\Delta)^{\frac{1}{2}}\left[(\varphi\cdot U)\right] = (-\Delta)^{\frac{1}{2}}\left(\varphi\right)\cdot U + \varphi\cdot (-\Delta)^{\frac{1}{2}}\left(U\right)\\
&\quad\quad\quad\quad\quad\quad\quad\quad -\frac{1}{\pi}\int_{\mathbb{R}}\frac{\left(\varphi(x)-\varphi(s)\right)\cdot(U(x)-U(s))}{|x-s|^2}ds
\end{aligned}
\end{equation*}
and hence
\begin{equation*}
\begin{aligned}
& (-\Delta)^{\frac{1}{2}}\left[(\varphi\cdot U)\right]U = \left((-\Delta)^{\frac{1}{2}}\varphi\cdot U\right)U + \left(\varphi\cdot U\right)\left[\left(\frac{1}{2\pi}\int_{\mathbb{R}}\frac{|U(x)-U(s)|^2}{|x-s|^2}ds\right)\right]U\\
&\quad\quad\quad\quad\quad\quad\quad\quad\quad -\left(\frac{1}{\pi}\int_{\mathbb{R}}\frac{\left[\varphi(x)-\varphi(s)\right]\cdot(U(x)-U(s))}{|x-s|^2}ds\right)U.
\end{aligned}
\end{equation*}
Thus
\begin{equation*}
\begin{aligned}
& -(-\Delta)^{\frac{1}{2}}\Pi_{U^\perp}\varphi\\
& = -(-\Delta)^{\frac{1}{2}}\varphi +  \left((-\Delta)^{\frac{1}{2}}\varphi\cdot U\right)U + \left(\varphi\cdot U\right)\left[\left(\frac{1}{2\pi}\int_{\mathbb{R}}\frac{|U(x)-U(s)|^2}{|x-s|^2}ds\right)\right]U\\
&\quad -\left(\frac{1}{\pi}\int_{\mathbb{R}}\frac{\left[\varphi(x)-\varphi(s)\right]\cdot(U(x)-U(s))}{|x-s|^2}ds\right)U+\left[(\varphi\cdot U)\right]\left[\left(\frac{1}{2\pi}\int_{\mathbb{R}}\frac{|U(x)-U(s)|^2}{|x-s|^2}ds\right)U\right]\\
&\quad -\frac{1}{\pi}\int_{\mathbb{R}}\frac{\left[(\varphi\cdot U)(x)-(\varphi\cdot U)(s)\right](U(x)-U(s))}{|x-s|^2}ds\\
& = \Pi_{U^\perp}\left[-(-\Delta)^{\frac{1}{2}}\varphi\right] -\left(\frac{1}{\pi}\int_{\mathbb{R}}\frac{\left[\varphi(x)-\varphi(s)\right]\cdot(U(x)-U(s))}{|x-s|^2}ds\right)U\\
&\quad+2\left[(\varphi\cdot U)\right]\left[\left(\frac{1}{2\pi}\int_{\mathbb{R}}\frac{|U(x)-U(s)|^2}{|x-s|^2}ds\right)U\right]\\
&\quad -\frac{1}{\pi}\int_{\mathbb{R}}\frac{\left[(\varphi\cdot U)(x)-(\varphi\cdot U)(s)\right](U(x)-U(s))}{|x-s|^2}ds.
\end{aligned}
\end{equation*}
Since
\begin{equation*}
\begin{aligned}
&\frac{1}{\pi}\int_{\mathbb{R}}\frac{(U(x)-U(s))\cdot\left[\Pi_{U^\perp}\varphi(x) -\Pi_{U^\perp}\varphi(s)\right]}{|x-s|^2}ds\\
&=\frac{1}{\pi}\int_{\mathbb{R}}\frac{(U(x)-U(s))\cdot\left[(\varphi-(\varphi\cdot U)U)(x) -(\varphi-(\varphi\cdot U)U)(s)\right]}{|x-s|^2}ds\\
&=\frac{1}{\pi}\int_{\mathbb{R}}\frac{(U(x)-U(s))\cdot\left[\varphi(x) -\varphi(s)\right]}{|x-s|^2}ds\\
&\quad-\frac{1}{\pi}\int_{\mathbb{R}}\frac{(U(x)-U(s))\cdot\left[(\varphi\cdot U)(x)U(x) -(\varphi\cdot U)(s)U(s)\right]}{|x-s|^2}ds\\
&=\frac{1}{\pi}\int_{\mathbb{R}}\frac{(U(x)-U(s))\cdot\left[\varphi(x) -\varphi(s)\right]}{|x-s|^2}ds\\
&\quad - \frac{1}{\pi}\int_{\mathbb{R}}\frac{|U(x)-U(s)|^2}{|x-s|^2}ds (\varphi\cdot U)(x)\\
&\quad + \frac{1}{\pi}\int_{\mathbb{R}}\frac{(U(x)-U(s))\left[(\varphi\cdot U)(x)-(\varphi\cdot U)(s)\right]}{|x-s|^2}ds\cdot U(x),
\end{aligned}
\end{equation*}
we finally obtain
\begin{eqnarray*}
&&L_U[\Pi_{U^\perp}\varphi]\\
&&=\Pi_{U^\perp}\left[-(-\Delta)^{\frac{1}{2}}\varphi\right] -\left(\frac{1}{\pi}\int_{\mathbb{R}}\frac{\left[\varphi(x)-\varphi(s)\right]\cdot(U(x)-U(s))}{|x-s|^2}ds\right)U\\
&&\quad+2\left[(\varphi\cdot U)\right]\left[\left(\frac{1}{2\pi}\int_{\mathbb{R}}\frac{|U(x)-U(s)|^2}{|x-s|^2}ds\right)U\right]\\
&&\quad -\frac{1}{\pi}\int_{\mathbb{R}}\frac{\left[(\varphi\cdot U)(x)-(\varphi\cdot U)(s)\right](U(x)-U(s))}{|x-s|^2}ds + \left(\frac{1}{2\pi}\int_{\mathbb{R}}\frac{|U(x)-U(s)|^2}{|x-s|^2}ds\right)\Pi_{U^\perp}\varphi\\
&&\quad + \left(\frac{1}{\pi}\int_{\mathbb{R}}\frac{(U(x)-U(s))\cdot\left[\varphi(x) -\varphi(s)\right]}{|x-s|^2}ds\right)U(x)\\
\end{eqnarray*}
\begin{eqnarray*}
&&\quad - \left(\frac{1}{\pi}\int_{\mathbb{R}}\frac{|U(x)-U(s)|^2}{|x-s|^2}ds (\varphi\cdot U)(x)\right)U(x)\\
&&\quad + \left(\frac{1}{\pi}\int_{\mathbb{R}}\frac{(U(x)-U(s))\left[(\varphi\cdot U)(x)-(\varphi\cdot U)(s)\right]}{|x-s|^2}ds\cdot U(x)\right)U(x)\\
&&=\Pi_{U^\perp}\left[-(-\Delta)^{\frac{1}{2}}\varphi\right] + \left(\frac{1}{2\pi}\int_{\mathbb{R}}\frac{|U(x)-U(s)|^2}{|x-s|^2}ds\right)\Pi_{U^\perp}\varphi\\
&&\quad -\frac{1}{\pi}\int_{\mathbb{R}}\frac{\left[(\varphi\cdot U)(x)-(\varphi\cdot U)(s)\right](U(x)-U(s))}{|x-s|^2}ds \\
&&\quad + \left(\frac{1}{\pi}\int_{\mathbb{R}}\frac{(U(x)-U(s))\left[(\varphi\cdot U)(x)-(\varphi\cdot U)(s)\right]}{|x-s|^2}ds\cdot U(x)\right)U(x).
\end{eqnarray*}
This completes the proof.
\end{proof}
\begin{lemma}\label{l:lemma2.2}
For a general function $\varphi(x) = \begin{pmatrix}
     a(x) \\
     b(x)
\end{pmatrix}\in \mathbb{R}^2$, $x\in \mathbb{R}$, we have
\begin{eqnarray*}
&&\tilde{L}_U[\varphi](x,t)=\\
&&
\quad\quad\mu^{-1}\left[\frac{1}{\pi}\int_{\mathbb{R}}\frac{a(s)-a(x)}{s-x}\frac{2\left(\frac{s-\xi}{\mu}\right)}
{\left[\left(\frac{s-\xi}{\mu}\right)^2+1\right]^2}ds Z_2\left(\frac{x-\xi}{\mu}\right)\right.\\
&&\quad\quad\quad\quad\quad +\frac{1}{\pi}\int_{\mathbb{R}}\frac{b(s)-b(x) }{s-x}\frac{\left(\frac{s-\xi}{\mu}\right)^2-1}{\left[\left(\frac{s-\xi}{\mu}\right)^2+1\right]^2}ds Z_2\left(\frac{x-\xi}{\mu}\right)\\
&&\quad\quad\quad\quad\quad+\frac{1}{\pi}\int_{\mathbb{R}}\frac{a(s)-a(x)}{s-x}\frac{2\left(\frac{s-\xi}{\mu}\right)^2}
{\left[\left(\frac{s-\xi}{\mu}\right)^2+1\right]^2}ds Z_3\left(\frac{x-\xi}{\mu}\right)\\
&&\quad\quad\quad\quad\quad\left.+\frac{1}{\pi}\int_{\mathbb{R}}\frac{b(s)-b(x)}{s-x}
\frac{\left(\frac{s-\xi}{\mu}\right)\left(\left(\frac{s-\xi}{\mu}\right)^2-1\right)}
{\left[\left(\frac{s-\xi}{\mu}\right)^2+1\right]^2}ds Z_3\left(\frac{x-\xi}{\mu}\right)\right].
\end{eqnarray*}
\end{lemma}
\begin{proof}
For a general function $\varphi(x) = \begin{pmatrix}
     a(x) \\
     b(x)
\end{pmatrix}$,
we have
\begin{eqnarray*}
&&\tilde{L}_U[\varphi]=\\
&&\quad \left(\frac{1}{2\pi}\int_{\mathbb{R}}\frac{|U(x)-U(s)|^2}{|x-s|^2}ds\right)\Pi_{U^\perp}\varphi\\
&&\quad -\frac{1}{\pi}\int_{\mathbb{R}}\frac{\left[(\varphi\cdot U)(x)-(\varphi\cdot U)(s)\right](U(x)-U(s))}{|x-s|^2}ds \\
&&\quad + \left(\frac{1}{\pi}\int_{\mathbb{R}}\frac{(U(x)-U(s))\left[(\varphi\cdot U)(x)-(\varphi\cdot U)(s)\right]}{|x-s|^2}ds\cdot U(x)\right)U(x)\\
\end{eqnarray*}
\begin{eqnarray*}
&&=\mu^{-1}\left(\frac{1}{2\pi}\int_{\mathbb{R}}\frac{|\omega(y)-\omega(s)|^2}{|y-s|^2}ds\right)\Pi_{\omega^\perp}\tilde{\varphi}(y)\\
&&\quad -\frac{\mu^{-1}}{\pi}\int_{\mathbb{R}}\frac{\left[(\tilde{\varphi}\cdot \omega)(y)-(\tilde{\varphi}\cdot \omega)(s)\right](\omega(y)-\omega(s))}{|y-s|^2}ds \\
&&\quad + \left(\frac{\mu^{-1}}{\pi}\int_{\mathbb{R}}\frac{\left[(\tilde{\varphi}\cdot \omega)(y)-(\tilde{\varphi}\cdot \omega)(s)\right](\omega(y)-\omega(s))}{|y-s|^2}ds\cdot \omega(y)\right)\omega(y)\\
&&= \mu^{-1}\left[\left(\frac{1}{2\pi}\int_{\mathbb{R}}\frac{|\omega(y)-\omega(s)|^2}{|y-s|^2}ds\right)\Pi_{\omega^\perp}\tilde{\varphi}(y)\right.\\
&&\quad\quad\quad\quad\quad -\frac{1}{\pi}\int_{\mathbb{R}}\frac{\left[(\tilde{\varphi}\cdot \omega)(y)-(\tilde{\varphi}\cdot \omega)(s)\right](\omega(y)-\omega(s))}{|y-s|^2}ds \\
&&\quad\quad\quad\quad\quad \left. + \left(\frac{1}{\pi}\int_{\mathbb{R}}\frac{\left[(\tilde{\varphi}\cdot \omega)(y)-(\tilde{\varphi}\cdot \omega)(s)\right](\omega(y)-\omega(s))}{|y-s|^2}ds\cdot \omega(y)\right)\omega(y)\right].
\end{eqnarray*}
Here $\varphi(x) = \begin{pmatrix}
     a(\mu y + \xi) \\
     b(\mu y + \xi)
\end{pmatrix} := \begin{pmatrix}
     \tilde{a}(y) \\
     \tilde{b}(y)
\end{pmatrix} = \tilde{\varphi}(y)$ and $y = \frac{x-\xi}{\mu}$.

By direct computation, we have
\begin{eqnarray*}
&&\left(\frac{1}{2\pi}\int_{\mathbb{R}}\frac{|\omega(y)-\omega(s)|^2}{|y-s|^2}ds\right)\Pi_{\omega^\perp}\tilde{\varphi}(y) \\ &&\quad=\frac{1}{\pi}\int_{\mathbb{R}}\frac{2}{\left(s^2+1\right) \left(y^2+1\right)}ds\begin{pmatrix}\frac{\left(y^2-1\right) \left(\left(y^2-1\right) \tilde{a}(y)-2 y \tilde{b}(y)\right)}{\left(y^2+1\right)^2}\\ \frac{2 y \left(y^2 (-\tilde{a}(y))+\tilde{a}(y)+2 y
   \tilde{b}(y)\right)}{\left(y^2+1\right)^2}\end{pmatrix}\\
&&\quad= \begin{pmatrix}\frac{1}{\pi}\int_{\mathbb{R}}\frac{2 \left(y^2-1\right) \left(\left(y^2-1\right) \tilde{a}(y)-2 y \tilde{b}(y)\right)}{\left(s^2+1\right) \left(y^2+1\right)^3}ds\\ \frac{1}{\pi}\int_{\mathbb{R}}\frac{4 y \left(y^2 (-\tilde{a}(y))+\tilde{a}(y)+2 y
   \tilde{b}(y)\right)}{\left(s^2+1\right) \left(y^2+1\right)^3}ds\end{pmatrix},
\end{eqnarray*}
\begin{eqnarray*}
&&\frac{1}{\pi}\int_{\mathbb{R}}\frac{\left[(\tilde{\varphi}\cdot \omega)(y)-(\tilde{\varphi}\cdot \omega)(s)\right](\omega(y)-\omega(s))}{|y-s|^2}ds\\
&&\quad\quad =  \begin{pmatrix}\frac{1}{\pi}\int_{\mathbb{R}}-\frac{2 (s y-1) \left(-2 \left(s^2+1\right) y \tilde{a}(y)+2 s \left(y^2+1\right) \tilde{a}(s)+\left(s^2-1\right) \left(y^2+1\right) \tilde{b}(s)-\left(s^2+1\right)
   \left(y^2-1\right) \tilde{b}(y)\right)}{\left(s^2+1\right)^2 \left(y^2+1\right)^2 (s-y)}ds\\ \frac{1}{\pi}\int_{\mathbb{R}}\frac{2 (s+y) \left(-2 \left(s^2+1\right) y \tilde{a}(y)+2 s \left(y^2+1\right)
   \tilde{a}(s)+\left(s^2-1\right) \left(y^2+1\right) \tilde{b}(s)-\left(s^2+1\right) \left(y^2-1\right) \tilde{b}(y)\right)}{\left(s^2+1\right)^2 \left(y^2+1\right)^2 (s-y)}ds\end{pmatrix}
\end{eqnarray*}
and
\begin{eqnarray*}
&&\left(\frac{1}{\pi}\int_{\mathbb{R}}\frac{\left[(\tilde{\varphi}\cdot \omega)(y)-(\tilde{\varphi}\cdot \omega)(s)\right](\omega(y)-\omega(s))}{|y-s|^2}ds\cdot \omega(y)\right)\omega(y)=\\
&&\quad\quad\quad\begin{pmatrix}\frac{1}{\pi}\int_{\mathbb{R}}\frac{8 \left(s^2+1\right) y^2 \tilde{a}(y)-8 s \left(y^3+y\right) \tilde{a}(s)-4 \left(s^2-1\right) \left(y^3+y\right) \tilde{b}(s)+4 \left(s^2+1\right) \left(y^2-1\right) y
   \tilde{b}(y)}{\left(s^2+1\right)^2 \left(y^2+1\right)^3}ds\\ \frac{1}{\pi}\int_{\mathbb{R}}\frac{2 \left(y^2-1\right) \left(2 \left(s^2+1\right) y \tilde{a}(y)-2 s \left(y^2+1\right) \tilde{a}(s)-\left(s^2-1\right)
   \left(y^2+1\right) \tilde{b}(s)+\left(s^2+1\right) \left(y^2-1\right) \tilde{b}(y)\right)}{\left(s^2+1\right)^2 \left(y^2+1\right)^3}ds\end{pmatrix},
\end{eqnarray*}
therefore
\begin{eqnarray*}
&& \left(\frac{1}{2\pi}\int_{\mathbb{R}}\frac{|\omega(y)-\omega(s)|^2}{|y-s|^2}ds\right)\Pi_{\omega^\perp}\tilde{\varphi}(y) - \frac{1}{\pi}\int_{\mathbb{R}}\frac{\left[(\tilde{\varphi}\cdot \omega)(y)-(\tilde{\varphi}\cdot \omega)(s)\right](\omega(y)-\omega(s))}{|y-s|^2}ds\\ &&\quad\quad + \left(\frac{1}{\pi}\int_{\mathbb{R}}\frac{\left[(\tilde{\varphi}\cdot \omega)(y)-(\tilde{\varphi}\cdot \omega)(s)\right](\omega(y)-\omega(s))}{|y-s|^2}ds\cdot \omega(y)\right)\omega(y)\\
&& = \begin{pmatrix}\frac{1}{\pi}\int_{\mathbb{R}}\frac{2 \left(y^2-1\right) \left(-\left(s^2+1\right) \tilde{a}(y) (s+y)+2 s \tilde{a}(s) (s y+1)+\left(s^2-1\right) \tilde{b}(s) (s y+1)-\left(s^2+1\right) \tilde{b}(y) (s
   y-1)\right)}{\left(s^2+1\right)^2 \left(y^2+1\right)^2 (s-y)}ds\\ \frac{1}{\pi}\int_{\mathbb{R}}\frac{4 y \left(\left(s^2+1\right) \tilde{a}(y) (s+y)-2 s \tilde{a}(s) (s y+1)-\left(s^2-1\right) \tilde{b}(s) (s
   y+1)+\left(s^2+1\right) \tilde{b}(y) (s y-1)\right)}{\left(s^2+1\right)^2 \left(y^2+1\right)^2 (s-y)}ds\end{pmatrix}\\
&& =
\frac{1}{\pi}\int_{\mathbb{R}}\frac{-\left(s^2+1\right)(s+y)\tilde{a}(y)+2 s(s y+1)\tilde{a}(s) +\left(s^2-1\right)(s y+1) \tilde{b}(s)-\left(s^2+1\right)(s
   y-1) \tilde{b}(y) }{\left(s^2+1\right)^2(s-y)}ds Z_2(y)\\
&& =
\frac{1}{\pi}\int_{\mathbb{R}}\frac{-\left(s^2+1\right)(s+y)\tilde{a}(y)+2 s(s y+1)\tilde{a}(s)}{\left(s^2+1\right)^2(s-y)}ds Z_2(y)\\
&&\quad +
\frac{1}{\pi}\int_{\mathbb{R}}\frac{\left(s^2-1\right)(s y+1) \tilde{b}(s)-\left(s^2+1\right)(s
   y-1) \tilde{b}(y) }{\left(s^2+1\right)^2(s-y)}ds Z_2(y)\\
&&=
\frac{1}{\pi}\int_{\mathbb{R}}\frac{2s\left[\tilde{a}(s)-\tilde{a}(y)\right]+ 2s^2\left[\tilde{a}(s)-\tilde{a}(y)\right]y - s^2(s-y)\tilde{a}(y) + (s-y)\tilde{a}(y)}{\left(s^2+1\right)^2(s-y)}ds Z_2(y)\\
&&\quad +
\frac{1}{\pi}\int_{\mathbb{R}}\frac{s^3y\left[\tilde{b}(s)-\tilde{b}(y)\right]- \left[\tilde{b}(s)-\tilde{b}(y)\right] + s^2\left[\tilde{b}(s)-\tilde{b}(y)\right] + 2s(s-y)\tilde{b}(y) -sy\left[\tilde{b}(s)-\tilde{b}(y)\right]}{\left(s^2+1\right)^2(s-y)}ds Z_2(y)\\
&&=
\frac{1}{\pi}\int_{\mathbb{R}}\frac{2s\left[\tilde{a}(s)-\tilde{a}(y)\right]}{\left(s^2+1\right)^2(s-y)}ds Z_2(y) +
\frac{1}{\pi}\int_{\mathbb{R}}\frac{2s^2\left[\tilde{a}(s)-\tilde{a}(y)\right]}{\left(s^2+1\right)^2(s-y)}ds yZ_2(y)\\
&&\quad +\frac{1}{\pi}\int_{\mathbb{R}}\frac{- s^2}{\left(s^2+1\right)^2}ds \tilde{a}(y)Z_2(y)+
\frac{1}{\pi}\int_{\mathbb{R}}\frac{1}{\left(s^2+1\right)^2}ds \tilde{a}(y)Z_2(y)\\
&&\quad +\frac{1}{\pi}\int_{\mathbb{R}}\frac{s^3\left[\tilde{b}(s)-\tilde{b}(y)\right]}{\left(s^2+1\right)^2(s-y)}ds yZ_2(y) +
\frac{1}{\pi}\int_{\mathbb{R}}\frac{- \left[\tilde{b}(s)-\tilde{b}(y)\right]}{\left(s^2+1\right)^2(s-y)}ds Z_2(y)\\
&&\quad+
\frac{1}{\pi}\int_{\mathbb{R}}\frac{s^2\left[\tilde{b}(s)-\tilde{b}(y)\right] }{\left(s^2+1\right)^2(s-y)}ds Z_2(y)\\
&&\quad+
\frac{1}{\pi}\int_{\mathbb{R}}\frac{2s}{\left(s^2+1\right)^2}ds \tilde{b}(y)Z_2(y)\\
&&\quad +\frac{1}{\pi}\int_{\mathbb{R}}\frac{-s\left[\tilde{b}(s)-\tilde{b}(y)\right]}{\left(s^2+1\right)^2(s-y)}ds yZ_2(y)\\
&&=
\frac{1}{\pi}\int_{\mathbb{R}}\frac{2s\left[\tilde{a}(s)-\tilde{a}(y)\right]}{\left(s^2+1\right)^2(s-y)}ds Z_2(y)+
\frac{1}{\pi}\int_{\mathbb{R}}\frac{(s^2-1)\left[\tilde{b}(s)-\tilde{b}(y)\right] }{\left(s^2+1\right)^2(s-y)}ds Z_2(y)\\
&&\quad +
\frac{1}{\pi}\int_{\mathbb{R}}\frac{2s^2\left[\tilde{a}(s)-\tilde{a}(y)\right]}{\left(s^2+1\right)^2(s-y)}ds Z_3(y)
+
\frac{1}{\pi}\int_{\mathbb{R}}\frac{s(s^2-1)\left[\tilde{b}(s)-\tilde{b}(y)\right]}{\left(s^2+1\right)^2(s-y)}ds Z_3(y).
\end{eqnarray*}
Hence we have
\begin{eqnarray*}
&&\tilde{L}_U[\varphi](x,t)=\\
&&\quad\quad\quad
\mu^{-1}\left[\frac{1}{\pi}\int_{\mathbb{R}}\frac{\tilde{a}(s)-\tilde{a}(y)}{s-y}\frac{2s}{\left(s^2+1\right)^2}ds Z_2(y)+
\frac{1}{\pi}\int_{\mathbb{R}}\frac{\tilde{b}(s)-\tilde{b}(y) }{s-y}\frac{s^2-1}{\left(s^2+1\right)^2}ds Z_2(y)\right.\\
&&\quad\quad\quad\quad \left. +
\frac{1}{\pi}\int_{\mathbb{R}}\frac{\tilde{a}(s)-\tilde{a}(y)}{s-y}\frac{2s^2}{\left(s^2+1\right)^2}ds Z_3(y)
+
\frac{1}{\pi}\int_{\mathbb{R}}\frac{\tilde{b}(s)-\tilde{b}(y)}{s-y}\frac{s(s^2-1)}{\left(s^2+1\right)^2}ds Z_3(y)\right]\\
&&\quad\quad\quad=
\mu^{-1}\left[\frac{1}{\pi}\int_{\mathbb{R}}\frac{a(s)-a(x)}{s-x}\frac{2\left(\frac{s-\xi}{\mu}\right)}
{\left[\left(\frac{s-\xi}{\mu}\right)^2+1\right]^2}ds Z_2\left(\frac{x-\xi}{\mu}\right)\right.\\ &&\quad\quad\quad\quad\quad\quad\quad +
\frac{1}{\pi}\int_{\mathbb{R}}\frac{b(s)-b(x) }{s-x}\frac{\left(\frac{s-\xi}{\mu}\right)^2-1}{\left[\left(\frac{s-\xi}{\mu}\right)^2+1\right]^2}ds Z_2\left(\frac{x-\xi}{\mu}\right)\\
&&\quad\quad\quad\quad\quad\quad\quad +
\frac{1}{\pi}\int_{\mathbb{R}}\frac{a(s)-a(x)}{s-x}\frac{2\left(\frac{s-\xi}{\mu}\right)^2}{\left[\left(\frac{s-\xi}{\mu}\right)^2+1\right]^2}ds Z_3\left(\frac{x-\xi}{\mu}\right)\\
&&\quad\quad\quad\quad\quad\quad\quad  \left.
+
\frac{1}{\pi}\int_{\mathbb{R}}\frac{b(s)-b(x)}{s-x}\frac{\left(\frac{s-\xi}{\mu}\right)\left(\left(\frac{s-\xi}{\mu}\right)^2-1\right)}
{\left[\left(\frac{s-\xi}{\mu}\right)^2+1\right]^2}ds Z_3\left(\frac{x-\xi}{\mu}\right)\right].
\end{eqnarray*}
This completes the proof.
\end{proof}

\subsection{Defining $\Phi^*$.}
Let $\varphi^0$ be a solution of the following equation
$$-\varphi^0_t-(-\Delta)^{\frac{1}{2}} \varphi^0  - \begin{pmatrix}
\frac{2(x-\xi)}{r^2+\mu^2}\\
0
\end{pmatrix}\dot{\mu}= 0$$
and
$$
p(t) := -2\dot{\mu},
$$
$$
z(r) := \sqrt{(x-\xi)^2 + \mu^2},
$$
$$
r:=|x-\xi|.
$$
Duhamel's formula gives us the following expression for a weak solution to the problem
$$-\psi^0_t-(-\Delta)^{\frac{1}{2}} \psi^0  -
\frac{2(x-\xi)}{r^2+\mu^2}\dot{\mu}= 0$$
as
\begin{equation*}
\psi^0(x,t) = -\int_{t_0}^{t}\int_{\mathbb{R}}\frac{1}{t-\tilde{s}}\frac{1}{1+\left(\frac{x-y}{t-\tilde{s}}\right)^2}\frac{2(y-\xi)}
{(y-\xi)^2+\mu^2}\dot{\mu}(\tilde{s})dyd\tilde{s}.
\end{equation*}
By direct computation, we have
$$
\psi^0 = \int_{t_0}^tp(\tilde{s})k(z(r), t-\tilde{s})d\tilde{s},\quad k(z(r), t-\tilde{s}) = \frac{x-\xi}{(x-\xi)^2+(\mu+t-\tilde{s})^2}.
$$
Then we define
$$
\Phi^0[\mu,\xi](x, t) = \begin{pmatrix}
\psi^0\\
0
\end{pmatrix}.
$$

Suppose $Z^*$ is a function independent of the parameter functions $\mu$, $\xi$ but just as a solution of the homogeneous half heat equation.
We consider a small, smooth function $Z^*_0:\mathbb{R}\to \mathbb{R}^2$ and the solution $Z^*(x, t)$ of the half heat equation
\begin{equation*}
\left\{\begin{array}{ll}
Z^*_t = -(-\Delta)^{\frac{1}{2}}Z^*\quad\text{ in }\mathbb{R}\times (t_0, \infty),\\
\lim_{t\to+\infty}Z^*(\cdot, t) = Z^*_0\quad\text{ in }\mathbb{R}.
\end{array}
\right.
\end{equation*}
Let us write
\begin{equation*}
Z^*(x, t)= \begin{pmatrix}
z^*_1(x, t) \\
z^*_2(x, t)
\end{pmatrix}.
\end{equation*}
The specific assumptions on $Z^*_0 = \begin{pmatrix}
z^*_{10}(x) \\
z^*_{20}(x)
\end{pmatrix}$ will be given in Section 2.5.

Define
$$
\Phi^*:= \Phi^0[\mu,\xi](x, t) + Z^*(x, t).
$$
We will compute the linear error
$$
-U_t -\partial_t\Pi_{U^\perp}\varphi^* + L_U(\Pi_{U^\perp}\varphi^*), \quad \varphi^* = \Phi^*
$$
for equation (\ref{e:balanceequation}) induced by this correction.
Define
\begin{equation}\label{e:linearerror}
\mathcal{E}^* = \mathcal{E}_{0}+\mathcal{E}_1 + (-\partial_t + L_U)[\Pi_{U^\perp}\Phi^*].
\end{equation}
\subsection{The linear error}
Based on Lemma \ref{l:lemma2.1} and Lemma \ref{l:lemma2.2}, the error $\mathcal{E}^*$ defined in (\ref{e:linearerror}) can be computed as
\begin{eqnarray*}
&&\mathcal{E}^* = \mathcal{E}_{0}+\mathcal{E}_1 + (-\partial_t + L_U)[\Pi_{U^\perp}\Phi^0] + (-\partial_t + L_U)[\Pi_{U^\perp}Z^*]\\
&&\quad = -\frac{\dot{\mu}}{\mu}Z_3(y) + \Pi_{U^\perp}[-\Phi^0_t-(-\Delta)^{\frac{1}{2}}\Phi^0] \\
&&\quad\quad +\int_{t_0}^t\frac{p(\tilde{s})}{\mu^2(\tilde{s})}\frac{\left(\frac{t-\tilde{s}}{\mu}+1\right)^2}{\left(\frac{t-\tilde{s}}{\mu}+2\right)^2 \left(\left(\frac{t-\tilde{s}}{\mu}+1\right)^2+y^2\right)}d\tilde{s}Z_3(y) \\
&&\quad\quad+\frac{1}{\mu}\left[
\frac{1}{\pi}\int_{\mathbb{R}}\frac{z_1^*(s,t)-z_1^*(x,t)}{s-x}\frac{2\left(\frac{s-\xi}{\mu}\right)^2}
{\left[\left(\frac{s-\xi}{\mu}\right)^2+1\right]^2}ds\right. \\
\end{eqnarray*}
\begin{eqnarray*}
&&\quad\quad\left. \quad\quad\quad +
\frac{1}{\pi}\int_{\mathbb{R}}\frac{z_2^*(s,t)-z_2^*(x,t)}{s-x}\frac{\left(\frac{s-\xi}{\mu}\right)\left(\left(\frac{s-\xi}{\mu}\right)^2-1\right)}
{\left[\left(\frac{s-\xi}{\mu}\right)^2+1\right]^2}ds
\right]Z_3(y)\\
&&\quad\quad - \frac{\dot{\xi}}{\mu}Z_2(y) -\int_{t_0}^t\frac{p(\tilde{s})}{\mu^2(\tilde{s})}\frac{y}{\left(\frac{t-\tilde{s}}{\mu}+2\right)^2 \left(\left(\frac{t-\tilde{s}}{\mu}+1\right)^2+y^2\right)}d\tilde{s}Z_2(y)\\
&&\quad\quad +\frac{1}{\mu}\left[\frac{1}{\pi}\int_{\mathbb{R}}\frac{z_1^*(s,t)-z_1^*(x,t)}{s-x}\frac{2\left(\frac{s-\xi}{\mu}\right)}
{\left[\left(\frac{s-\xi}{\mu}\right)^2+1\right]^2}ds\right.\\
&&\quad\quad\quad\quad\quad\left. + \frac{1}{\pi}\int_{\mathbb{R}}\frac{z_2^*(s,t)-z_2^*(x,t) }{s-x}\frac{\left(\frac{s-\xi}{\mu}\right)^2-1}{\left[\left(\frac{s-\xi}{\mu}\right)^2+1\right]^2}ds\right]Z_2(y)\\
&&\quad\quad + (\Phi^0\cdot U)U_t + (\Phi^0\cdot U_t)U + (Z\cdot U)U_t + (Z\cdot U_t)U\\
&&\quad := \mathcal{E}_{0} + \mathcal{E}_{1} + \mathcal{E}_{2},
\end{eqnarray*}
where
\begin{eqnarray*}
&&\mathcal{E}_{0} = -\frac{\dot{\mu}}{\mu}Z_3(y) + \Pi_{U^\perp}[-\Phi^0_t-(-\Delta)^{\frac{1}{2}}\Phi^0] \\
&&\quad\quad +\int_{t_0}^t\frac{p(\tilde{s})}{\mu^2(\tilde{s})}\frac{\left(\frac{t-\tilde{s}}{\mu}+1\right)^2}{\left(\frac{t-\tilde{s}}{\mu}+2\right)^2 \left(\left(\frac{t-\tilde{s}}{\mu}+1\right)^2+y^2\right)}d\tilde{s}Z_3(y) \\
&&\quad\quad+\frac{1}{\mu}\left[
\frac{1}{\pi}\int_{\mathbb{R}}\frac{z_1^*(s,t)-z_1^*(x,t)}{s-x}\frac{2\left(\frac{s-\xi}{\mu}\right)^2}
{\left[\left(\frac{s-\xi}{\mu}\right)^2+1\right]^2}ds\right. \\
&&\quad\quad\left. \quad\quad\quad +
\frac{1}{\pi}\int_{\mathbb{R}}\frac{z_2^*(s,t)-z_2^*(x,t)}{s-x}\frac{\left(\frac{s-\xi}{\mu}\right)\left(\left(\frac{s-\xi}{\mu}\right)^2-1\right)}
{\left[\left(\frac{s-\xi}{\mu}\right)^2+1\right]^2}ds
\right]Z_3(y)\\
&&\quad = \frac{\dot{\mu}}{\mu}\begin{pmatrix}
\frac{-4y(y^2-1)}{(1+y^2)^3} \\
\frac{8y^2}{(1+y^2)^3}
\end{pmatrix} \\
&&\quad\quad +\int_{t_0}^t\frac{p(\tilde{s})}{\mu^2(\tilde{s})}\frac{\left(\frac{t-\tilde{s}}{\mu}+1\right)^2}{\left(\frac{t-\tilde{s}}{\mu}+2\right)^2 \left(\left(\frac{t-\tilde{s}}{\mu}+1\right)^2+y^2\right)}d\tilde{s}Z_3(y) \\
\end{eqnarray*}
\begin{eqnarray*}
&&\quad\quad+\frac{1}{\mu}\left[
\frac{1}{\pi}\int_{\mathbb{R}}\frac{z_1^*(s,t)-z_1^*(x,t)}{s-x}\frac{2\left(\frac{s-\xi}{\mu}\right)^2}
{\left[\left(\frac{s-\xi}{\mu}\right)^2+1\right]^2}ds\right. \\
&&\quad\quad\left. \quad\quad\quad +
\frac{1}{\pi}\int_{\mathbb{R}}\frac{z_2^*(s,t)-z_2^*(x,t)}{s-x}\frac{\left(\frac{s-\xi}{\mu}\right)\left(\left(\frac{s-\xi}{\mu}\right)^2-1\right)}
{\left[\left(\frac{s-\xi}{\mu}\right)^2+1\right]^2}ds
\right]Z_3(y),
\end{eqnarray*}
\begin{eqnarray*}
&&\mathcal{E}_{1}= - \frac{\dot{\xi}}{\mu}Z_2(y) -\int_{t_0}^t\frac{p(\tilde{s})}{\mu^2(\tilde{s})}\frac{y}{\left(\frac{t-\tilde{s}}{\mu}+2\right)^2 \left(\left(\frac{t-\tilde{s}}{\mu}+1\right)^2+y^2\right)}d\tilde{s}Z_2(y)\\
&&\quad\quad +\frac{1}{\mu}\left[\frac{1}{\pi}\int_{\mathbb{R}}\frac{z_1^*(s,t)-z_1^*(x,t)}{s-x}\frac{2\left(\frac{s-\xi}{\mu}\right)}
{\left[\left(\frac{s-\xi}{\mu}\right)^2+1\right]^2}ds\right.\\
&&\quad\quad\quad\quad\quad\left. + \frac{1}{\pi}\int_{\mathbb{R}}\frac{z_2^*(s,t)-z_2^*(x,t) }{s-x}\frac{\left(\frac{s-\xi}{\mu}\right)^2-1}{\left[\left(\frac{s-\xi}{\mu}\right)^2+1\right]^2}ds\right]Z_2(y)
\end{eqnarray*}
and
\begin{eqnarray*}
\mathcal{E}_{2} = (\Phi^0\cdot U)U_t + (\Phi^0\cdot U_t)U + (Z\cdot U)U_t + (Z\cdot U_t)U.
\end{eqnarray*}
Here $y = \frac{x-\xi(t)}{\mu(t)}$.
\subsection{The choice at main order of $\mu$ and $\xi$: improving the inner error.}
In this subsection, we come back to the original linearized problem which can be written as
\begin{equation*}
\begin{aligned}
\mathcal{\varphi} = -\partial_t\varphi + L_U(\varphi) + \mathcal{E}^* + b(x,t)U = 0,\quad\varphi\cdot U = 0
\end{aligned}
\end{equation*}
and discuss how to adjust the parameter functions $\mu$, $\xi$ such that a correction $\varphi$ can be found so that $\mathcal{E}(\varphi)$ is globally smaller than the first approximation error $-U_t$.

If we use the self-similar transformation
$$
\Phi(x,t) = \phi(y, t),\quad y = \frac{x-\xi}{\mu},\quad \rho = |y|,
$$
then
\begin{equation*}
\begin{aligned}
0= -\mu\partial_t\Phi + L_\omega(\phi) + \mu\mathcal{E}^* + b(x,t)\omega,\quad \phi\cdot \omega = 0.
\end{aligned}
\end{equation*}
An improvement of the approximation can be obtained if we solve the following time independent equation
\begin{equation}\label{e:ellipticapproximation}
\begin{aligned}
0 = L_\omega(\phi) + \mu\mathcal{E}^*, \quad\phi\cdot \omega = 0,\quad \lim_{|y|\to \infty}\phi(y, t) = 0\quad\text{in}\quad\mathbb{R}.
\end{aligned}
\end{equation}
The decay condition is added in order to not essentially modify the size of the error far away. From the nondegeneracy result of \cite{sire2017nondegeneracy}, a necessary condition for the solvability of (\ref{e:ellipticapproximation}) is that $\mu\mathcal{E}^*$ is orthogonal to $Z_2(y)$ and $Z_3(y)$ (defined in the introduction section) in the $L^2$ sense. Now testing equation (\ref{e:ellipticapproximation}) with $Z_3(y)$ and integrating by parts, due to the involved decays, we obtain
\begin{equation}\label{equationformu}
\mu\int_{\mathbb{R}}\mathcal{E}^*\cdot Z_3dy = 0.
\end{equation}
From the computations in Section 2.4, we have
\begin{eqnarray*}
&&0 = \mu\int_{\mathbb{R}}\mathcal{E}^*\cdot Z_3dy = -\pi\dot{\mu} + 2\pi\left[
\frac{1}{\pi}\int_{\mathbb{R}}\frac{z_1^*(s,t)-z_1^*(x,t)}{s-x}\frac{2\left(\frac{s-\xi}{\mu}\right)^2}
{\left[\left(\frac{s-\xi}{\mu}\right)^2+1\right]^2}ds\right. \\
&&\quad\quad\left. \quad +
\frac{1}{\pi}\int_{\mathbb{R}}\frac{z_2^*(s,t)-z_2^*(x,t)}{s-x}\frac{\left(\frac{s-\xi}{\mu}\right)\left(\left(\frac{s-\xi}{\mu}\right)^2-1\right)}
{\left[\left(\frac{s-\xi}{\mu}\right)^2+1\right]^2}ds
\right]\Bigg|_{x=\xi(t)} + 2\pi\mu\int_{t_0}^t\frac{p(\tilde{s})}{\mu^2(\tilde{s})}\frac{\left(\frac{t-\tilde{s}}{\mu}+1\right)^2}
{\left(\frac{t-\tilde{s}}{\mu}+2\right)^4}d\tilde{s}\\
&&\approx -\pi\dot{\mu} + 2\pi\mu\left[
\frac{1}{\pi}\int_{\mathbb{R}}\frac{z_2^*(s+\xi(t),t)-z_2^*(x+\xi(t),t)}{s-x}\frac{s^3}
{\left(s^2+1\right)^2}ds
\right]\Bigg|_{x=0}\\
&&\quad + 2\pi\mu\int_{t_0}^t\frac{p(\tilde{s})}{\mu^2(\tilde{s})}\frac{\left(\frac{t-\tilde{s}}{\mu}+1\right)^2}
{\left(\frac{t-\tilde{s}}{\mu}+2\right)^4}d\tilde{s}.
\end{eqnarray*}
We can achieve this by the simple ansatz $\mu(t) = e^{-\kappa t}$ where $\kappa$ is a positive constant to be determined. Under this assumption, we have $\int_{t_0}^t\frac{p(\tilde{s})}{\mu^2(\tilde{s})}\frac{\left(\frac{t-\tilde{s}}{\mu}+1\right)^2}
{\left(\frac{t-\tilde{s}}{\mu}+2\right)^4}d\tilde{s}\approx \frac{7}{12}\kappa$.

Indeed, for some small positive number $\delta$ to be chosen, we decompose the integral $\int_{t_0}^t\frac{\dot{\mu}(\tilde{s})}{\mu^2(\tilde{s})}\frac{\left(\frac{t-\tilde{s}}{\mu}+1\right)^2}
{\left(\frac{t-\tilde{s}}{\mu}+2\right)^4}d\tilde{s}$ as
\begin{eqnarray*}
\int_{t_0}^t\frac{\dot{\mu}(\tilde{s})}{\mu^2(\tilde{s})}\frac{\left(\frac{t-\tilde{s}}{\mu}+1\right)^2}
{\left(\frac{t-\tilde{s}}{\mu}+2\right)^4}d\tilde{s}
& = & \int_{t_0}^{t-\delta}\frac{\dot{\mu}(\tilde{s})}{\mu^2(\tilde{s})}\frac{\left(\frac{t-\tilde{s}}{\mu}+1\right)^2}
{\left(\frac{t-\tilde{s}}{\mu}+2\right)^4}d\tilde{s} + \int_{t-\delta}^t\frac{\dot{\mu}(\tilde{s})}{\mu^2(\tilde{s})}\frac{\left(\frac{t-\tilde{s}}{\mu}+1\right)^2}
{\left(\frac{t-\tilde{s}}{\mu}+2\right)^4}d\tilde{s}:= I_1 + I_2.
\end{eqnarray*}
For the first integral we have $t-s >\delta$ and hence
$$
0\leq -I_1 \leq \frac{-1}{\delta^2} \int_{t_0}^{t-\delta}\dot{\mu}(\tilde{s})d\tilde{s}= \frac{1}{\delta^2}\left(e^{-\kappa t_0}-e^{-\kappa (t-\delta)}\right)
\leq \frac{1}{\delta^2} e^{-\kappa t_0}.
$$
For $I_2 = \int_{t-\delta}^t\frac{\dot{\mu}(\tilde{s})}{\mu^2(\tilde{s})}\frac{\left(\frac{t-\tilde{s}}{\mu}+1\right)^2}
{\left(\frac{t-\tilde{s}}{\mu}+2\right)^4}d\tilde{s}$, we use change of variables as
$\tilde{s}= t- \mu (\tilde{s}) \hat{s}$, then
$$
d\tilde{s} = -\frac{\mu(\tilde{s})}{1+\dot{\mu}(\tilde{s})\hat{s}}d\hat{s}
$$
and the integral becomes
\begin{eqnarray*}
&&I_2=\int_{t-\delta}^t\frac{\dot{\mu}(\tilde{s})}{\mu^2(\tilde{s})}\frac{\left(\frac{t-\tilde{s}}{\mu}+1\right)^2}
{\left(\frac{t-\tilde{s}}{\mu}+2\right)^4}d\tilde{s} = \int^{\frac{\delta}{\mu(t-\delta)}}_{0} \frac{\dot{\mu}(\tilde{s})}{\mu (\tilde{s})} \frac{(1+\hat{s})^2}{ (2+\hat{s})^4}\frac{1}{ 1+\dot{\mu}(\tilde{s}) \hat{s}} d \hat{s}.
\end{eqnarray*}
Note that $1+\dot{\mu}(\tilde{s}) \hat{s} =1-\kappa\mu(\tilde{s}) \hat{s}= 1-\kappa(t-\tilde{s})
> 1-\kappa\delta$, therefore $ d\tilde{s} = (1+ O(\delta)) d \hat{s}$ and
\begin{equation*}
I_2= -\kappa \left(\int_0^{ \frac{\delta}{\mu(t-\delta)}} \frac{(1+\hat{s})^2}{ (2+\hat{s})^4}d \hat{s}+o(1)\right)
= -\frac{7}{24}\kappa + o(1)
\end{equation*}
as long as $\frac{\delta}{\mu (t-\delta)}$ is large.
Therefore, we have
\begin{eqnarray*}
\int_{t_0}^t\frac{\dot{\mu}(\tilde{s})}{\mu^2(\tilde{s})}\frac{\left(\frac{t-\tilde{s}}{\mu}+1\right)^2}{\left(\frac{t-\tilde{s}}
{\mu}+2\right)^4}d\tilde{s} = -\frac{7}{24}\kappa + o(1)
\end{eqnarray*}
when $t_0$ is chosen sufficiently large and $\delta = e^{-\frac{\kappa}{3}t_0}$. Thus we have
\begin{eqnarray*}
\int_{t_0}^t\frac{p(\tilde{s})}{\mu^2(\tilde{s})}\frac{\left(\frac{t-\tilde{s}}{\mu}+1\right)^2}{\left(\frac{t-\tilde{s}}{\mu}+2\right)^4}d\tilde{s} = \frac{7}{12}\kappa + o(1).
\end{eqnarray*}

Hence equation (\ref{equationformu}) becomes
$$
\left(\kappa + 2\left[
\frac{1}{\pi}\int_{\mathbb{R}}\frac{z_{2}^*(s+\xi(t),t)-z_{2}^*(x+\xi(t),t)}{s-x}\frac{s^3}
{\left(s^2+1\right)^2}ds
\right]\Bigg|_{x=0} + \frac{7}{6}\kappa\right)\pi\mu\approx 0.
$$
If we choose the noise $Z^*(x, t)$ such that $\left[
\frac{1}{\pi}\int_{\mathbb{R}}\frac{z_{20}^*(s+q,t)-z_{20}^*(x+q,t)}{s-x}\frac{s^3}
{\left(s^2+1\right)^2}ds
\right]\Bigg|_{x=0} < 0$, the first approximation of the function $\mu(t)$ is given by
\begin{equation}\label{e:blowuprate}
\mu_0(t) = e^{-\kappa_0t}, \quad \kappa_0 = -\frac{12}{13}\left[
\frac{1}{\pi}\int_{\mathbb{R}}\frac{z_{20}^*(s+q,t)-z_{20}^*(x+q,t)}{s-x}\frac{s^3}
{\left(s^2+1\right)^2}ds
\right]\Bigg|_{x=0}.
\end{equation}

Now we justify the choice at main order of the translation parameter function $\xi(t)$. We do this in a similar way as before, now testing the error $\mu\mathcal{E}^*$ against the function $Z_2(y)$, we have
\begin{equation*}
\mu\int_{\mathbb{R}}\mathcal{E}^*\cdot Z_2dy \approx 0.
\end{equation*}
From the computations in Section 2.4, we get
\begin{equation*}
-2\pi\dot{\xi} \thickapprox 0.
\end{equation*}
Hence $\dot{\xi}(t) = 0$ and it is natural to set that $\xi(t_0) = q$. Then the first order approximation of the translation parameter function $\xi(t)$ should be
\begin{equation}\label{e:definitionofxi0}
\xi_0(t) = q.
\end{equation}

\subsection{The final ansatz.}
Let us fix the parameter functions $\mu_0(t)$, $\xi_0(t)$ defined in (\ref{e:blowuprate}) and (\ref{e:definitionofxi0}). Then we write
$$
\mu(t) = \mu_0(t) + \lambda(t),\quad \xi(t) = \xi_0(t) +\xi_1(t) = q + \xi_1(t).
$$
We are looking for a small solution $\varphi$ of
\begin{equation}\label{e:modifiedequation}
\mathcal{E}^* -\partial_t\Pi_{U^\perp}\varphi + L_U(\Pi_{U^\perp}\varphi)+N_U(\Pi_{U^\perp}[\Phi^0 + Z^*+\varphi])+\tilde{b}(x,t)U = 0.
\end{equation}
In other words, let $t_0 > 0$, we want the function
$$
u = U+\Pi_{U^\perp}[\Phi^0 + Z^*+\varphi] + a(\Pi_{U^\perp}[\Phi^0 + Z^*+\varphi])U
$$
solves the problem
\begin{equation*}
\left\{\begin{array}{ll}
        u_t = -(-\Delta)^{\frac{1}{2}}u + \left(\frac{1}{2\pi}\int_{\mathbb{R}}\frac{|u(x)-u(s)|^2}{|x-s|^2}ds\right)u\quad\text{ in }\mathbb{R}\times [t_0, \infty),\\
        u(\cdot, t_0) = u_0\quad\text{ in }\mathbb{R}
       \end{array}
\right.
\end{equation*}
when $t_0$ is sufficiently large. This provides a solution $u(x, t) = u(x, t-t_0)$ to the main problem (\ref{e:main}).

\section{The outer-inner gluing procedure}
Using Lemma \ref{l:lemma2.1} and possibly modifying $\tilde{b}(x,t)$, equation (\ref{e:modifiedequation}) can be rewritten as
\begin{equation}\label{e:equation3.1}
\begin{aligned}
0 &= \mathcal{E}^* -\partial_t\Pi_{U^\perp}\varphi + L_U(\Pi_{U^\perp}\varphi)+N_U(\Pi_{U^\perp}[\Phi^0 + Z^*+\varphi])+\tilde{b}(x,t)U\\
& = \mathcal{E}^* -\partial_t\varphi -(-\Delta)^{\frac{1}{2}}\varphi + \left(\frac{1}{2\pi}\int_{\mathbb{R}}\frac{|U(x)-U(s)|^2}{|x-s|^2}ds\right)\varphi\\
&\quad -\frac{1}{\pi}\int_{\mathbb{R}}\frac{\left[(\varphi\cdot U)(x)-(\varphi\cdot U)(s)\right](U(x)-U(s))}{|x-s|^2}ds \\
&\quad + \left(\frac{1}{\pi}\int_{\mathbb{R}}\frac{(U(x)-U(s))\left[(\varphi\cdot U)(x)-(\varphi\cdot U)(s)\right]}{|x-s|^2}ds\cdot U(x)\right)U(x)\\
&\quad -(\varphi\cdot U)U_t + N_U(\Pi_{U^\perp}[\Phi^0 + Z^*+\varphi])+\tilde{b}(x,t)U.
\end{aligned}
\end{equation}

Let $\eta_0(s)$ be a smooth cut-off function with $\eta_0(s) = 1$ for $s < 1$ and $=0$ for $s > 2$. Take a sufficiently large constant
\begin{equation*}
R = e^{\rho t_0}
\end{equation*}
for $\rho\in (0, 1)$ sufficiently small and $t_0$ is the initial time.
We define
$$
\eta(x, t):=\eta_0\left(\frac{|x-\xi(t)|}{R\mu_0(t)}\right)
$$
and consider a function $\varphi(x, t)$ of the following special form
\begin{equation}\label{e:errorterm}
\varphi(x, t) = \eta \tilde{\phi}(x, t) + \psi(x, t)
\end{equation}
for a function $\tilde{\phi}(x, t) = \phi\left(\frac{x-\xi(t)}{\mu_0(t)}, t\right)$ and $\phi(y, t_0) = 0$. Here and in what follows, we use the notation $y := \frac{x-\xi}{\mu_0}$. Note that we only concern about the values of $\varphi$ in the direction perpendicular to $U$, so we assume that $\tilde{\phi}(x, t)\cdot \omega\left(\frac{x-\xi(t)}{\mu_0(t)}\right) \equiv 0$ for $|y|\leq 2R$ and $\varphi(x, t)\cdot \omega\left(\frac{x-\xi(t)}{\mu(t)}\right) \equiv 0$.
Then it is straightforward to check that $\varphi$ defined by (\ref{e:errorterm}) solves (\ref{e:equation3.1}) if the pair $(\phi, \psi)$ satisfies the following system of evolution equations
\begin{equation}\label{e:innerproblembeforechangeofvariables}
\left\{
\begin{aligned}
&\mu_0(t)\partial_t \phi\\ &= -(-\Delta)^{\frac{1}{2}}\phi + \frac{2}{1+|y|^2}\phi
+\frac{\mu_0}{\pi}\int_{\mathbb{R}}\frac{\left[\phi\left(\frac{x-\xi}{\mu_0}\right)-\phi\left(\frac{s-\xi}{\mu_0}\right)\right]
\left[\omega\left(\frac{x-\xi}{\mu_0}\right)-\omega\left(\frac{s-\xi}{\mu_0}\right)\right]}{|x-s|^2}ds \omega\left(\frac{x-\xi}{\mu_0}\right)\\
& + \mu_0\Pi_{U^\perp}\mathcal{E}^*(\xi+\mu_0y, t) +\frac{2\frac{\mu_0}{\mu}}{1+\left|\frac{\mu_0}{\mu}y\right|^2}\Pi_{U^\perp}\psi \\ & -\frac{\mu_0}{\pi}\int_{\mathbb{R}}\frac{\left[\psi(x)\cdot \omega\left(\frac{x-\xi}{\mu}\right)-\psi(s)\cdot \omega\left(\frac{s-\xi}{\mu}\right)\right]\left[\omega\left(\frac{x-\xi}{\mu}\right)-\omega\left(\frac{s-\xi}{\mu}\right)\right]}{|x-s|^2}ds \\
& + \left(\frac{\mu_0}{\pi}\int_{\mathbb{R}}\frac{\left[\psi(x)\cdot \omega\left(\frac{x-\xi}{\mu}\right)-\psi(s)\cdot \omega\left(\frac{s-\xi}{\mu}\right)\right]\left[\omega\left(\frac{x-\xi}{\mu}\right)-\omega\left(\frac{s-\xi}{\mu}\right)\right]}{|x-s|^2}ds\cdot \omega\left(\frac{x-\xi}{\mu}\right)\right)\omega\left(\frac{x-\xi}{\mu}\right)\\
& + B^1[\phi] + B^2[\phi] + B^3[\phi]\quad \text{in }B_{2R}(0)\times [t_0, \infty),\\
&\phi = 0\quad\text{in }B_{2R}(0)\times \{t_0\},
\end{aligned}
\right.
\end{equation}
\begin{equation*}
\begin{aligned}
B^1[\phi]: = \left(\frac{2\frac{\mu_0}{\mu}}{1+\left|\frac{\mu_0}{\mu}y\right|^2} - \frac{2}{1+|y|^2}\right)\phi,
\end{aligned}
\end{equation*}
\begin{equation*}
\begin{aligned}
B^2[\phi]: = &-\frac{\mu_0}{\pi}\int_{\mathbb{R}}\frac{\left[\phi\left(\frac{x-\xi}{\mu_0}\right)\cdot \omega\left(\frac{x-\xi}{\mu}\right)-\phi\left(\frac{s-\xi}{\mu_0}\right)\cdot \omega\left(\frac{s-\xi}{\mu}\right)\right]\left[\omega\left(\frac{x-\xi}{\mu}\right)-\omega\left(\frac{s-\xi}{\mu}\right)\right]}{|x-s|^2}ds\\
&+\frac{\mu_0}{\pi}\int_{\mathbb{R}}\frac{\left[(\phi\cdot \omega)\left(\frac{x-\xi}{\mu_0}\right)-(\phi\cdot \omega)\left(\frac{s-\xi}{\mu_0}\right)\right]\left[\omega\left(\frac{x-\xi}{\mu_0}\right)-\omega\left(\frac{s-\xi}{\mu_0}\right)\right]}
{|x-s|^2}ds,\\
\end{aligned}
\end{equation*}
\begin{equation*}
\begin{aligned}
&B^3[\phi]: = \\
&\left(\frac{\mu_0}{\pi}\int_{\mathbb{R}}\frac{\left[\phi\left(\frac{x-\xi}{\mu_0}\right)\cdot \omega\left(\frac{x-\xi}{\mu}\right)-\phi\left(\frac{s-\xi}{\mu_0}\right)\cdot \omega\left(\frac{s-\xi}{\mu}\right)\right]\left[\omega\left(\frac{x-\xi}{\mu}\right)-\omega\left(\frac{s-\xi}{\mu}\right)\right]}{|x-s|^2}ds\cdot \omega\left(\frac{x-\xi}{\mu}\right)\right)\omega\left(\frac{x-\xi}{\mu}\right)\\
&-\left(\frac{\mu_0}{\pi}\int_{\mathbb{R}}\frac{\left[(\phi\cdot \omega)\left(\frac{x-\xi}{\mu_0}\right)-(\phi\cdot \omega)\left(\frac{s-\xi}{\mu_0}\right)\right]\left[\omega\left(\frac{x-\xi}{\mu_0}\right)-\omega\left(\frac{s-\xi}{\mu_0}\right)\right]}
{|x-s|^2}ds\cdot \omega\left(\frac{x-\xi}{\mu_0}\right)\right)\omega\left(\frac{x-\xi}{\mu_0}\right)
\end{aligned}
\end{equation*}
and
\begin{equation}\label{e:outerproblem}
\left\{
\begin{aligned}
&\partial_t \psi = (-\Delta)^{\frac{1}{2}}\psi\\
& +(1-\eta)\left[\frac{2\frac{1}{\mu}}{1+\left|\frac{x-\xi}{\mu}\right|^2}\psi -\frac{1}{\pi}\int_{\mathbb{R}}\frac{\left[\psi(x)\cdot \omega\left(\frac{x-\xi}{\mu}\right)-\psi(s)\cdot \omega\left(\frac{s-\xi}{\mu}\right)\right]\left[\omega\left(\frac{x-\xi}{\mu}\right)-\omega\left(\frac{s-\xi}{\mu}\right)\right]}{|x-s|^2}ds\right. \\
& \left.+ \left(\frac{1}{\pi}\int_{\mathbb{R}}\frac{\left[\psi(x)\cdot \omega\left(\frac{x-\xi}{\mu}\right)-\psi(s)\cdot \omega\left(\frac{s-\xi}{\mu}\right)\right]\left[\omega\left(\frac{x-\xi}{\mu}\right)-\omega\left(\frac{s-\xi}{\mu}\right)\right]}{|x-s|^2}ds\cdot \omega\left(\frac{x-\xi}{\mu}\right)\right)\omega\left(\frac{x-\xi}{\mu}\right)\right]\\
&  + \frac{\dot{\mu}_0}{\mu_0}\eta y\cdot \nabla_y\phi + \eta\frac{\dot{\xi}}{\mu_0}\cdot\nabla_y\phi + N_U(\Pi_{U^\perp}[\Phi^0 + Z^*+\eta\phi+\psi]) + (1-\eta)\Pi_{U^\perp}\mathcal{E}^*\\
&  -\left[(-\Delta)^{\frac{1}{2}}\eta\right] \phi - \partial_t\eta\phi\left(\frac{x-\xi(t)}{\mu_0(t)}\right)+\frac{1}{\pi}\int_{\mathbb{R}}\frac{(\eta(x)-\eta(s))\left(\phi\left(\frac{x-\xi(t)}
{\mu_0(t)}, t\right)-\phi\left(\frac{s-\xi(t)}{\mu_0(t)}, t\right)\right)}{|x-s|^2}ds\\
&  -\frac{1}{\pi}\int_{\mathbb{R}}\frac{(\eta(x)-\eta(s))\left[\phi\left(\frac{s-\xi(t)}{\mu_0(t)}, t\right)\cdot \omega\left(\frac{s-\xi}{\mu}\right)\right]\left[\omega\left(\frac{x-\xi}{\mu}\right)-\omega\left(\frac{s-\xi}{\mu}\right)\right]}{|x-s|^2}ds\\
&  +\left(\frac{1}{\pi}\int_{\mathbb{R}}\frac{(\eta(x)-\eta(s))\left[\phi\left(\frac{s-\xi(t)}{\mu_0(t)}, t\right)\cdot \omega \left(\frac{s-\xi}{\mu}\right)\right]\left[\omega\left(\frac{x-\xi}{\mu}\right)-\omega\left(\frac{s-\xi}{\mu}\right)\right]}{|x-s|^2}ds\cdot \omega\left(\frac{x-\xi}{\mu}\right)\right)\omega\left(\frac{x-\xi}{\mu}\right)\\
&\text{ in }\mathbb{R}\times [t_0, +\infty).
\end{aligned}
\right.
\end{equation}

Consider the change of variable
\begin{equation*}
t = t(\tau), \quad \frac{d}{d\tau}t  = \mu_0(t),
\end{equation*}
then equation (\ref{e:innerproblembeforechangeofvariables}) becomes
\begin{equation}\label{e:innerproblem}
\left\{
\begin{aligned}
&\partial_\tau \phi = -(-\Delta)^{\frac{1}{2}}\phi + \frac{2}{1+|y|^2}\phi\\
&\quad\quad\quad +\frac{\mu_0}{\pi}\int_{\mathbb{R}}\frac{\left[\phi\left(\frac{x-\xi}{\mu_0}\right)-\phi\left(\frac{s-\xi}{\mu_0}\right)\right]
\left[\omega\left(\frac{x-\xi}{\mu_0}\right)-\omega\left(\frac{s-\xi}{\mu_0}\right)\right]}{|x-s|^2}ds \omega\left(\frac{x-\xi}{\mu_0}\right)\\
&\quad\quad\quad +H[\lambda,\xi,\dot{\lambda},\dot{\xi},\phi](y,t(\tau)) = L_\omega[\phi](y) + H[\lambda,\xi,\dot{\lambda},\dot{\xi},\phi](y,t(\tau))\text{ in }B_{2R}(0)\times [\tau_0, \infty),\\
&\phi = 0\quad\text{in }B_{2R}(0)\times \{\tau_0\},
\end{aligned}
\right.
\end{equation}
with
\begin{equation*}
\begin{aligned}
&H[\lambda,\xi,\dot{\lambda},\dot{\xi},\phi](y,t) =
\mu_0\Pi_{U^\perp}\mathcal{E}^*(\xi+\mu_0y, t) +\frac{2\frac{\mu_0}{\mu}}{1+\left|\frac{\mu_0}{\mu}y\right|^2}\Pi_{U^\perp}\psi \\ & \quad\quad\quad \quad\quad\quad\quad\quad\quad -\frac{\mu_0}{\pi}\int_{\mathbb{R}}\frac{\left[\psi(x)\cdot \omega\left(\frac{x-\xi}{\mu}\right)-\psi(s)\cdot \omega\left(\frac{s-\xi}{\mu}\right)\right]\left[\omega\left(\frac{x-\xi}{\mu}\right)-\omega\left(\frac{s-\xi}{\mu}\right)\right]}{|x-s|^2}ds \\
& + \left(\frac{\mu_0}{\pi}\int_{\mathbb{R}}\frac{\left[\psi(x)\cdot \omega\left(\frac{x-\xi}{\mu}\right)-\psi(s)\cdot \omega\left(\frac{s-\xi}{\mu}\right)\right]\left[\omega\left(\frac{x-\xi}{\mu}\right)-\omega\left(\frac{s-\xi}{\mu}\right)\right]}{|x-s|^2}ds\cdot \omega\left(\frac{x-\xi}{\mu}\right)\right)\omega\left(\frac{x-\xi}{\mu}\right)\\
& + B^1[\phi] + B^2[\phi] + B^3[\phi]
\end{aligned}
\end{equation*}
and $\tau_0$ satisfies $t_0 = t(\tau_0)$.

(\ref{e:innerproblem}) or (\ref{e:innerproblembeforechangeofvariables}) is the so-called inner problem and (\ref{e:outerproblem}) is the outer problem. The strategy we use to solve this system is:
for a fixed function $\phi(y, t)$ in a suitable class and parameter functions $\xi$, $\lambda$, we solve equation (\ref{e:outerproblem}) for $\psi$ as an operator $\Psi = \Psi[\lambda, \xi, \dot{\lambda}, \dot{\xi}, \phi]$. Then we insert this $\Psi$ into equation (\ref{e:innerproblem}) and obtain
\begin{equation}\label{e:innerproblemsubstitute}
\left\{
\begin{array}{ll}
\partial_\tau\phi = L_\omega[\phi] + H[\lambda,\xi,\dot{\lambda},\dot{\xi},\phi]\quad \text{in }B_{2R}(0)\times [\tau_0, +\infty),\\
\phi(\cdot, \tau_0) = 0\quad\text{in }B_{2R}(0).
\end{array}
\right.
\end{equation}
We will solve (\ref{e:innerproblemsubstitute}) by the Contraction Mapping Theorem involving an inverse for the following linear problem
\begin{equation}\label{e:innerproblemlinear}
\left\{
\begin{array}{ll}
\partial_\tau\phi = L_\omega[\phi] + h(y, \tau)\quad \text{in }B_{2R}(0)\times [\tau_0, +\infty),\\
\phi(\cdot, \tau_0) = 0\quad\text{in }B_{2R}(0).
\end{array}
\right.
\end{equation}
Providing certain orthogonality conditions hold, we will find a solution $\phi$ of (\ref{e:innerproblemlinear}) which defines a linear operator of $h$ with good $L^\infty$-weight estimates.

The class of functions $h$ we will use is well-represented by asymptotic behavior of the error term $\mu_0\Pi_{U^\perp}\mathcal{E}^*$ in the scaled variables $(y,\tau)$. From the computations in Section 2.4, we know that
\begin{equation}\label{e3:estimateoferror}
|\mu_0\Pi_{U^\perp}\mathcal{E}^*(\xi + \mu_0 y, t)| \lesssim \frac{\mu_0(t)}{1+\rho^{1+a}},\quad \rho = |y|
\end{equation}
for some constant $0 < a < 1$. Observing that
\begin{equation*}
\tau(t) \thicksim \frac{1}{\kappa_0}e^{\kappa_0 t} \thicksim \frac{1}{\kappa_0}\frac{1}{\mu_0(t)},
\end{equation*}
hence
\begin{equation*}
\mu_0(\tau): = \mu_0(t(\tau)) \thicksim e^{-\kappa_0\frac{1}{\kappa_0}\log(\kappa_0\tau)}\thicksim \frac{1}{\kappa_0\tau}.
\end{equation*}

\section{The outer problem}
In this section, we solve the outer problem (\ref{e:outerproblem}) for a given small function $\phi$ and the considered range of parameters $\lambda$, $\xi$, $\dot{\lambda}$, $\dot{\xi}$ in the form of a nonlinear operator
$$
\psi(x, t) = \Psi[\lambda, \xi, \dot{\lambda}, \dot{\xi}, \phi](x,t).
$$
We consider the following problem
\begin{equation}\label{e:equation4.1}
\begin{aligned}
\partial_t \psi &= (-\Delta)^{\frac{1}{2}}\psi + f(x, t)\quad \text{ in }\mathbb{R}\times [t_0, +\infty)
\end{aligned}
\end{equation}
with
\begin{equation*}
\begin{aligned}
&f(x, t)\\ &= (1-\eta)\left[\frac{2\frac{1}{\mu}}{1+\left|\frac{x-\xi}{\mu}\right|^2}\psi -\frac{1}{\pi}\int_{\mathbb{R}}\frac{\left[\psi(x)\cdot \omega\left(\frac{x-\xi}{\mu}\right)-\psi(s)\cdot \omega\left(\frac{s-\xi}{\mu}\right)\right]\left[\omega\left(\frac{x-\xi}{\mu}\right)-\omega\left(\frac{s-\xi}{\mu}\right)\right]}{|x-s|^2}ds\right. \\
& \left.+ \left(\frac{1}{\pi}\int_{\mathbb{R}}\frac{\left[\psi(x)\cdot \omega\left(\frac{x-\xi}{\mu}\right)-\psi(s)\cdot \omega\left(\frac{s-\xi}{\mu}\right)\right]\left[\omega\left(\frac{x-\xi}{\mu}\right)-\omega\left(\frac{s-\xi}{\mu}\right)\right]}{|x-s|^2}ds\cdot \omega\left(\frac{x-\xi}{\mu}\right)\right)\omega\left(\frac{x-\xi}{\mu}\right)\right]\\
&  + \frac{\dot{\mu}_0}{\mu_0}\eta y\cdot \nabla_y\phi + \eta\frac{\dot{\xi}}{\mu_0}\cdot\nabla_y\phi + N_U(\Pi_{U^\perp}[\Phi^0 + Z^*+\eta\phi+\psi]) + (1-\eta)\Pi_{U^\perp}\mathcal{E}^*\\
&  -\left[(-\Delta)^{\frac{1}{2}}\eta\right] \phi - \partial_t\eta\phi\left(\frac{x-\xi(t)}{\mu_0(t)}\right)+\frac{1}{\pi}\int_{\mathbb{R}}\frac{(\eta(x)-\eta(s))\left(\phi\left(\frac{x-\xi(t)}
{\mu_0(t)}, t\right)-\phi\left(\frac{s-\xi(t)}{\mu_0(t)}, t\right)\right)}{|x-s|^2}ds\\
&  -\frac{1}{\pi}\int_{\mathbb{R}}\frac{(\eta(x)-\eta(s))\left[\phi\left(\frac{s-\xi(t)}{\mu_0(t)}, t\right)\cdot \omega\left(\frac{s-\xi}{\mu}\right)\right]\left[\omega\left(\frac{x-\xi}{\mu}\right)-\omega\left(\frac{s-\xi}{\mu}\right)\right]}{|x-s|^2}ds+\\
&  \left(\frac{1}{\pi}\int_{\mathbb{R}}\frac{(\eta(x)-\eta(s))\left[\phi\left(\frac{s-\xi(t)}{\mu_0(t)}, t\right)\cdot \omega \left(\frac{s-\xi}{\mu}\right)\right]\left[\omega\left(\frac{x-\xi}{\mu}\right)-\omega\left(\frac{s-\xi}{\mu}\right)\right]}{|x-s|^2}ds\cdot \omega\left(\frac{x-\xi}{\mu}\right)\right)\omega\left(\frac{x-\xi}{\mu}\right).
\end{aligned}
\end{equation*}
Let us fixe a small constant $\sigma\in (0,\frac{7}{3})$. For a function $h(t): [t_0, \infty)\to \mathbb{R}$ and $\delta > 0$, we introduce the following weighted $L^\infty$ norm
\begin{equation*}
\|h\|_\delta: = \|\mu_0(t)^{-\delta}h(t)\|_{L^\infty[t_0, \infty)}.
\end{equation*}
In what follows we assume that the parameter functions $\lambda$, $\xi$, $\dot{\lambda}$, $\dot{\xi}$ satisfy
\begin{equation}\label{e:assumptionsondotlambda}
\|\dot{\lambda}\|_{1+\sigma} + \|\dot{\xi}\|_{1+\sigma} \leq c
\end{equation}
and
\begin{equation}\label{e:assumptionsonlambda}
\|\lambda(t)\|_{1+\sigma} + \|\xi(t) - q\|_{1+\sigma}\leq c
\end{equation}
for some positive constant $c$.

Equation (\ref{e:equation4.1}) is nonlinear, we will use the Contraction Mapping Theorem to solve it.
For this, we first need an estimate for the corresponding linear equation.
This is the context of the following subsection. We will use the symbol $\lesssim$ to say $\leq C$, for a constant $C > 0$, its value may change from line to line but it is independent of $t$ and $t_0$.
\subsection{The nonhomogeneous linear half heat equation}
In this subsection, we construct a solution $\psi$ for the nonhomogeneous linear half heat equation
\begin{equation}\label{e:outproblemmode}
\begin{array}{ll}
        \psi_t = -(-\Delta)^{\frac{1}{2}}\psi + f \quad\text{ in }\mathbb{R}\times [t_0, +\infty)
       \end{array}
\end{equation}
which decays at infinity.
We assume that for two real numbers $\nu, \alpha\in (0, 1)$, the nonhomogeneous term $f(x, t)$ satisfies
\begin{equation}\label{e:assumptiononf}
|f(x, t)| \leq M\frac{\mu_0^{\nu-1}(t)}{1+|y|^{1+\alpha}},\quad y = \frac{x-\xi(t)}{\mu_0(t)}
\end{equation}
and denote by $\|f\|_{*, 1+\alpha, \nu}$ the least $M > 0$ such that (\ref{e:assumptiononf}) holds.
Using the heat kernel (see \cite{CabreRoquejoffreCMP}), we know that
\begin{equation}\label{e:duhamelintegral}
\psi(x, t) = \int_{t}^{+\infty}\int_{\mathbb{R}}\frac{s-t}{(s-t)^2+(x-y)^2}f(y, s)dyds = \int_{t}^{+\infty}\int_{\mathbb{R}}\frac{1}{s-t}\frac{1}{1+\left(\frac{x-y}{s-t}\right)^2}f(y, s)dyds
\end{equation}
is a solution of (\ref{e:outproblemmode}).
Then we have
\begin{lemma}\label{l4:lemma4.1}
Assume that $\|f\|_{*, 1+\alpha, \nu} < +\infty$ for some $\nu, \alpha\in (0, 1)$.
Let $\psi$ be the solution of (\ref{e:outproblemmode}) defined by (\ref{e:duhamelintegral}),
then we have
\begin{equation}\label{e:heatestimation1}
|\psi(x, t)|\lesssim \|f\|_{*, 1+\alpha, \nu}\frac{\mu_0^{\nu}(t)\log(1+|y|)}{1+|y|^{\alpha}}
\end{equation}
with $y = \frac{x-\xi(t)}{\mu_0(t)}$. Moreover, the following local H\"{o}lder estimate holds,
\begin{equation}\label{e:holderestimateforheatequation}
[\psi(\cdot, t)]_{\eta, B_{3\mu_0(t)R}(\xi)}\lesssim \|f\|_{*, 1+\alpha, \nu}\frac{\mu_0^{\nu-\eta}(t)\log(1+|y|)}{1+|y|^{\alpha+\eta}} \quad\text{ for }|y| = \left|\frac{x-\xi(t)}{\mu_0(t)}\right|\leq 3R.
\end{equation}
Here $[\psi]_{\eta, B_{3\mu_0(t)R}(\xi)}: = \sup_{x\neq y\in B_{3\mu_0(t)R}(\xi)}\frac{|f(x, t) - f(y, t)|}{|x-y|^\eta}$ is the local H\"{o}lder norm with respect to the space variable, $\eta\in (\frac{1}{2}, 1)$.
\end{lemma}
\begin{proof}
We prove estimate (\ref{e:heatestimation1}).
Since $|f(x, t)| \leq \|f\|_{*, 1+\alpha, \nu}\frac{\mu_0^{\nu-1}(t)}{1+|y|^{1+\alpha}}$, $y = \frac{|x-\xi(t)|}{\mu_0(t)}$ and $(\xi(t)-q)/\mu_0(t) =o(1)$, we have
\begin{equation*}
\begin{aligned}
&|f(x, t)| \leq \|f\|_{*, 1+\alpha, \nu}\frac{\mu_0^{\nu-1}(t)}{1+|y|^{1+\alpha}}\thicksim \|f\|_{*, 1+\alpha, \nu}\frac{\mu_0^{\nu-1}(t)}{1+\left|\frac{x-q}{\mu_0(t)}\right|^{1+\alpha}}.
\end{aligned}
\end{equation*}
Then
\begin{equation*}
\begin{aligned}
&|\psi(x, t)| \lesssim \|f\|_{*, 1+\alpha, \nu}\int_{t}^{t + 1}\int_{\mathbb{R}}\frac{1}{s-t}\frac{1}{1+\left(\frac{x-y}{s-t}\right)^2}\frac{\mu_0^{\nu-1}(s)}
{1+\left|\frac{y-q}{\mu_0(s)}\right|^{1+\alpha}}dyds\\
&\quad\quad\quad\quad + \|f\|_{*, 1+\alpha, \nu}\int_{t + 1}^{+\infty}\int_{\mathbb{R}}\frac{1}{s-t}\frac{1}{1+\left(\frac{x-y}{s-t}\right)^2}\frac{\mu_0^{\nu-1}(s)}
{1+\left|\frac{y-q}{\mu_0(s)}\right|^{1+\alpha}}dyds.
\end{aligned}
\end{equation*}

To estimate the integral $\int_{t + 1}^{+\infty}\int_{\mathbb{R}}\frac{1}{s-t}\frac{1}{1+\left(\frac{x-y}{s-t}\right)^2}\frac{\mu_0^{\nu-1}(s)}
{1+\left|\frac{y-q}{\mu_0(s)}\right|^{1+\alpha}}dyds$.
First, we assume that $|y|\geq C$ for some positive constant $C$, that is to say, $\frac{1}{|y|}\thicksim \frac{1}{1+|y|}$, then
\begin{equation*}
\begin{aligned}
&\int_{t + 1}^{+\infty}\int_{\mathbb{R}}\frac{1}{s-t}\frac{1}{1+\left(\frac{x-y}{s-t}\right)^2}\frac{\mu_0^{\nu-1}(s)}
{1+\left|\frac{y-q}{\mu_0(s)}\right|^{1+\alpha}}dyds\\
&\quad = \int_{t + 1}^{+\infty}\frac{\mu_0^{\nu-1}(s)}{s-t}ds\int_{\mathbb{R}}\frac{1}{1+\left(\frac{y}{s-t}\right)^2}\frac{1}{1+\left|\frac{x-y-q}
{\mu_0(s)}\right|^{1+\alpha}}dy\\
\end{aligned}
\end{equation*}
\begin{equation*}
\begin{aligned}
&\quad \lesssim \int_{t + 1}^{+\infty}\frac{\mu_0^{\nu-1}(s)}{s-t}ds\int_{\mathbb{R}}\frac{1}{1+\left|\frac{x-y-q}{\mu_0(s)}\right|^{1+\alpha}}dy\\
&\quad \lesssim \int_{t + 1}^{+\infty}\frac{\mu_0^{\nu-1}(s)}{s-t}\frac{\mu_0^{1+\alpha}(s)}{\left|x-q\right|^\alpha}ds\\
&\quad \lesssim \frac{\mu_0^{\nu+\alpha}(t)}{|x-q|^{\alpha}}\lesssim\frac{\mu_0^{\nu}(t)\log(1+|y|)}{1+|y|^{\alpha}}.
\end{aligned}
\end{equation*}
If $|y|\leq C$ for some positive constant $C$, then
\begin{equation*}
\begin{aligned}
&\int_{t+1}^{+\infty}\int_{\mathbb{R}}\frac{1}{s-t}\frac{1}{1+\left(\frac{x-y}{s-t}\right)^2}\frac{\mu_0^{\nu-1}(s)}{1+\left|\frac{y-q}
{\mu_0(s)}\right|^{1+\alpha}}dyds\\
&\quad\quad\quad\quad\lesssim \int_{t+1}^{+\infty}\frac{\mu_0^{\nu}(s)}{s-t}ds\int_{\mathbb{R}}\frac{1}{1+\mu_0^2(s)
\left(\frac{\frac{x-q}{\mu_0(s)}-y}{s-t}\right)^2}\frac{1}{1+|y|^{1+\alpha}}dy\\
&\quad\quad\quad\quad\lesssim \int_{t+1}^{+\infty}\mu_0^{\nu}(s)ds\\
&\quad\quad\quad\quad\lesssim \mu_0^{\nu}(t)\lesssim\frac{\mu_0^{\nu}(t)\log(1+|y|)}{1+|y|^{\alpha}}.
\end{aligned}
\end{equation*}

Next, we estimate $\int_{t}^{t + 1}\int_{\mathbb{R}}\frac{1}{s-t}\frac{1}{1+\left(\frac{x-y}{s-t}\right)^2}\frac{\mu_0^{\nu-1}(s)}{1+\left|\frac{y-q}{\mu_0(s)}\right|^{1+\alpha}}dyds$.
We write
\begin{equation*}
\begin{aligned}
&\psi(x, t) = \int_{t}^{t + 1}\int_{\mathbb{R}}\frac{1}{s-t}\frac{1}{1+\left(\frac{x-y}{s-t}\right)^2}\frac{\mu_0^{\nu+\alpha}(s)}{\mu_0^{1+\alpha}(s)+|y-q|^{1+\alpha}}dyds\\
&\quad\quad\quad = \int_{t}^{t + 1}\int_{|x-y|\leq 1}\frac{1}{s-t}\frac{1}{1+\left(\frac{x-y}{s-t}\right)^2}\frac{\mu_0^{\nu+\alpha}(s)}{\mu_0^{1+\alpha}(s)+|y-q|^{1+\alpha}}dyds\\
&\quad\quad\quad\quad + \int_{t}^{t + 1}\int_{|x-y|\geq 1}\frac{1}{s-t}\frac{1}{1+\left(\frac{x-y}{s-t}\right)^2}\frac{\mu_0^{\nu+\alpha}(s)}{\mu_0^{1+\alpha}(s)+|y-q|^{1+\alpha}}dyds\\
&\quad\quad\quad = \int_{t}^{t + 1}\int_{0 < x-y\leq 1}\frac{1}{s-t}\frac{1}{1+\left(\frac{x-y}{s-t}\right)^2}\frac{\mu_0^{\nu+\alpha}(s)}{\mu_0^{1+\alpha}(s)+|y-q|^{1+\alpha}}dyds\\
&\quad\quad\quad\quad + \int_{t}^{t + 1}\int_{-1 < x-y\leq 0}\frac{1}{s-t}\frac{1}{1+\left(\frac{x-y}{s-t}\right)^2}\frac{\mu_0^{\nu+\alpha}(s)}{\mu_0^{1+\alpha}(s)+|y-q|^{1+\alpha}}dyds\\
&\quad\quad\quad\quad + \int_{t}^{t + 1}\int_{x-y \geq 1}\frac{1}{s-t}\frac{1}{1+\left(\frac{x-y}{s-t}\right)^2}\frac{\mu_0^{\nu+\alpha}(s)}{\mu_0^{1+\alpha}(s)+|y-q|^{1+\alpha}}dyds\\
&\quad\quad\quad\quad + \int_{t}^{t + 1}\int_{x-y \leq -1}\frac{1}{s-t}\frac{1}{1+\left(\frac{x-y}{s-t}\right)^2}\frac{\mu_0^{\nu+\alpha}(s)}{\mu_0^{1+\alpha}(s)+|y-q|^{1+\alpha}}dyds\\
&\quad\quad\quad := I_1 + I_2 + I_3 + I_4.
\end{aligned}
\end{equation*}
Using the variable transformation $p = \frac{x-y}{s-t}$, $\frac{ds}{s-t} = -\frac{1}{p}dp$, we have
\begin{equation*}
\begin{aligned}
&I_1 \lesssim \mu_0^{\nu+\alpha}(t)\int_{x-1}^x\frac{1}{\mu_0^{1+\alpha}(t+1) + |y-q|^{1+\alpha}}dy\int_{x-y}^\infty\frac{1}{p}\frac{1}{1+p^2}dp\\
&\quad\lesssim \mu_0^{\nu+\alpha}(t) \int_{x-1}^x\frac{1}{\mu_0^{1+\alpha}(t+1) + |y-q|^{1+\alpha}}\log\frac{\sqrt{1+(x-y)^2}}{x-y}dy\\
&\quad\lesssim \mu_0^{\nu+\alpha}(t)\int_{x-1}^x\frac{1}{\mu_0^{1+\alpha}(t+1) + |y-q|^{1+\alpha}}(\log 2-\log(x-y))dy\\
&\quad = \mu_0^{\nu+\alpha}(t)\int_{x-q-1}^{x-q}\frac{1}{\mu_0^{1+\alpha}(t+1) + |y|^{1+\alpha}}(\log 2-\log(x-q-y))dy\\
&\quad\lesssim \mu_0^{\nu+\alpha}(t)\frac{\log(1+|x-q|)}{\mu_0^{\alpha}(t+1) + |x-q|^{\alpha}}\lesssim \frac{\mu_0^{\nu}(t)\log(1+|y|)}{1+|y|^{\alpha}}.
\end{aligned}
\end{equation*}
Similarly, we have
\begin{equation*}
\begin{aligned}
&I_2 \lesssim \mu_0^{\nu+\alpha}(t)\int_{x}^{x+1}\frac{1}{\mu_0^{1+\alpha}(t+1) + |y-q|^{1+\alpha}}dy\int_{y-x}^\infty\frac{1}{p}\frac{1}{1+p^2}dp\\
&\quad\lesssim \mu_0^{\nu+\alpha}(t)\int_{x}^{x+1}\frac{1}{\mu_0^{1+\alpha}(t+1) + |y-q|^{1+\alpha}}\log\frac{\sqrt{1+(y-x)^2}}{y-x}dy\\
&\quad\lesssim \mu_0^{\nu+\alpha}(t)\int_{x}^{x+1}\frac{1}{\mu_0^{1+\alpha}(t+1) + |y-q|^{1+\alpha}}(\log 2-\log(y-x))dy\\
&\quad\lesssim \mu_0^{\nu+\alpha}(t)\int_{x-q}^{x-q+1}\frac{1}{\mu_0^{1+\alpha}(t+1) + |y|^{1+\alpha}}(\log 2-\log(y-x+q))dy\\
&\quad\lesssim \mu_0^{\nu+\alpha}(t)\frac{\log(1+|x-q|)}{\mu_0^{\alpha}(t+1) + |x-q|^{\alpha}}\lesssim \frac{\mu_0^{\nu}(t)\log(1+|y|)}{1+|y|^{\alpha}},
\end{aligned}
\end{equation*}
\begin{equation*}
\begin{aligned}
&I_3 \lesssim \mu_0^{\nu+\alpha}(t)\int_{-\infty}^{x-1}\frac{1}{\mu_0^{1+\alpha}(t+1) + |y-q|^{1+\alpha}}dy\int_{x-y}^\infty\frac{1}{p}\frac{1}{1+p^2}dp\\
&\quad\lesssim \mu_0^{\nu+\alpha}(t)\int_{-\infty}^{x-1}\frac{1}{\mu_0^{1+\alpha}(t+1) + |y-q|^{1+\alpha}}\log\frac{\sqrt{1+(x-y)^2}}{x-y} dy\\
&\quad\lesssim \mu_0^{\nu+\alpha}(t)\frac{1}{\mu_0^{\alpha}(t+1) + |x-q|^{\alpha}}\lesssim \frac{\mu_0^{\nu}(t)\log(1+|y|)}{1+|y|^{\alpha}},
\end{aligned}
\end{equation*}
\begin{equation*}
\begin{aligned}
&I_4\lesssim \mu_0^{\nu+\alpha}(t)\int_{x+1}^{+\infty}\frac{1}{\mu_0^{1+\alpha}(t+1) + |y-q|^{1+\alpha}}dy\int_{y-x}^\infty\frac{1}{p}\frac{1}{1+p^2}dp\\
&\quad\lesssim \mu_0^{\nu+\alpha}(t)\int_{x+1}^{+\infty}\frac{1}{\mu_0^{1+\alpha}(t+1) + |y-q|^{1+\alpha}}\log\frac{\sqrt{1+(y-x)^2}}{y-x}dy\\
&\quad\lesssim \mu_0^{\nu+\alpha}(t)\frac{1}{\mu_0^{\alpha}(t+1) + |x-q|^{\alpha}}\lesssim \frac{\mu_0^{\nu}(t)\log(1+|y|)}{1+|y|^{\alpha}}.
\end{aligned}
\end{equation*}
Combine the above estimates, we get (\ref{e:heatestimation1}).

Now we prove the local H\"{o}lder estimate (\ref{e:holderestimateforheatequation}), define
\begin{equation*}
\psi(x, t):=\tilde{\psi}\left(\frac{x-\xi}{\mu_0},\tau(t)\right)
\end{equation*}
where $\dot{\tau}(t) = \mu_0^{-1}(t)$, namely $\tau(t)\thicksim e^{\kappa_0t}$ as $t\to +\infty$. Without loss of generality, we assume $\tau(t_0)\geq 2$ by fixing $t_0$ large enough, then $\tilde{\psi}$ satisfies the following fractional heat equation with a drift
\begin{equation}\label{e4:15}
\partial_\tau\tilde{\psi} = -(-\Delta)^{\frac{1}{2}}\tilde{\psi} + a(z, t)\cdot \nabla_z\tilde{\psi} + \tilde{f}(z, \tau)
\end{equation}
for $|z|\leq \delta\mu_0^{-1}$,
where
\begin{equation*}
\tilde{f}(z, \tau) = \mu_0(t)f(\xi + \mu_0 z, t(\tau)).
\end{equation*}
The uniformly small coefficients $a(z, t)$ in \eqref{e4:15} is given by
\begin{equation*}
a(z, t):=\dot{\mu_0}z+\dot{\xi}.
\end{equation*}
Then
\begin{equation*}\label{e4:18}
|\tilde{f}(z, \tau)|\lesssim \mu_0(t(\tau))^{\nu}\frac{\|f\|_{*, 1+\alpha, \nu}}{1+|z|^{1+\alpha}}
\end{equation*}
and
\begin{equation*}\label{e4:19}
|\tilde{\psi}(z, \tau)|\lesssim \mu_0(t(\tau))^{\nu}\frac{\|f\|_{*, 1+\alpha, \nu}\log(1+|z|)}{1+|z|^\alpha}.
\end{equation*}
Now for a fixed constant $\eta\in (\frac{1}{2}, 1)$ and $\tau_1 \geq \tau(t_0)+2$, from the regularity estimates for parabolic integro-differential (see \cite{silvestreium2012differentiability} and \cite{silvestre2014regularity}), we obtain
\begin{equation*}
\begin{aligned}
~[\tilde{\psi}(\cdot,\tau_1)]_{\eta, B_{1}(0)} &\lesssim \|\tilde{\psi}\|_{L^\infty}+\|\tilde{f}\|_{L^\infty}\\
&\lesssim \mu_0((\tau_1-1))^{\nu}\|f\|_{*, 1+\alpha, \nu}\\
&\lesssim \mu_0(t(\tau_1))^{\nu}\|f\|_{*, 1+\alpha, \nu}.
\end{aligned}
\end{equation*}
Therefore, choosing an appropriate constant $c_n$ such that for any $t\geq c_nt_0$ we get
\begin{equation}\label{e4:20}
(3R\mu_0)^{\eta}[\psi(\cdot,t)]_{\eta, B_{3R\mu_0}(\xi)}\lesssim \mu_0^\nu\|f\|_{*, 1+\alpha, \nu}.
\end{equation}
Estimate (\ref{e4:20}) also holds for $t_0\leq t\leq c_nt_0$ by a similar parabolic regularity estimate. Hence (\ref{e:holderestimateforheatequation}) holds for any $t\geq t_0$.
\end{proof}
\subsection{The outer problem}
We assume that there exists a number $M > 0$ such that
\begin{equation}\label{e4:assumptiononphi}
(1+|y|)|\nabla\phi|\chi_{\{|y|\leq 2R\}} + |\phi|\leq M\frac{\mu_0^\sigma(t)}{1+|y|^\alpha}
\end{equation}
for a given $0 < \alpha \leq a < 1$, $\sigma\in (0, 1)$ and the least number $M$ is denoted as $\|\phi\|_{\alpha,\sigma}$. Suppose
\begin{equation*}
\|\phi\|_{\alpha,\sigma}\leq ce^{-\varepsilon t_0}
\end{equation*}
for some $\varepsilon > 0$ small enough. Then we have the following result.
\begin{prop}\label{p4:4.1}
There exists $t_0$ large such that problem (\ref{e:outerproblem}) has a solution $\psi = \Psi(\lambda, \xi, \dot{\lambda}, \dot{\xi}, \phi)$ and for $y = \frac{x-\xi(t)}{\mu_0(t)}$, the following estimates hold,
\begin{equation}\label{e4:pointwiseestimate}
|\psi(x, t)|\lesssim e^{-\varepsilon t_0}\frac{\mu_0^{\sigma}(t)\log(1+|y|)}{1+|y|^\alpha}
\end{equation}
and
\begin{equation}\label{e:estimateholder}
[\psi(x, t)]_{\eta, B_{3\mu_0(t)R}(\xi)}\lesssim e^{-\varepsilon t_0}\frac{\mu_0^{\sigma-\eta}(t)\log(1+|y|)}{1+|y|^{\alpha+\eta}}\quad\text{ for }\quad |y| = \left|\frac{x-\xi(t)}{\mu_0(t)}\right|\leq 3R.
\end{equation}
\end{prop}
\begin{proof}
Lemma \ref{l4:lemma4.1} defines a linear operator $T$ that associates the solution $\psi = T(f)$ of problem (\ref{e:outproblemmode}) to any given functions $f(x, t)$. If we define $\|\psi\|_{**}$ as the least $M > 0$ such that
\begin{equation*}
|\psi(x, t)|\lesssim M\frac{\mu_0^{\sigma}(t)\log(1+|y|)}{1+|y|^{\alpha}},
\end{equation*}
then Lemma \ref{l4:lemma4.1} tells us
\begin{equation*}
\|T(f)\|_{**}\lesssim \|f\|_{*,1+\alpha, \sigma}.
\end{equation*}
Hence the function $\psi$ is a solution to (\ref{e:outerproblem}) if it is a fixed point of the operator defined as
\begin{equation*}\label{e4:34}
\mathcal{A}(\psi):=T(f(\psi)),
\end{equation*}
where
\begin{equation*}
\begin{aligned}
&f(\psi)(x, t)\\ &= (1-\eta)\left[\frac{2\frac{1}{\mu}}{1+\left|\frac{x-\xi}{\mu}\right|^2}\psi -\frac{1}{\pi}\int_{\mathbb{R}}\frac{\left[\psi(x)\cdot \omega\left(\frac{x-\xi}{\mu}\right)-\psi(s)\cdot \omega\left(\frac{s-\xi}{\mu}\right)\right]\left[\omega\left(\frac{x-\xi}{\mu}\right)-\omega\left(\frac{s-\xi}{\mu}\right)\right]}{|x-s|^2}ds\right. \\
& \left.+ \left(\frac{1}{\pi}\int_{\mathbb{R}}\frac{\left[\psi(x)\cdot \omega\left(\frac{x-\xi}{\mu}\right)-\psi(s)\cdot \omega\left(\frac{s-\xi}{\mu}\right)\right]\left[\omega\left(\frac{x-\xi}{\mu}\right)-\omega\left(\frac{s-\xi}{\mu}\right)\right]}{|x-s|^2}ds\cdot \omega\left(\frac{x-\xi}{\mu}\right)\right)\omega\left(\frac{x-\xi}{\mu}\right)\right]\\
&  + \frac{\dot{\mu}_0}{\mu_0}\eta y\cdot \nabla_y\phi + \eta\frac{\dot{\xi}}{\mu_0}\cdot\nabla_y\phi + N_U(\Pi_{U^\perp}[\Phi^0 + Z^*+\eta\phi+\psi]) + (1-\eta)\Pi_{U^\perp}\mathcal{E}^*\\
&  -\left[(-\Delta)^{\frac{1}{2}}\eta\right] \phi - \partial_t\eta\phi\left(\frac{x-\xi(t)}{\mu_0(t)}\right)+\frac{1}{\pi}\int_{\mathbb{R}}\frac{(\eta(x)-\eta(s))\left(\phi\left(\frac{x-\xi(t)}
{\mu_0(t)}, t\right)-\phi\left(\frac{s-\xi(t)}{\mu_0(t)}, t\right)\right)}{|x-s|^2}ds\\
&  -\frac{1}{\pi}\int_{\mathbb{R}}\frac{(\eta(x)-\eta(s))\left[\phi\left(\frac{s-\xi(t)}{\mu_0(t)}, t\right)\cdot \omega\left(\frac{s-\xi}{\mu}\right)\right]\left[\omega\left(\frac{x-\xi}{\mu}\right)-\omega\left(\frac{s-\xi}{\mu}\right)\right]}{|x-s|^2}ds+\\
&  \left(\frac{1}{\pi}\int_{\mathbb{R}}\frac{(\eta(x)-\eta(s))\left[\phi\left(\frac{s-\xi(t)}{\mu_0(t)}, t\right)\cdot \omega \left(\frac{s-\xi}{\mu}\right)\right]\left[\omega\left(\frac{x-\xi}{\mu}\right)-\omega\left(\frac{s-\xi}{\mu}\right)\right]}{|x-s|^2}ds\cdot \omega\left(\frac{x-\xi}{\mu}\right)\right)\omega\left(\frac{x-\xi}{\mu}\right).
\end{aligned}
\end{equation*}
We will prove the existence of a fixed point $\psi$ for $\mathcal{A}$ by the Contraction Mapping Theorem.
To do this, we have the following estimates:
\begin{equation}\label{e:estimate4.1}
\begin{aligned}
&\left|(1-\eta)\frac{2\frac{1}{\mu}}{1+\left|\frac{x-\xi}{\mu}\right|^2}\psi\right|\lesssim \mu_0^\varepsilon(t_0)\|\psi\|_{**}\frac{\mu_0^{\sigma-1}(t)}{1+|y|^{1+\alpha}},
\end{aligned}
\end{equation}
\begin{equation}\label{e:estimate4.2}
\begin{aligned}
&\left|\frac{1}{\pi}\int_{\mathbb{R}}\frac{\left[\psi(x)\cdot \omega\left(\frac{x-\xi}{\mu}\right)-\psi(s)\cdot \omega\left(\frac{s-\xi}{\mu}\right)\right]\left[\omega\left(\frac{x-\xi}{\mu}\right)-\omega\left(\frac{s-\xi}{\mu}\right)
\right]}{|x-s|^2}ds\right|\lesssim \mu_0^\varepsilon(t_0)\|\psi\|_{**}\frac{\mu_0^{\sigma-1}(t)}{1+|y|^{1+\alpha}},
\end{aligned}
\end{equation}
\begin{equation}\label{e:estimate4.3}
\begin{aligned}
&\left|\left(\frac{1}{\pi}\int_{\mathbb{R}}\frac{\left[\psi(x)\cdot \omega\left(\frac{x-\xi}{\mu}\right)-\psi(s)\cdot \omega\left(\frac{s-\xi}{\mu}\right)\right]\left[\omega\left(\frac{x-\xi}{\mu}\right)-\omega\left(\frac{s-\xi}{\mu}\right)\right]}{|x-s|^2}ds\cdot \omega\left(\frac{x-\xi}{\mu}\right)\right)\omega\left(\frac{x-\xi}{\mu}\right)\right|\\
&\quad\quad\quad\quad\quad\quad\quad\quad\quad\quad\quad\quad\quad\quad\quad\quad\quad\quad\quad\quad\quad\quad\quad\quad\quad\quad\quad\quad\quad\quad
\lesssim\mu_0^\varepsilon(t_0)\|\psi\|_{**}\frac{\mu_0^{\sigma-1}(t)}{1+|y|^{1+\alpha}},
\end{aligned}
\end{equation}
\begin{equation}\label{e:estimate4.4}
\begin{aligned}
&\left|\frac{\dot{\mu}_0}{\mu_0}\eta y\cdot \nabla_y\phi + \eta\frac{\dot{\xi}}{\mu_0}\cdot\nabla_y\phi\right|\lesssim \|\phi\|_{\alpha,\sigma}\frac{\mu_0(t)^{\sigma-1}}{1+|y|^{1+\alpha}},
\end{aligned}
\end{equation}
\begin{equation}\label{e:estimate4.5}
\begin{aligned}
&\left|-\left[(-\Delta)^{\frac{1}{2}}\eta\right] \phi\left(\frac{x-\xi(t)}{\mu_0(t)}\right) - \partial_t\eta\phi\left(\frac{x-\xi(t)}{\mu_0(t)}\right)\right| \lesssim\|\phi\|_{\alpha,\sigma}\frac{\mu_0^{\sigma-1}(t)}{1+|y|^{1+\alpha}},
\end{aligned}
\end{equation}
\begin{equation}\label{e:estimate4.6}
\begin{aligned}
&\left|\frac{1}{\pi}\int_{\mathbb{R}}\frac{(\eta(x)-\eta(s))\left(\phi\left(\frac{x-\xi(t)}{\mu_0(t)}, t\right)-\phi\left(\frac{s-\xi(t)}{\mu_0(t)}, t\right)\right)}{|x-s|^2}ds\right|\lesssim \|\phi\|_{\alpha,\sigma}\frac{\mu_0(t)^{\sigma-1}}{1+|y|^{1+\alpha}},
\end{aligned}
\end{equation}
\begin{equation}\label{e:estimate4.7}
\begin{aligned}
&\left|\frac{1}{\pi}\int_{\mathbb{R}}\frac{(\eta(x)-\eta(s))\left[\phi\left(\frac{s-\xi(t)}{\mu_0(t)}, t\right)\cdot \omega\left(\frac{s-\xi}{\mu}\right)\right]\left[\omega\left(\frac{x-\xi}{\mu}\right)-\omega\left(\frac{s-\xi}{\mu}\right)\right]}
{|x-s|^2}ds\right|\lesssim \|\phi\|_{\alpha,\sigma}\frac{\mu_0(t)^{\sigma-1}}{1+|y|^{1+\alpha}},
\end{aligned}
\end{equation}
\begin{equation}\label{e:estimate4.777}
\begin{aligned}
&\left|\left(\frac{1}{\pi}\int_{\mathbb{R}}\frac{(\eta(x)-\eta(s))\left[\phi\left(\frac{s-\xi(t)}{\mu_0(t)}, t\right)\cdot \omega \left(\frac{s-\xi}{\mu}\right)\right]\left[\omega\left(\frac{x-\xi}{\mu}\right)-\omega\left(\frac{s-\xi}{\mu}\right)\right]}{|x-s|^2}ds\cdot \omega\left(\frac{x-\xi}{\mu}\right)\right)\omega\left(\frac{x-\xi}{\mu}\right)\right|\\
&\quad\quad\quad\quad\quad\quad\quad\quad\quad\quad\quad\quad\quad\quad\quad\quad\quad\quad\quad\quad\quad\quad\quad\quad\quad\quad\quad\quad
\quad\quad\quad\quad \lesssim \|\phi\|_{\alpha,\sigma}\frac{\mu_0(t)^{\sigma-1}}{1+|y|^{1+\alpha}},
\end{aligned}
\end{equation}
\begin{equation}\label{e:estimate4.10}
\left|(1-\eta)\Pi_{U^\perp}\mathcal{E}^*\right|\lesssim \mu_0(t_0)^\varepsilon\frac{\mu_0(t)^{\sigma-1}}{1+|y|^{1+\alpha}},
\end{equation}
\begin{equation}\label{e:estimate4.9}
\begin{aligned}
&\left|N_U(\Pi_{U^\perp}[\Phi^0 + Z^*+\eta\phi+\psi])\right|
\lesssim \mu_0(t_0)^{\varepsilon}\left(\|\phi\|_{\alpha,\sigma} + \|\psi\|_{**}\right)\frac{\mu_0(t)^{\sigma-1}}{1+|y|^{\alpha+1}}.
\end{aligned}
\end{equation}

{\it Proof of (\ref{e:estimate4.1})}:
\begin{equation*}
\begin{aligned}
&\left|(1-\eta)\frac{2\frac{1}{\mu}}{1+\left|\frac{x-\xi}{\mu}\right|^2}\psi\right|\lesssim (1-\eta)\frac{2\frac{1}{\mu}}{1+y^2}\|\psi\|_{**}\frac{\mu_0^{\sigma}(t)\log(1+|y|)}{1+|y|^{\alpha}}\\
&\quad\lesssim (1-\eta)\frac{\log(1+|y|)}{1+|y|}\|\psi\|_{**}\frac{\mu_0^{\sigma-1}(t)}{1+|y|^{1+\alpha}}\\
&\quad\lesssim \mu_0^\varepsilon(t_0)\|\psi\|_{**}\frac{\mu_0^{\sigma-1}(t)}{1+|y|^{1+\alpha}}.
\end{aligned}
\end{equation*}
Here we have used the fact that $|y|\geq R$ when $1-\eta\neq 0$.

{\it Proof of (\ref{e:estimate4.2})}:
\begin{equation*}
\begin{aligned}
&(1-\eta)\left|\frac{1}{\pi}\int_{\mathbb{R}}\frac{\left[\psi(x)\cdot \omega\left(\frac{x-\xi}{\mu}\right)-\psi(s)\cdot \omega\left(\frac{s-\xi}{\mu}\right)\right]\left[\omega\left(\frac{x-\xi}{\mu}\right)-\omega\left(\frac{s-\xi}{\mu}\right)\right]}{|x-s|^2}ds\right|\\
&\quad\lesssim \frac{1}{\pi}\int_{\mathbb{R}}\frac{\left|\omega\left(\frac{x-\xi}{\mu}\right)-\omega\left(\frac{s-\xi}{\mu}\right)\right|^2}{|x-s|^2}ds\psi(x) + \frac{1}{2\pi}\int_{\mathbb{R}}\frac{|\psi(x)-\psi(s)|\left|\omega\left(\frac{x-\xi}{\mu}\right)-\omega\left(\frac{s-\xi}
{\mu}\right)\right|^2}{|x-s|^2}ds\\
&\quad \lesssim \frac{\mu_0^{-1}}{1+|y|^2}\psi(x) + \frac{1}{2\pi}\int_{\mathbb{R}}\frac{|\psi(x)-\psi(s)|\left|\omega\left(\frac{x-\xi}{\mu}\right)-\omega\left(\frac{s-\xi}
{\mu}\right)\right|^2}{|x-s|^2}ds\\
&\quad \lesssim \|\psi\|_{**}\mu_0^{\sigma-1}(t)\frac{\log(1+|y|)}{1+|y|^2}  \lesssim \|\psi\|_{**}\frac{\mu_0(t)^{\sigma-1}}{1+|y|^{\alpha+1}}\frac{\log(1+|y|)}{1+|y|^{1-\alpha}}\\
&\quad \lesssim \mu_0^\varepsilon(t_0)\|\psi\|_{**}\frac{\mu_0^{\sigma-1}(t)}{1+|y|^{1+\alpha}}.
\end{aligned}
\end{equation*}

{\it Proof of (\ref{e:estimate4.3})}:
\begin{equation*}
\begin{aligned}
&\left|\left(\frac{1}{\pi}\int_{\mathbb{R}}\frac{\left[\psi(x)\cdot \omega\left(\frac{x-\xi}{\mu}\right)-\psi(s)\cdot \omega\left(\frac{s-\xi}{\mu}\right)\right]\left[\omega\left(\frac{x-\xi}{\mu}\right)-\omega\left(\frac{s-\xi}{\mu}\right)\right]}{|x-s|^2}ds\cdot \omega\left(\frac{x-\xi}{\mu}\right)\right)\omega\left(\frac{x-\xi}{\mu}\right)\right|\\
&\lesssim \left|\frac{1}{\pi}\int_{\mathbb{R}}\frac{\left[\psi(x)\cdot \omega\left(\frac{x-\xi}{\mu}\right)-\psi(s)\cdot \omega\left(\frac{s-\xi}{\mu}\right)\right]\left[\omega\left(\frac{x-\xi}{\mu}\right)-\omega\left(\frac{s-\xi}{\mu}\right)\right]}{|x-s|^2}ds\right|\\
&\lesssim \mu_0^\varepsilon(t_0)\|\psi\|_{**}\frac{\mu_0^{\sigma-1}(t)}{1+|y|^{1+\alpha}}.
\end{aligned}
\end{equation*}

{\it Proof of (\ref{e:estimate4.4})}:
Using the fact $(1+|y|)|\nabla\phi|\chi_{\{|y|\leq 2R\}} + |\phi|\lesssim \|\phi\|_{\alpha,\sigma}\frac{\mu_0(t)^{\sigma}}{1+|y|^{\alpha}}$, we have
\begin{equation*}
\begin{aligned}
&\left|\frac{\dot{\mu}_0}{\mu_0}\eta y\cdot \nabla_y\phi + \eta\frac{\dot{\xi}}{\mu_0}\cdot\nabla_y\phi\right|\\
&\lesssim \|\phi\|_{\alpha,\sigma}\frac{\mu_0^\sigma(t)}{1+|y|^\alpha} + \mu_0(t)^{\sigma}\|\phi\|_{\alpha,\sigma}\frac{\mu_0^\sigma(t)}{1+|y|^{\alpha+1}}\\
&\lesssim 2R\|\phi\|_{\alpha,\sigma}\frac{\mu_0^\sigma(t)}{1+|y|^{1+\alpha}} + \mu_0(t)^{\sigma}\|\phi\|_{\alpha,\sigma}\frac{\mu_0^\sigma(t)}{1+|y|^{\alpha+1}}\\
&\lesssim 2R\mu_0(t)\|\phi\|_{\alpha,\sigma}\frac{\mu_0^{\sigma-1}(t)}{1+|y|^{1+\alpha}} + \mu_0(t)^{1+\sigma}\|\phi\|_{\alpha,\sigma}\frac{\mu_0^{\sigma-1}(t)}{1+|y|^{\alpha+1}}\\
&\lesssim \mu_0(t_0)^{\varepsilon}\|\phi\|_{\alpha,\sigma}\frac{\mu_0(t)^{\sigma-1}}{1+|y|^{1+\alpha}}.
\end{aligned}
\end{equation*}

{\it Proof of (\ref{e:estimate4.5})}:
Using the fact that $[(-\Delta)^{\frac{1}{2}}\eta_0](x)$ decays like $\frac{1}{|x|^2}$, we have
\begin{equation*}
\begin{aligned}
&\left|-\left[(-\Delta)^{\frac{1}{2}}\eta\right] \phi\left(\frac{x-\xi(t)}{\mu_0(t)}\right) - \partial_t\eta\phi\left(\frac{x-\xi(t)}{\mu_0(t)}\right)\right| \\
&\lesssim\frac{R(\mu_0(t))^{-1}}{R^2+|y|^2}\|\phi\|_{\alpha,\sigma}\frac{\mu_0^\sigma(t)}{1+|y|^\alpha} + \|\phi\|_{\alpha,\sigma}\frac{\mu_0^\sigma(t)}{1+|y|^\alpha}\chi_{\{R\leq |y|\leq 2R\}}\\
&\lesssim\|\phi\|_{\alpha,\sigma}\frac{\mu_0^{\sigma-1}(t)}{1+|y|^{\alpha+1}} + (1+|y|)\|\phi\|_{\alpha,\sigma}\frac{\mu_0^\sigma(t)}{1+|y|^{\alpha+1}}\chi_{\{R\leq |y|\leq 2R\}}\\
&\lesssim \|\phi\|_{\alpha,\sigma}\frac{\mu_0^{\sigma-1}(t)}{1+|y|^{1+\alpha}} + (1+2R)\mu_0(t_0)\|\phi\|_{\alpha,\sigma}\frac{\mu_0^{\sigma-1}(t)}{1+|y|^{1+\alpha}}\\
&\lesssim \|\phi\|_{\alpha,\sigma}\frac{\mu_0^{\sigma-1}(t)}{1+|y|^{1+\alpha}}.
\end{aligned}
\end{equation*}

{\it Proof of (\ref{e:estimate4.6})}:
We have
\begin{equation*}
\begin{aligned}
&\left|\frac{1}{\pi}\int_{\mathbb{R}}\frac{(\eta(x)-\eta(s))\left(\phi\left(\frac{x-\xi(t)}{\mu_0(t)}, t\right)-\phi\left(\frac{s-\xi(t)}{\mu_0(t)}, t\right)\right)}{|x-s|^2}ds\right|\\
&\lesssim \frac{1}{\pi}\int_{\mathbb{R}}\frac{|\eta(x)-\eta(s)|}{|x-s|^2}ds\left|\phi\left(\frac{x-\xi(t)}{\mu_0(t)}, t\right)\right|+ \frac{1}{\pi}\int_{\mathbb{R}}\frac{\left|\eta(x)-\eta(s)\right|\left|\phi\left(\frac{s-\xi(t)}{\mu_0(t)}, t\right)\right|}{|x-s|^2}ds\\
&\lesssim \frac{1}{\pi}\int_{\mathbb{R}}\frac{\left|\eta(x)-\eta(s)\right|}{|x-s|^2}ds \|\phi\|_{\alpha,\sigma}\frac{\mu_0^\sigma(t)}{1+|y|^\alpha} + \frac{1}{\pi}\int_{\mathbb{R}}\frac{\left|\phi\left(\frac{s-\xi(t)}{\mu_0(t)}, t\right)\right|}{|x-s|^2}ds\\
&\lesssim \|\phi\|_{\alpha,\sigma}\frac{\mu_0^\sigma(t)}{1+|y|^\alpha}(R\mu_0(t))^{-1}[-(-\Delta)^{\frac{1}{2}}\eta_0]\left(\frac{y}{R}\right)+
\frac{1}{\pi}\frac{1}{\mu_0}\int_{\mathbb{R}}\frac{\left|\phi\left(s, t\right)\right|}{|y-s|^2}ds\\
&\lesssim \|\phi\|_{\alpha,\sigma}\frac{\mu_0(t)^{\sigma-1}R^{-1}}{1+|y/R|^{2}}\frac{1}{1+|y|^\alpha} + \|\phi\|_{\alpha,\sigma}\mu_0(t)^{\sigma-1}\int_{\mathbb{R}}\frac{1}{|y-s|^2|s|^\alpha}ds\\
&\lesssim \|\phi\|_{\alpha,\sigma}\frac{\mu_0(t)^{\sigma-1}}{1+|y|^{1+\alpha}}.
\end{aligned}
\end{equation*}

{\it Proof of (\ref{e:estimate4.7})}:
\begin{equation*}
\begin{aligned}
&\left|\frac{1}{\pi}\int_{\mathbb{R}}\frac{(\eta(x)-\eta(s))\left[\phi\left(\frac{s-\xi(t)}{\mu_0(t)}, t\right)\cdot \omega\left(\frac{s-\xi}{\mu}\right)\right]\left[\omega\left(\frac{x-\xi}{\mu}\right)-\omega\left(\frac{s-\xi}{\mu}\right)\right]}{|x-s|^2}ds\right|\\
&\lesssim\frac{1}{\pi}\int_{\mathbb{R}}\frac{\left|\eta(x)-\eta(s)\right|\left|\phi\left(\frac{s-\xi(t)}{\mu_0(t)}, t\right)\right|}{|x-s|^2}ds\\
&\lesssim\frac{1}{\pi}\int_{\mathbb{R}}\frac{\left|\phi\left(\frac{s-\xi(t)}{\mu_0(t)}, t\right)\right|}{|x-s|^2}ds = \frac{1}{\pi}\frac{1}{\mu_0}\int_{\mathbb{R}}\frac{\left|\phi\left(s, t\right)\right|}{|y-s|^2}ds\\
&\lesssim \|\phi\|_{\alpha,\sigma}\mu_0(t)^{\sigma-1}\int_{\mathbb{R}}\frac{1}{|y-s|^2|s|^\alpha}ds\\
&\lesssim \|\phi\|_{\alpha,\sigma}\frac{\mu_0(t)^{\sigma-1}}{1+|y|^{1+\alpha}}.
\end{aligned}
\end{equation*}

{\it Proof of (\ref{e:estimate4.777})}:
\begin{equation*}
\begin{aligned}
&\left|\left(\frac{1}{\pi}\int_{\mathbb{R}}\frac{(\eta(x)-\eta(s))\left[\phi\left(\frac{s-\xi(t)}{\mu_0(t)}, t\right)\cdot \omega \left(\frac{s-\xi}{\mu}\right)\right]\left[\omega\left(\frac{x-\xi}{\mu}\right)-\omega\left(\frac{s-\xi}{\mu}\right)\right]}{|x-s|^2}ds\cdot \omega\left(\frac{x-\xi}{\mu}\right)\right)\omega\left(\frac{x-\xi}{\mu}\right)\right|\\
&\lesssim\left|\frac{1}{\pi}\int_{\mathbb{R}}\frac{(\eta(x)-\eta(s))\left[\phi\left(\frac{s-\xi(t)}{\mu_0(t)}, t\right)\cdot \omega\left(\frac{s-\xi}{\mu}\right)\right]\left[\omega\left(\frac{x-\xi}{\mu}\right)-\omega\left(\frac{s-\xi}{\mu}\right)\right]}{|x-s|^2}ds\right|
\lesssim \|\phi\|_{\alpha,\sigma}\frac{\mu_0(t)^{\sigma-1}}{1+|y|^{1+\alpha}}.
\end{aligned}
\end{equation*}

{\it Proof of (\ref{e:estimate4.10})}:
From (\ref{e3:estimateoferror}), we have
\begin{equation*}
\left|(1-\eta)\Pi_{U^\perp}\mathcal{E}^*\right|\lesssim \frac{1}{1+|y|^{1+\alpha}} \lesssim \mu_0(t_0)^{1-\sigma}\frac{\mu_0(t)^{\sigma-1}}{1+|y|^{1+\alpha}}\lesssim\mu_0(t_0)^{\varepsilon}\frac{\mu_0(t)^{\sigma-1}}{1+|y|^{1+\alpha}}.
\end{equation*}

{\it Proof of (\ref{e:estimate4.9})}:
Using the decay of $\phi$ and $\psi$, by direct computation, we estimate the nonlinear term
\begin{equation*}
\begin{aligned}
&\left|N_U(\Pi_{U^\perp}[\Phi^0 + Z^*+\eta\phi+\psi])\right| = \\
&= \left|\left(\frac{1}{\pi}\int_{\mathbb{R}}\frac{(a(x)U(x)-a(s)U(s))\cdot(U(x)+\Pi_{U^\perp}\varphi(x)-U(s) -\Pi_{U^\perp}\varphi(s))}{|x-s|^2}ds\right.\right.\\ &\quad\quad\quad\quad\quad\quad\quad\left.+\frac{1}{\pi}\int_{\mathbb{R}}\frac{(U(x)-U(s))\cdot(\Pi_{U^\perp}\varphi(x) -\Pi_{U^\perp}\varphi(s))}{|x-s|^2}ds\right.\\ &\quad\quad\quad\quad\quad\quad\quad\left. +\frac{1}{2\pi}\int_{\mathbb{R}}\frac{(\Pi_{U^\perp}\varphi(x)-\Pi_{U^\perp}\varphi(s))\cdot(\Pi_{U^\perp}\varphi(x) -\Pi_{U^\perp}\varphi(s))}{|x-s|^2}ds\right.\\
&\quad\quad\quad\quad\quad\quad\quad +\left.\frac{1}{2\pi}\int_{\mathbb{R}}\frac{(a(x)U(x)-a(s)U(s))\cdot(a(x)U(x)-a(s)U(s))}{|x-s|^2}ds\right)\Pi_{U^\perp}\varphi\\
&\quad\quad\quad\quad\quad\quad\quad \left.-aU_t - \frac{1}{\pi}\int_{\mathbb{R}}\frac{(a(x)-a(s))\cdot(U(x) - U(s))}{|x-s|^2}ds\right|\lesssim \mu_0(t_0)^{\varepsilon}\frac{\mu_0^{-1}}{1+|y|^2}|\pi_{U^\perp}\varphi(y)|\\
&\lesssim \mu_0(t_0)^{\varepsilon}\frac{\mu_0^{-1}}{1+|y|^2}\left(\|\phi\|_{\alpha,\sigma}\frac{\mu_0(t)^{\sigma}}{1+|y|^{\alpha}} + \|\psi\|_{**}\frac{\mu_0^\sigma(t)}{1+|y|^\alpha}\right)
\lesssim \mu_0(t_0)^{\varepsilon}\left(\|\phi\|_{\alpha,\sigma} + \|\psi\|_{**}\right)\frac{\mu_0(t)^{\sigma-1}}{1+|y|^{\alpha+1}}.
\end{aligned}
\end{equation*}
Indeed, we have
\begin{equation*}
\begin{aligned}
&\frac{1}{\pi}\int_{\mathbb{R}}\frac{(a(x)U(x)-a(s)U(s))\cdot(U(x)+\Pi_{U^\perp}\varphi(x)-U(s) -\Pi_{U^\perp}\varphi(s))}{|x-s|^2}ds \\
&\quad = \frac{1}{\pi}\int_{\mathbb{R}}\frac{(a(x)U(x)-a(s)U(s))\cdot(U(x)-U(s))}{|x-s|^2}ds \\
&\quad\quad + \frac{1}{\pi}\int_{\mathbb{R}}\frac{(a(x)U(x)-a(s)U(s))\cdot(\Pi_{U^\perp}\varphi(x) -\Pi_{U^\perp}\varphi(s))}{|x-s|^2}ds\\
&\quad = a(x)\frac{1}{\pi}\int_{\mathbb{R}}\frac{(U(x)-U(s))\cdot(U(x)-U(s))}{|x-s|^2}ds \\
&\quad\quad + \frac{1}{2\pi}\int_{\mathbb{R}}\frac{(a(x)-a(s))|U(x)-U(s)|^2}{|x-s|^2}ds\\
&\quad\quad + a(x)\frac{1}{\pi}\int_{\mathbb{R}}\frac{(U(x)-U(s))\cdot(\Pi_{U^\perp}\varphi(x) -\Pi_{U^\perp}\varphi(s))}{|x-s|^2}ds\\
&\quad\quad + \frac{1}{\pi}\int_{\mathbb{R}}\frac{(a(x)-a(s))U(s)\cdot(\Pi_{U^\perp}\varphi(x) -\Pi_{U^\perp}\varphi(s))}{|x-s|^2}ds.
\end{aligned}
\end{equation*}
Since $a(x)$ is bounded,
\begin{equation*}
\left|a(x)\frac{1}{\pi}\int_{\mathbb{R}}\frac{(U(x)-U(s))\cdot(U(x)-U(s))}{|x-s|^2}ds\right|\lesssim \mu_0(t_0)^{\varepsilon}\frac{\mu_0^{-1}}{1+|y|^2},
\end{equation*}
\begin{equation*}
\begin{aligned}
&\left|\frac{1}{2\pi}\int_{\mathbb{R}}\frac{(a(x)-a(s))\left|\omega\left(\frac{x-\xi}{\mu}\right)-\omega\left(\frac{s-\xi}
{\mu}\right)\right|^2}{|x-s|^2}ds\right|\lesssim \mu_0(t_0)^{\varepsilon}\frac{1}{2\pi}\int_{\mathbb{R}}\frac{ |U(x)-U(s)|^2}{|x-s|^2}ds\lesssim \mu_0(t_0)^{\varepsilon}\frac{\mu_0^{-1}}{1+|y|^2},
\end{aligned}
\end{equation*}
\begin{equation*}
\begin{aligned}
&a(x)\frac{1}{\pi}\int_{\mathbb{R}}\frac{(U(x)-U(s))\cdot(\Pi_{U^\perp}\varphi(x) -\Pi_{U^\perp}\varphi(s))}{|x-s|^2}ds\\
&=a(x)\frac{1}{\pi}\int_{\mathbb{R}}\frac{\left(\omega\left(\frac{x-\xi}{\mu}\right)-\omega\left(\frac{s-\xi}{\mu}\right)\right)
\cdot\left(\widetilde{\Pi_{U^\perp}\varphi}\left(\frac{x-\xi}{\mu}\right) -\widetilde{\Pi_{U^\perp}\varphi}\left(\frac{s-\xi}{\mu}\right)\right)}{|x-s|^2}ds\\
&=a(x)\frac{1}{\pi}\frac{1}{\mu}\int_{\mathbb{R}}\frac{\left(\omega\left(y\right)-\omega\left(z'\right)\right)
\cdot\left(\widetilde{\Pi_{U^\perp}\varphi}\left(y\right) -\widetilde{\Pi_{U^\perp}\varphi}\left(z'\right)\right)}{|y-z'|^2}dz'\\
& = a(x)\frac{1}{\pi}\frac{1}{\mu}\int_{B(y,\epsilon)}\frac{\left(\omega\left(y\right)-\omega\left(z'\right)\right)
\cdot\left(\widetilde{\Pi_{U^\perp}\varphi}\left(y\right) -\widetilde{\Pi_{U^\perp}\varphi}\left(z'\right)\right)}{|y-z'|^2}dz' \\
&\quad + a(x)\frac{1}{\pi}\frac{1}{\mu}\int_{B(y,\epsilon)^c}\frac{\left(\omega\left(y\right)-\omega\left(z'\right)\right)
\cdot\left(\widetilde{\Pi_{U^\perp}\varphi}\left(y\right) -\widetilde{\Pi_{U^\perp}\varphi}\left(z'\right)\right)}{|y-z'|^2}dz'\\
&\lesssim \mu_0(t_0)^{\varepsilon}(\|\phi\|_{\alpha,\sigma}+\|\psi\|_{**})\frac{\mu_0^{\sigma-1}}{1+|y|^2}.
\end{aligned}
\end{equation*}
Here, for a function $f(x)$, we use the notation $f(x) = \widetilde{f}\left(\frac{x-\xi}{\mu}\right)$. By the local H\"{o}lder estimates of $\phi$ and $\psi$, we know that $\frac{\left(\widetilde{\Pi_{U^\perp}\varphi}\left(y\right) -\widetilde{\Pi_{U^\perp}\varphi}\left(z'\right)\right)}{|y-z'|^\eta}\leq (\|\phi\|_{\alpha,\sigma}+\|\psi\|_{**})\mu_0^{\sigma-1}$ for $z'\in B(y, \epsilon)$, $\eta > \frac{1}{2}$ and $\epsilon > 0$ is sufficiently small.
Similarly, we have
\begin{equation*}
\begin{aligned}
&\left|\frac{1}{\pi}\int_{\mathbb{R}}\frac{(a(x)-a(s))U(s)\cdot(\Pi_{U^\perp}\varphi(x) -\Pi_{U^\perp}\varphi(s))}{|x-s|^2}ds\right|\\
&\lesssim\mu_0(t_0)^{\varepsilon}\int_{\mathbb{R}}\frac{\left|(\Pi_{U^\perp}\varphi(x) -\Pi_{U^\perp}\varphi(s))\right|}{|x-s|^2}ds\\
&\lesssim\mu_0(t_0)^{\varepsilon}\int_{\mathbb{R}}\frac{\left|\widetilde{\Pi_{U^\perp}\varphi}\left(\frac{x-\xi}{\mu}\right) -\widetilde{\Pi_{U^\perp}\varphi}\left(\frac{s-\xi}{\mu}\right)\right|}{|x-s|^2}ds\\
& = \frac{\mu_0(t_0)^{\varepsilon}}{\mu}\int_{\mathbb{R}}\frac{\left|\widetilde{\Pi_{U^\perp}\varphi}\left(y\right) -\widetilde{\Pi_{U^\perp}\varphi}\left(z'\right)\right|}{|y-z'|^2}dz'\\
& = \frac{\mu_0(t_0)^{\varepsilon}}{\mu}\int_{B(y, \epsilon)}\frac{\left|\widetilde{\Pi_{U^\perp}\varphi}\left(y\right) -\widetilde{\Pi_{U^\perp}\varphi}\left(z'\right)\right|}{|y-z'|^2}dz' + \frac{\mu_0(t_0)^{\varepsilon}}{\mu}\int_{B^c(y, \epsilon)}\frac{\left|\widetilde{\Pi_{U^\perp}\varphi}\left(y\right) -\widetilde{\Pi_{U^\perp}\varphi}\left(z'\right)\right|}{|y-z'|^2}dz'\\
& \lesssim \mu_0(t_0)^{\varepsilon}(\|\phi\|_{\alpha,\sigma}+\|\psi\|_{**})\frac{\mu_0^{\sigma-1}}{1+|y|^2}.
\end{aligned}
\end{equation*}
Analogously, we have the same estimate for the other integral terms, we omit the details.

Now we apply the Contraction Mapping Theorem to prove the existence of a fixed point $\psi$ for the operator $\mathcal{A}$.
First, we define
\begin{equation*}\label{e4:54}
\mathcal{B}=\{\psi:\|\psi\|_{**}\leq Me^{-\varepsilon t_0}\}.
\end{equation*}
The positive large constant $M$ is independent of $t$ and $t_0$.
For any $\psi\in \mathcal{B}$, by the above estimates, $\mathcal{A}(\psi)\in \mathcal{B}$. Now we prove that for any $\psi_1$, $\psi_2\in\mathcal{B}$, there holds
\begin{equation*}\label{e4:55}
\|\mathcal{A}(\psi_1) - \mathcal{A}(\psi_2)\|_{**}\leq C\|\psi_1-\psi_2\|_{**},
\end{equation*}
where $C<1$ is a constant depending on $t_0$ , if $t_0$ is chosen sufficiently large.
We claim that
\begin{equation}\label{e4:101}
\begin{aligned}
&\left|(1-\eta)\frac{2\frac{1}{\mu}}{1+\left|\frac{x-\xi}{\mu}\right|^2}\left(\psi_1-\psi_2\right)\right|\lesssim \mu_0(t_0)^{\varepsilon}\|\psi_1-\psi_2\|_{**}\frac{\mu_0(t)^{\sigma-1}}{1+|y|^{1+\alpha}},
\end{aligned}
\end{equation}

\begin{equation}\label{e4:102}
\begin{aligned}
&\left|\frac{1}{\pi}\int_{\mathbb{R}}\frac{\left[\psi_1(x)\cdot \omega\left(\frac{x-\xi}{\mu}\right)-\psi_1(s)\cdot \omega\left(\frac{s-\xi}{\mu}\right)\right]\left[\omega\left(\frac{x-\xi}{\mu}\right)-\omega\left(\frac{s-\xi}{\mu}\right)\right]}{|x-s|^2}ds\right.\\ &\quad\quad\quad\quad\quad\quad\quad\quad\quad\quad\left.-
\frac{1}{\pi}\int_{\mathbb{R}}\frac{\left[\psi_2(x)\cdot \omega\left(\frac{x-\xi}{\mu}\right)-\psi_2(s)\cdot \omega\left(\frac{s-\xi}{\mu}\right)\right]\left[\omega\left(\frac{x-\xi}{\mu}\right)-\omega\left(\frac{s-\xi}{\mu}\right)\right]}{|x-s|^2}ds\right|\\
&\quad\quad\quad\quad\quad\quad\quad\quad\quad\quad\quad\quad\quad\quad\quad\quad\quad\quad\quad\quad\quad\quad\quad\quad
\lesssim \mu_0(t_0)^{\varepsilon}\|\psi_1-\psi_2\|_{**}\frac{\mu_0(t)^{\sigma-1}}{1+|y|^{1+\alpha}},
\end{aligned}
\end{equation}

\begin{equation}\label{e4:103}
\begin{aligned}
&\left|\left(\frac{1}{\pi}\int_{\mathbb{R}}\frac{\left[\psi_1(x)\cdot \omega\left(\frac{x-\xi}{\mu}\right)-\psi_1(s)\cdot \omega\left(\frac{s-\xi}{\mu}\right)\right]\left[\omega\left(\frac{x-\xi}{\mu}\right)-\omega\left(\frac{s-\xi}{\mu}\right)\right]}{|x-s|^2}ds\cdot \omega\left(\frac{x-\xi}{\mu}\right)\right)\omega\left(\frac{x-\xi}{\mu}\right)\right.\\
&\left.-\left(\frac{1}{\pi}\int_{\mathbb{R}}\frac{\left[\psi_2(x)\cdot \omega\left(\frac{x-\xi}{\mu}\right)-\psi_2(s)\cdot \omega\left(\frac{s-\xi}{\mu}\right)\right]\left[\omega\left(\frac{x-\xi}{\mu}\right)-\omega\left(\frac{s-\xi}{\mu}\right)\right]}{|x-s|^2}ds\cdot \omega\left(\frac{x-\xi}{\mu}\right)\right)\omega\left(\frac{x-\xi}{\mu}\right)\right|\\
&\quad\quad\quad\quad\quad\quad\quad\quad\quad\quad\quad\quad\quad\quad\quad\quad\quad\quad\quad\quad\quad\quad\quad\quad \lesssim \mu_0(t_0)^{\varepsilon}\|\psi_1-\psi_2\|_{**}\frac{\mu_0(t)^{\sigma-1}}{1+|y|^{1+\alpha}},
\end{aligned}
\end{equation}

\begin{equation}\label{e4:104}
\begin{aligned}
&\left|N_U(\Pi_{U^\perp}[\Phi^0 + Z^*+\eta\phi+\psi_1]) - N_U(\Pi_{U^\perp}[\Phi^0 + Z^*+\eta\phi+\psi_2])\right|\\
&\quad\quad\quad\quad\quad\quad\quad\quad\quad\quad\quad\quad\lesssim \mu_0(t_0)^{\varepsilon}(\|\psi_1\|_{**}+ \|\psi_2\|_{**} + \|\phi\|_{\alpha,\sigma})\|\psi_1-\psi_2\|_{**}\frac{\mu^{\sigma-1}_0(t)}{1+|y|^{1+\alpha}}.
\end{aligned}
\end{equation}

{\it Proof of (\ref{e4:101})}: We have
\begin{equation*}
\begin{aligned}
&\left|(1-\eta)\frac{2\frac{1}{\mu}}{1+\left|\frac{x-\xi}{\mu}\right|^2}\left(\psi_1-\psi_2\right)\right|\lesssim (1-\eta)\frac{2\frac{1}{\mu}}{1+y^2}\|\psi_1-\psi_2\|_{**}\frac{\mu_0^{\sigma}}{1+|y|^\alpha}\\
&\quad\lesssim (1-\eta)\frac{\log(1+|y|)}{1+|y|}\|\psi_1-\psi_2\|_{**}\frac{\mu_0(t)^{\sigma-1}}{1+|y|^{1+\alpha}}\\
&\quad\lesssim \mu_0(t_0)^{\varepsilon}\|\psi_1-\psi_2\|_{**}\frac{\mu_0(t)^{\sigma-1}}{1+|y|^{1+\alpha}}.
\end{aligned}
\end{equation*}

{\it Proof of (\ref{e4:102})}: We have
\begin{equation*}
\begin{aligned}
&\left|\frac{1}{\pi}\int_{\mathbb{R}}\frac{\left[\psi_1(x)\cdot \omega\left(\frac{x-\xi}{\mu}\right)-\psi_1(s)\cdot \omega\left(\frac{s-\xi}{\mu}\right)\right]\left[\omega\left(\frac{x-\xi}{\mu}\right)-\omega\left(\frac{s-\xi}{\mu}\right)\right]}{|x-s|^2}ds\right.\\ &\quad\quad\quad\quad\quad\quad\quad\quad\quad\quad\left.-
\frac{1}{\pi}\int_{\mathbb{R}}\frac{\left[\psi_2(x)\cdot \omega\left(\frac{x-\xi}{\mu}\right)-\psi_2(s)\cdot \omega\left(\frac{s-\xi}{\mu}\right)\right]\left[\omega\left(\frac{x-\xi}{\mu}\right)-\omega\left(\frac{s-\xi}{\mu}\right)\right]}{|x-s|^2}ds\right|\\
&\quad=\left|\frac{1}{\pi}\int_{\mathbb{R}}\frac{\left[(\psi_1(x)-\psi_2(x))\cdot \omega\left(\frac{x-\xi}{\mu}\right)-(\psi_1(s)-\psi_2(s))\cdot \omega\left(\frac{s-\xi}{\mu}\right)\right]\left[\omega\left(\frac{x-\xi}{\mu}\right)-\omega\left(\frac{s-\xi}{\mu}\right)\right]}{|x-s|^2}ds\right|\\
&\quad\lesssim \frac{1}{\pi}\int_{\mathbb{R}}\frac{\left|\omega\left(\frac{x-\xi}{\mu}\right)-\omega\left(\frac{s-\xi}
{\mu}\right)\right|^2}{|x-s|^2}ds\left|\psi_1(x)-\psi_2(x)\right|\\ &\quad\quad +\frac{1}{2\pi}\int_{\mathbb{R}}\frac{|(\psi_1(x)-\psi_2(x))-(\psi_1(s)-\psi_2(s))|\left|\omega
\left(\frac{x-\xi}{\mu}\right)-\omega\left(\frac{s-\xi}{\mu}\right)\right|^2}{|x-s|^2}ds\\
&\quad \lesssim \frac{\mu_0^{-1}}{1+|y|^2}\left|\psi_1(x)-\psi_2(x)\right|\\
&\quad\quad + \frac{1}{2\pi}\int_{\mathbb{R}}\frac{|(\psi_1(x)-\psi_2(x))-(\psi_1(s)-\psi_2(s))|\left|\omega\left(\frac{x-\xi}
{\mu}\right)-\omega\left(\frac{s-\xi}{\mu}\right)\right|^2}{|x-s|^2}ds\\
&\quad \lesssim \|\psi_1-\psi_2\|_{**}\mu_0^{\sigma-1}(t)\frac{\log(1+|y|)}{1+|y|^2}  \lesssim \|\psi_1-\psi_2\|_{**}\frac{\mu_0(t)^{\sigma-1}}{1+|y|^{\alpha+1}}\frac{\log(1+|y|)}{1+|y|^{1-\alpha}}\\
&\quad \lesssim \mu_0(t_0)^{\varepsilon}\|\psi_1-\psi_2\|_{**}\frac{\mu_0(t)^{\sigma-1}}{1+|y|^{1+\alpha}}.
\end{aligned}
\end{equation*}

{\it Proof of (\ref{e4:103})}: Similar to the proof of (\ref{e4:102}), we have
\begin{equation*}
\begin{aligned}
&\left|\left(\frac{1}{\pi}\int_{\mathbb{R}}\frac{\left[\psi_1(x)\cdot \omega\left(\frac{x-\xi}{\mu}\right)-\psi_1(s)\cdot \omega\left(\frac{s-\xi}{\mu}\right)\right]\left[\omega\left(\frac{x-\xi}{\mu}\right)-\omega\left(\frac{s-\xi}{\mu}\right)\right]}{|x-s|^2}ds\cdot \omega\left(\frac{x-\xi}{\mu}\right)\right)\omega\left(\frac{x-\xi}{\mu}\right)\right.\\
&\left.-\left(\frac{1}{\pi}\int_{\mathbb{R}}\frac{\left[\psi_2(x)\cdot \omega\left(\frac{x-\xi}{\mu}\right)-\psi_2(s)\cdot \omega\left(\frac{s-\xi}{\mu}\right)\right]\left[\omega\left(\frac{x-\xi}{\mu}\right)-\omega\left(\frac{s-\xi}{\mu}\right)\right]}{|x-s|^2}ds\cdot \omega\left(\frac{x-\xi}{\mu}\right)\right)\omega\left(\frac{x-\xi}{\mu}\right)\right|\\
\end{aligned}
\end{equation*}
\begin{equation*}
\begin{aligned}
&\quad\lesssim \mu_0(t_0)^{\varepsilon}\|\psi_1-\psi_2\|_{**}\frac{\mu_0(t)^{\sigma-1}}{1+|y|^{1+\alpha}}.
\end{aligned}
\end{equation*}

{\it Proof of (\ref{e4:104})}:
Using the decay of $\phi$ and $\psi$, we estimate the nonlinear term
\begin{equation*}
\begin{aligned}
&\left|N_U(\Pi_{U^\perp}[\Phi^0 + Z^*+\eta\phi+\psi_1]) - N_U(\Pi_{U^\perp}[Z^*+\phi + \eta\phi+\psi_2])\right| = \\
& \left|\left(\frac{1}{\pi}\int_{\mathbb{R}}\frac{(a(\Pi_{U^\perp}\varphi_1)(x)U(x)-a(\Pi_{U^\perp}\varphi_1)(s)U(s))
\cdot(U(x)+\Pi_{U^\perp}\varphi_1(x)-U(s) -\Pi_{U^\perp}\varphi_1(s))}{|x-s|^2}ds\right.\right.\\ &\left.+\frac{1}{\pi}\int_{\mathbb{R}}\frac{(U(x)-U(s))\cdot(\Pi_{U^\perp}\varphi_1(x) -\Pi_{U^\perp}\varphi_1(s))}{|x-s|^2}ds\right.\\ &\left. +\frac{1}{2\pi}\int_{\mathbb{R}}\frac{(\Pi_{U^\perp}\varphi_1(x)-\Pi_{U^\perp}\varphi_1(s))\cdot(\Pi_{U^\perp}\varphi_1(x) -\Pi_{U^\perp}\varphi_1(s))}{|x-s|^2}ds\right.\\
& +\left.\frac{1}{2\pi}\int_{\mathbb{R}}\frac{(a(\Pi_{U^\perp}\varphi_1)(x)U(x)-a(\Pi_{U^\perp}\varphi_1)(s)U(s))
\cdot(a(\Pi_{U^\perp}\varphi_1)(x)U(x)-a(\Pi_{U^\perp}\varphi_1)(s)U(s))}{|x-s|^2}ds\right)\Pi_{U^\perp}\varphi_1\\
& -a(\Pi_{U^\perp}\varphi_1)U_t - \frac{1}{\pi}\int_{\mathbb{R}}\frac{(a(\Pi_{U^\perp}\varphi_1)(x)-a(\Pi_{U^\perp}\varphi_1)(s))\cdot(U(x) - U(s))}{|x-s|^2}ds\\
& -
\left(\frac{1}{\pi}\int_{\mathbb{R}}\frac{(a(\Pi_{U^\perp}\varphi_2)(x)U(x)-a(\Pi_{U^\perp}\varphi_2)(s)U(s))\cdot(U(x)+\Pi_{U^\perp}\varphi_2(x)-U(s) -\Pi_{U^\perp}\varphi_2(s))}{|x-s|^2}ds\right.\\ &\left.+\frac{1}{\pi}\int_{\mathbb{R}}\frac{(U(x)-U(s))\cdot(\Pi_{U^\perp}\varphi_2(x) -\Pi_{U^\perp}\varphi_2(s))}{|x-s|^2}ds\right.\\ &\left. +\frac{1}{2\pi}\int_{\mathbb{R}}\frac{(\Pi_{U^\perp}\varphi_2(x)-\Pi_{U^\perp}\varphi_2(s))\cdot(\Pi_{U^\perp}\varphi_2(x) -\Pi_{U^\perp}\varphi_2(s))}{|x-s|^2}ds\right.\\
& +\left.\frac{1}{2\pi}\int_{\mathbb{R}}\frac{(a(\Pi_{U^\perp}\varphi_2)(x)U(x)-a(\Pi_{U^\perp}\varphi_2)(s)U(s))
\cdot(a(\Pi_{U^\perp}\varphi_2)(x)U(x)-a(\Pi_{U^\perp}\varphi_2)(s)U(s))}{|x-s|^2}ds\right)\Pi_{U^\perp}\varphi_2\\
& \left.+a(\Pi_{U^\perp}\varphi_2)U_t + \frac{1}{\pi}\int_{\mathbb{R}}\frac{(a(\Pi_{U^\perp}\varphi_2)(x)-a(\Pi_{U^\perp}\varphi_2)(s))\cdot(U(x) - U(s))}{|x-s|^2}ds
\right|,
\end{aligned}
\end{equation*}
with $\varphi_1 = \eta\phi + \psi_1$ and $\varphi_2 = \eta\phi + \psi_2$.
A typical term is
\begin{equation*}
\begin{aligned} &\left(\frac{1}{\pi}\int_{\mathbb{R}}\frac{(a(\Pi_{U^\perp}\varphi_1)(x)U(x)-a(\Pi_{U^\perp}\varphi_1)(s)U(s))\cdot(U(x)+\Pi_{U^\perp}\varphi_1(x)-U(s) -\Pi_{U^\perp}\varphi_1(s))}{|x-s|^2}ds\right)\Pi_{U^\perp}\varphi_1\\
& -
\left(\frac{1}{\pi}\int_{\mathbb{R}}\frac{(a(\Pi_{U^\perp}\varphi_2)(x)U(x)-a(\Pi_{U^\perp}\varphi_2)(s)U(s))\cdot(U(x)+\Pi_{U^\perp}\varphi_2(x)-U(s) -\Pi_{U^\perp}\varphi_2(s))}{|x-s|^2}ds\right)\Pi_{U^\perp}\varphi_2.
\end{aligned}
\end{equation*}
We compute it as
\begin{equation*}
\begin{aligned} &\left(\frac{1}{\pi}\int_{\mathbb{R}}\frac{(a(\Pi_{U^\perp}\varphi_1)(x)U(x)-a(\Pi_{U^\perp}\varphi_1)(s)U(s))\cdot(U(x)+\Pi_{U^\perp}\varphi_1(x)-U(s) -\Pi_{U^\perp}\varphi_1(s))}{|x-s|^2}ds\right)\Pi_{U^\perp}\varphi_1\\
& -
\left(\frac{1}{\pi}\int_{\mathbb{R}}\frac{(a(\Pi_{U^\perp}\varphi_2)(x)U(x)-a(\Pi_{U^\perp}\varphi_2)(s)U(s))\cdot(U(x)+\Pi_{U^\perp}\varphi_2(x)-U(s) -\Pi_{U^\perp}\varphi_2(s))}{|x-s|^2}ds\right)\Pi_{U^\perp}\varphi_2\\
& =  \left(\frac{1}{\pi}\int_{\mathbb{R}}\frac{(a(\Pi_{U^\perp}\varphi_1)(x)U(x)-a(\Pi_{U^\perp}\varphi_1)(s)U(s))\cdot(U(x)+\Pi_{U^\perp}\varphi_1(x)-U(s) -\Pi_{U^\perp}\varphi_1(s))}{|x-s|^2}ds\right.\\
& -
\left.\frac{1}{\pi}\int_{\mathbb{R}}\frac{(a(\Pi_{U^\perp}\varphi_2)(x)U(x)-a(\Pi_{U^\perp}\varphi_2)(s)U(s))\cdot(U(x)+\Pi_{U^\perp}\varphi_2(x)-U(s) -\Pi_{U^\perp}\varphi_2(s))}{|x-s|^2}ds\right)\Pi_{U^\perp}\varphi_1\\
& +
\left.\frac{1}{\pi}\int_{\mathbb{R}}\frac{(a(\Pi_{U^\perp}\varphi_2)(x)U(x)-a(\Pi_{U^\perp}\varphi_2)(s)U(s))\cdot(U(x)+\Pi_{U^\perp}\varphi_2(x)-U(s) -\Pi_{U^\perp}\varphi_2(s))}{|x-s|^2}ds\right)\left(\Pi_{U^\perp}\varphi_1 - \Pi_{U^\perp}\varphi_2\right).
\end{aligned}
\end{equation*}
The term
\begin{equation*}
\begin{aligned}
&\left(\frac{1}{\pi}\int_{\mathbb{R}}\frac{(a(\Pi_{U^\perp}\varphi_1)(x)U(x)-a(\Pi_{U^\perp}\varphi_1)(s)U(s))\cdot(U(x)+\Pi_{U^\perp}\varphi_1(x)-U(s) -\Pi_{U^\perp}\varphi_1(s))}{|x-s|^2}ds\right.\\
& -
\left.\frac{1}{\pi}\int_{\mathbb{R}}\frac{(a(\Pi_{U^\perp}\varphi_2)(x)U(x)-a(\Pi_{U^\perp}\varphi_2)(s)U(s))\cdot(U(x)+\Pi_{U^\perp}\varphi_2(x)-U(s) -\Pi_{U^\perp}\varphi_2(s))}{|x-s|^2}ds\right)\Pi_{U^\perp}\varphi_1\\
& = \left(\frac{1}{\pi}\int_{\mathbb{R}}\frac{(a(\Pi_{U^\perp}\varphi_1)(x)U(x)-a(\Pi_{U^\perp}\varphi_1)(s)U(s))\cdot(U(x)+\Pi_{U^\perp}\varphi_1(x)-U(s) -\Pi_{U^\perp}\varphi_1(s))}{|x-s|^2}ds\right.\\
& -
\frac{1}{\pi}\int_{\mathbb{R}}\frac{(a(\Pi_{U^\perp}\varphi_2)(x)U(x)-a(\Pi_{U^\perp}\varphi_2)(s)U(s))\cdot(U(x)+\Pi_{U^\perp}\varphi_1(x)-U(s) -\Pi_{U^\perp}\varphi_1(s))}{|x-s|^2}ds\\
& +
\frac{1}{\pi}\int_{\mathbb{R}}\frac{(a(\Pi_{U^\perp}\varphi_2)(x)U(x)-a(\Pi_{U^\perp}\varphi_2)(s)U(s))\cdot(U(x)+\Pi_{U^\perp}\varphi_1(x)-U(s) -\Pi_{U^\perp}\varphi_1(s))}{|x-s|^2}ds\\
& -
\left.\frac{1}{\pi}\int_{\mathbb{R}}\frac{(a(\Pi_{U^\perp}\varphi_2)(x)U(x)-a(\Pi_{U^\perp}\varphi_2)(s)U(s))
\cdot(U(x)+\Pi_{U^\perp}\varphi_2(x)-U(s) -\Pi_{U^\perp}\varphi_2(s))}{|x-s|^2}ds\right)\Pi_{U^\perp}\varphi_1.
\end{aligned}
\end{equation*}
Furthermore,
\begin{equation*}
\begin{aligned}
& \left|\left(\frac{1}{\pi}\int_{\mathbb{R}}\frac{(a(\Pi_{U^\perp}\varphi_1)(x)U(x)-a(\Pi_{U^\perp}\varphi_1)(s)U(s))
\cdot(U(x)+\Pi_{U^\perp}\varphi_1(x)-U(s) -\Pi_{U^\perp}\varphi_1(s))}{|x-s|^2}ds\right.\right.\\
& -
\left.\left.\frac{1}{\pi}\int_{\mathbb{R}}\frac{(a(\Pi_{U^\perp}\varphi_2)(x)U(x)-a(\Pi_{U^\perp}\varphi_2)(s)U(s))
\cdot(U(x)+\Pi_{U^\perp}\varphi_1(x)-U(s) -\Pi_{U^\perp}\varphi_1(s))}{|x-s|^2}ds\right)\Pi_{U^\perp}\varphi_1\right|\\
&\lesssim
\left|\left(\frac{1}{\pi}\int_{\mathbb{R}}\frac{(a(\Pi_{U^\perp}\varphi_2)(s)U(s)-a(\Pi_{U^\perp}\varphi_1)(s)U(s))
\cdot(U(x)+\Pi_{U^\perp}\varphi_1(x)-U(s) -\Pi_{U^\perp}\varphi_1(s))}{|x-s|^2}ds\right.\right.\\
& -
\left.\left.\frac{1}{\pi}\int_{\mathbb{R}}\frac{U(x)+\Pi_{U^\perp}\varphi_1(x)-U(s) -\Pi_{U^\perp}\varphi_1(s)}{|x-s|^2}ds(a(\Pi_{U^\perp}\varphi_2)(x)U(x)-a(\Pi_{U^\perp}\varphi_1)(x)U(x))\cdot\right)\Pi_{U^\perp}\varphi_1\right|\\
&\lesssim
\left|\left(\frac{1}{\pi}\int_{\mathbb{R}}\frac{(a(\Pi_{U^\perp}\varphi_2)(s)-a(\Pi_{U^\perp}\varphi_1)(s))U(s)
\cdot(U(x)+\Pi_{U^\perp}\varphi_1(x)-U(s) -\Pi_{U^\perp}\varphi_1(s))}{|x-s|^2}ds\right.\right.\\
& -
\left.\left.\frac{1}{\pi}\int_{\mathbb{R}}\frac{U(x)+\Pi_{U^\perp}\varphi_1(x)-U(s) -\Pi_{U^\perp}\varphi_1(s)}{|x-s|^2}ds(a(\Pi_{U^\perp}\varphi_2)(x)-a(\Pi_{U^\perp}\varphi_1)(x))U(x)\cdot\right)\Pi_{U^\perp}\varphi_1\right|\\
&\lesssim \mu_0(t_0)^{\varepsilon}(\|\psi_1\|_{**}+ \|\psi_2\|_{**} + \|\phi\|_{\alpha,\sigma})\|\psi_1-\psi_2\|_{**}\frac{\mu^{\sigma-1}_0(t)}{1+|y|^{1+\alpha}}
\end{aligned}
\end{equation*}
and
\begin{equation*}
\begin{aligned}
&\left|\left(
\frac{1}{\pi}\int_{\mathbb{R}}\frac{(a(\Pi_{U^\perp}\varphi_2)(x)U(x)-a(\Pi_{U^\perp}\varphi_2)(s)U(s))\cdot(U(x)+\Pi_{U^\perp}\varphi_1(x)-U(s) -\Pi_{U^\perp}\varphi_1(s))}{|x-s|^2}ds\right.\right.\\
& -
\left.\left.\frac{1}{\pi}\int_{\mathbb{R}}\frac{(a(\Pi_{U^\perp}\varphi_2)(x)U(x)-a(\Pi_{U^\perp}\varphi_2)(s)U(s))
\cdot(U(x)+\Pi_{U^\perp}\varphi_2(x)-U(s) -\Pi_{U^\perp}\varphi_2(s))}{|x-s|^2}ds\right)\right|\Pi_{U^\perp}\varphi_1\\
& = \left|\left(
\frac{1}{\pi}\int_{\mathbb{R}}\frac{(a(\Pi_{U^\perp}\varphi_2)(x)U(x)-a(\Pi_{U^\perp}\varphi_2)(s)U(s))
\cdot(\Pi_{U^\perp}\varphi_1(x)-\Pi_{U^\perp}\varphi_2(x) -\Pi_{U^\perp}\varphi_1(s)+ \Pi_{U^\perp}\varphi_2(s))}{|x-s|^2}ds\right)\right|\Pi_{U^\perp}\varphi_1\\
& \leq  \left|\left(
\frac{1}{\pi}\int_{\mathbb{R}}\frac{(a(\Pi_{U^\perp}\varphi_2)(x)U(x)-a(\Pi_{U^\perp}\varphi_2)(s)U(s))
\cdot(\Pi_{U^\perp}\varphi_1(x)-\Pi_{U^\perp}\varphi_2(x) )}{|x-s|^2}ds\right)\right|\Pi_{U^\perp}\varphi_1\\
&+ \left|\left(
\frac{1}{\pi}\int_{\mathbb{R}}\frac{(a(\Pi_{U^\perp}\varphi_2)(x)U(x)-a(\Pi_{U^\perp}\varphi_2)(s)U(s))
\cdot(\Pi_{U^\perp}\varphi_1(s)-\Pi_{U^\perp}\varphi_2(s) )}{|x-s|^2}ds\right)\right|\Pi_{U^\perp}\varphi_1\\
\end{aligned}
\end{equation*}
\begin{equation*}
\begin{aligned}
&\lesssim \mu_0(t_0)^{\varepsilon}(\|\psi_1\|_{**}+ \|\psi_2\|_{**} + \|\phi\|_{\alpha,\sigma})\|\psi_1-\psi_2\|_{**}\frac{\mu^{\sigma-1}_0(t)}{1+|y|^{1+\alpha}}.
\end{aligned}
\end{equation*}
Analogously, we have the same estimate for the other integral terms and finally get
\begin{equation*}
\begin{aligned}
&\left|N_U(\Pi_{U^\perp}[\Phi^0 + Z^*+\eta\phi+\psi_1]) - N_U(\Pi_{U^\perp}[Z^*+\phi + \eta\phi+\psi_2])\right|\\
&\quad\quad\lesssim \mu_0(t_0)^{\varepsilon}(\|\psi_1\|_{**}+ \|\psi_2\|_{**} + \|\phi\|_{\alpha,\sigma})\|\psi_1-\psi_2\|_{**}\frac{\mu^{\sigma-1}_0(t)}{1+|y|^{1+\alpha}}.
\end{aligned}
\end{equation*}

Hence
\begin{equation*}\label{e4:62}
\|\mathcal{A}(\psi_1) - \mathcal{A}(\psi_2)\|_{**}\leq C\|\psi_1-\psi_2\|_{**}
\end{equation*}
holds with $C < 1$ when $t_0$ is sufficiently large. Therefore, if $t_0$ is sufficiently large, the operator $\mathcal{A}$ is a contraction map in $\mathcal{B}$. By the Contraction Mapping Theorem, we get the existence part of Proposition \ref{p4:4.1}. (\ref{e:estimateholder}) follows directly from (\ref{e:holderestimateforheatequation}). This completes the proof.
\end{proof}
\begin{prop}\label{p4:4.2}
Under the assumptions in Proposition \ref{p4:4.1}, $\Psi$ depends smoothly on the parameters $\lambda$, $\xi$, $\dot{\lambda}$, $\dot{\xi}$, $\phi$. For $y=\frac{x-\xi}{\mu_{0}}$, we have the following estimates
\begin{equation}\label{e4:64}
\big|\partial_\lambda\Psi[\lambda,\xi,\dot{\lambda},\dot{\xi},\phi][\bar{\lambda}](x, t)\big|\lesssim e^{-\varepsilon t_0}\|\bar{\lambda}\|_{1+\sigma}\frac{\mu^{\sigma}(t)\log(1+|y|)}{1+|y|^{\alpha}},
\end{equation}
\begin{equation}\label{e4:74}
\big|\partial_\xi\Psi[\lambda,\xi,\dot{\lambda},\dot{\xi},\phi][\bar{\xi}](x, t)\big|\lesssim e^{-\varepsilon t_0}\|\bar{\xi}\|_{1+\sigma}\frac{\mu^{\sigma}(t)\log(1+|y|)}{1+|y|^{\alpha}},
\end{equation}
\begin{equation}\label{e4:75}
\big|\partial_{\dot{\xi}}\Psi[\lambda,\xi,\dot{\lambda},\dot{\xi},\phi][\dot{\bar{\xi}}](x, t)\big|\lesssim e^{-\varepsilon t_0}\|\dot{\bar{\xi}}\|_{1+\sigma}\frac{\mu^{\sigma}(t)\log(1+|y|)}{1+|y|^{\alpha}},
\end{equation}
\begin{equation}\label{e4:76}
\big|\partial_{\dot{\lambda}}\Psi[\lambda,\xi,\dot{\lambda},\dot{\xi},\phi][\dot{\bar{\lambda}}](x, t)\big|\lesssim e^{-\varepsilon t_0}\|\dot{\bar{\lambda}}\|_{1+\sigma}\frac{\mu^{\sigma}(t)\log(1+|y|)}{1+|y|^{\alpha}},
\end{equation}
\begin{equation}\label{e4:84}
\big|\partial_{\phi}\Psi[\lambda,\xi,\dot{\lambda},\dot{\xi},\phi][\bar{\phi}](x, t)\big|\lesssim e^{-\varepsilon t_0}\|\bar{\phi}(t)\|_{\alpha, \sigma}\frac{\mu^{\sigma}(t)\log(1+|y|)}{1+|y|^{\alpha}}.
\end{equation}
\end{prop}
\begin{proof}
{\bf Step 1. Proof of (\ref{e4:64}) and (\ref{e4:74})}:
From Proposition \ref{p4:4.1}, the function $\Psi[\lambda]$ is a solution to (\ref{e:outerproblem})  for all $\lambda$ satisfying (\ref{e:assumptionsonlambda}). Differentiating (\ref{e:outerproblem}) with respect to $\lambda$, we obtain a nonlinear equation, from the Implicit Function Theorem, the solutions are given by $\partial_{\lambda}\Psi[\bar{\lambda}](x, t)$. Denoting $Z:=\partial_{\lambda}\Psi[\bar{\lambda}](x, t)$, then $Z$ is a solution of the following nonlinear problem
\begin{equation}\label{e:equation4}
\begin{aligned}
\partial_t \psi &= (-\Delta)^{\frac{1}{2}}\psi + g(x, t)\quad \text{ in }\mathbb{R}\times [t_0, +\infty)
\end{aligned}
\end{equation}
with
\begin{equation*}
\begin{aligned}
&g(x, t)= (1-\eta) \frac{\partial}{\partial\bar{\lambda}}\left\{\frac{2\frac{1}{\mu}}{1+\left|\frac{x-\xi}{\mu}\right|^2}\psi -\frac{1}{\pi}\int_{\mathbb{R}}\frac{\left[\psi(x)\cdot \omega\left(\frac{x-\xi}{\mu}\right)-\psi(s)\cdot \omega\left(\frac{s-\xi}{\mu}\right)\right]\left[\omega\left(\frac{x-\xi}{\mu}\right)-\omega\left(\frac{s-\xi}{\mu}\right)\right]}{|x-s|^2}ds
\right. \\
&+ \left(\frac{1}{\pi}\int_{\mathbb{R}}\frac{\left[\psi(x)\cdot \omega\left(\frac{x-\xi}{\mu}\right)-\psi(s)\cdot \omega\left(\frac{s-\xi}{\mu}\right)\right]\left[\omega\left(\frac{x-\xi}{\mu}\right)-\omega\left(\frac{s-\xi}{\mu}\right)\right]}{|x-s|^2}ds\cdot \omega\left(\frac{x-\xi}{\mu}\right)\right)\omega\left(\frac{x-\xi}{\mu}\right)\\
&  + N_U(\Pi_{U^\perp}[\Phi^0 + Z^*+\eta\phi+\psi]) + (1-\eta)\Pi_{U^\perp}\mathcal{E}^*\\
&  -\frac{1}{\pi}\int_{\mathbb{R}}\frac{(\eta(x)-\eta(s))\left[\phi\left(\frac{s-\xi(t)}{\mu_0(t)}, t\right)\cdot \omega\left(\frac{s-\xi}{\mu}\right)\right]\left[\omega\left(\frac{x-\xi}{\mu}\right)-\omega\left(\frac{s-\xi}{\mu}\right)\right]}{|x-s|^2}ds+\\
\end{aligned}
\end{equation*}
\begin{equation*}
\begin{aligned}
&  \left.\left(\frac{1}{\pi}\int_{\mathbb{R}}\frac{(\eta(x)-\eta(s))\left[\phi\left(\frac{s-\xi(t)}{\mu_0(t)}, t\right)\cdot \omega \left(\frac{s-\xi}{\mu}\right)\right]\left[\omega\left(\frac{x-\xi}{\mu}\right)-\omega\left(\frac{s-\xi}{\mu}\right)\right]}{|x-s|^2}ds\cdot \omega\left(\frac{x-\xi}{\mu}\right)\right)\omega\left(\frac{x-\xi}{\mu}\right)\right\}[\bar{\lambda}].
\end{aligned}
\end{equation*}
If we define $\|Z\|_{*}$ as the least $M > 0$ such that
\begin{equation*}
|Z(x, t)|\lesssim M\|\bar{\lambda}\|_{1+\sigma}\frac{\mu_0^{\sigma}(t)\log(1+|y|)}{1+|y|^{\alpha}},
\end{equation*}
then we have the following estimates
\begin{equation}\label{e4:120}
\begin{aligned}
&\left|\frac{\partial}{\partial\bar{\lambda}}\left(\frac{2\frac{1}{\mu}}{1+\left|\frac{x-\xi}{\mu}\right|^2}\psi\right)[\bar{\lambda}]\right|\lesssim e^{-\varepsilon t_0}\left(1+\|Z\|_*\right)\frac{\mu_0^{\sigma-1}}{1+|y|^{1+\alpha}}\|\bar{\lambda}\|_{1+\sigma},
\end{aligned}
\end{equation}

\begin{equation}\label{e4:121}
\begin{aligned}
&\left|\left(\Pi_{U^\perp}\mathcal{E}^*\right)[\bar{\lambda}]\right|\lesssim e^{-\varepsilon t_0}\frac{\mu_0^{\sigma-1}(t)}{1+|y|^{1+\alpha}}\|\bar{\lambda}\|_{1+\sigma},
\end{aligned}
\end{equation}

\begin{equation}\label{e4:122}
\begin{aligned}
&\left|\frac{\partial}{\partial\bar{\lambda}}\left\{-\frac{1}{\pi}\int_{\mathbb{R}}\frac{\left[\psi(x)\cdot \omega\left(\frac{x-\xi}{\mu}\right)-\psi(s)\cdot \omega\left(\frac{s-\xi}{\mu}\right)\right]\left[\omega\left(\frac{x-\xi}{\mu}\right)-\omega\left(\frac{s-\xi}{\mu}\right)\right]}
{|x-s|^2}ds\right\}[\bar{\lambda}]\right|\\
&\quad\quad\quad\quad\quad\quad\quad\quad\quad\quad\quad\quad\quad\quad\quad\quad\quad\quad\quad\quad\lesssim e^{-\varepsilon t_0}(1+\|Z\|_*)\frac{\mu_0^{\sigma-1}(t)}{1+|y|^{1+\alpha}}\|\bar{\lambda}\|_{1+\sigma},
\end{aligned}
\end{equation}

\begin{equation}\label{e4:123}
\begin{aligned}
&\left|\frac{\partial}{\partial\bar{\lambda}}\left\{\left(\frac{1}{\pi}\int_{\mathbb{R}}\frac{\left[\psi(x)\cdot \omega\left(\frac{x-\xi}{\mu}\right)-\psi(s)\cdot \omega\left(\frac{s-\xi}{\mu}\right)\right]\left[\omega\left(\frac{x-\xi}{\mu}\right)-\omega\left(\frac{s-\xi}{\mu}\right)\right]}{|x-s|^2}ds\cdot \omega\left(\frac{x-\xi}{\mu}\right)\right)\omega\left(\frac{x-\xi}{\mu}\right)\right\}[\bar{\lambda}]\right|\\
&\quad\quad\quad\quad\quad\quad\quad\quad\quad\quad\quad\quad\quad\quad\quad\quad\quad\quad\quad\quad\quad\quad\quad\lesssim e^{-\varepsilon t_0}(1+\|Z\|_*)\frac{\mu_0^{\sigma-1}(t)}{1+|y|^{1+\alpha}}\|\bar{\lambda}\|_{1+\sigma},
\end{aligned}
\end{equation}

\begin{equation}\label{e4:125}
\begin{aligned}
&\left|\frac{\partial}{\partial\bar{\lambda}}N_U(\Pi_{U^\perp}[\Phi^0 + Z^*+\eta\phi+\psi])\right|\lesssim e^{-\varepsilon t_0}\frac{\mu_0^{\sigma-1}(t)}{1+|y|^{1+\alpha}}\|\bar{\lambda}\|_{1+\sigma} + \|Z\|_{*}e^{-\varepsilon t_0}\frac{\mu_0^{\sigma-1}(t)}{1+|y|^{1+\alpha}}\|\bar{\lambda}\|_{1+\sigma},
\end{aligned}
\end{equation}

\begin{equation}\label{e4:126}
\left|\frac{\partial}{\partial\bar{\lambda}}\left\{\frac{1}{\pi}\int_{\mathbb{R}}\frac{(\eta(x)-\eta(s))\left[\phi\left(\frac{s-\xi(t)}{\mu_0(t)}, t\right)\cdot \omega\left(\frac{s-\xi}{\mu}\right)\right]\left[\omega\left(\frac{x-\xi}{\mu}\right)-\omega\left(\frac{s-\xi}{\mu}\right)\right]}{|x-s|^2}ds\right\}
\right|\lesssim e^{-\varepsilon t_0}\|\bar{\lambda}\|_{1+\sigma}\frac{\mu_0(t)^{\sigma-1}}{1+|y|^{1+\alpha}},
\end{equation}

\begin{equation}\label{e4:127}
\begin{aligned}
&\left|\frac{\partial}{\partial\bar{\lambda}}\left\{\left(\frac{1}{\pi}\int_{\mathbb{R}}\frac{(\eta(x)-\eta(s))\left[\phi\left(\frac{s-\xi(t)}
{\mu_0(t)}, t\right)\cdot \omega \left(\frac{s-\xi}{\mu}\right)\right]\left[\omega\left(\frac{x-\xi}{\mu}\right)-\omega\left(\frac{s-\xi}{\mu}\right)\right]}{|x-s|^2}ds\cdot \omega\left(\frac{x-\xi}{\mu}\right)\right)\omega\left(\frac{x-\xi}{\mu}\right)\right\}
\right|\\
&\quad\quad\quad\quad\quad\quad\quad\quad\quad\quad\quad\quad\quad\quad\quad\quad\quad\quad\quad\quad\quad\quad\quad\quad\quad\quad\quad\quad\quad\quad \lesssim e^{-\varepsilon t_0}\|\bar{\lambda}\|_{1+\sigma}\frac{\mu_0(t)^{\sigma-1}}{1+|y|^{1+\alpha}}.
\end{aligned}
\end{equation}

{\it Proof of (\ref{e4:120})}: We have
\begin{equation*}
\begin{aligned}
&\left|\frac{\partial}{\partial\bar{\lambda}}\left(\frac{2\frac{1}{\mu}}{1+\left|\frac{x-\xi}{\mu}\right|^2}\psi\right)[\bar{\lambda}]\right|\\
&\lesssim \frac{\frac{1}{\mu}}{1+|y|^2}\left|\psi\frac{\bar{\lambda}}{\mu} + Z\right|\\
&\lesssim \frac{\frac{1}{\mu}}{1+|y|^2}\frac{\bar{\lambda}}{\mu}\|\psi\|_{**}\frac{\mu_0^{\sigma}\log(1+|y|)}{1+|y|^\alpha} + \frac{\frac{1}{\mu}}{1+|y|^2}\|Z\|_{*}\frac{\mu_0^{\sigma}\log(1+|y|)}{1+|y|^\alpha}\|\bar{\lambda}\|_{1+\sigma}\\
&\lesssim e^{-\varepsilon t_0}\left(1+\|Z\|_*\right)\frac{\mu_0^{\sigma-1}}{1+|y|^{1+\alpha}}\|\bar{\lambda}\|_{1+\sigma}.
\end{aligned}
\end{equation*}

{\it Proof of (\ref{e4:121})}: From Section 2.4, we have
\begin{equation*}
\begin{aligned}
&\left|\left(\Pi_{U^\perp}\mathcal{E}^*\right)[\bar{\lambda}]\right|\lesssim \frac{1}{1+|y|^{1+\alpha}}\left|\frac{\bar{\lambda}}{\mu}\right|\lesssim e^{-\varepsilon t_0}\frac{\mu_0^{\sigma-1}}{1+|y|^{1+\alpha}}\|\bar{\lambda}\|_{1+\sigma}.
\end{aligned}
\end{equation*}

{\it Proof of (\ref{e4:122})}: We have
\begin{equation*}
\begin{aligned}
&\left|\frac{\partial}{\partial\bar{\lambda}}\left\{-\frac{1}{\pi}\int_{\mathbb{R}}\frac{\left[\psi(x)\cdot \omega\left(\frac{x-\xi}{\mu}\right)-\psi(s)\cdot \omega\left(\frac{s-\xi}{\mu}\right)\right]\left[\omega\left(\frac{x-\xi}{\mu}\right)-\omega\left(\frac{s-\xi}{\mu}\right)\right]}{|x-s|^2}ds\right\}
[\bar{\lambda}]\right|\\
& = \frac{\bar{\lambda}}{\mu}\left|\left\{-\frac{1}{\pi}\int_{\mathbb{R}}\frac{\left[\psi(x)\cdot \left(\nabla\omega\left(\frac{x-\xi}{\mu}\right)\cdot\left(\frac{x-\xi}{\mu}\right)\right) -\psi(s)\cdot \left(\nabla\omega\left(\frac{s-\xi}{\mu}\right)\cdot\left(\frac{s-\xi}{\mu}\right)\right)\right]\left[\omega\left(\frac{x-\xi}{\mu}\right)
-\omega\left(\frac{s-\xi}{\mu}\right)\right]}{|x-s|^2}ds\right\}\right| \\ & \quad +
\frac{\bar{\lambda}}{\mu}\left|\left\{-\frac{1}{\pi}\int_{\mathbb{R}}\frac{\left[\psi(x)\cdot \omega\left(\frac{x-\xi}{\mu}\right)-\psi(s)\cdot \omega\left(\frac{s-\xi}{\mu}\right)\right]\left[\left(\nabla\omega\left(\frac{x-\xi}{\mu}\right)
\cdot\left(\frac{x-\xi}{\mu}\right)\right)-\left(\nabla\omega\left(\frac{x-\xi}{\mu}\right)\cdot\left(\frac{x-\xi}{\mu}\right)\right)\right]}
{|x-s|^2}ds\right\}\right| \\
&\quad + \left|\frac{1}{\pi}\int_{\mathbb{R}}\frac{\left[Z(x)\cdot \omega\left(\frac{x-\xi}{\mu}\right)-Z(s)\cdot \omega\left(\frac{s-\xi}{\mu}\right)\right]\left[\omega\left(\frac{x-\xi}{\mu}\right)-\omega\left(\frac{s-\xi}{\mu}\right)\right]}{|x-s|^2}ds\right|\\
&\lesssim \left|\frac{\bar{\lambda}}{\mu}\right|\frac{\|\psi\|_{**}\mu_0^{\sigma}(t)}{1+|y|^2} + \frac{\|Z\|_{*}\mu_0^{\sigma}(t)}{1+|y|^2}\|\bar{\lambda}\|_{1+\sigma}\lesssim e^{-\varepsilon t_0}(1+\|Z\|_*)\frac{\mu_0^{\sigma-1}(t)}{1+|y|^{1+\alpha}}\|\bar{\lambda}\|_{1+\sigma}.
\end{aligned}
\end{equation*}

{\it Proof of (\ref{e4:123})}: We have
\begin{equation*}
\begin{aligned}
&\left|\frac{\partial}{\partial\bar{\lambda}}\left\{\left(\frac{1}{\pi}\int_{\mathbb{R}}\frac{\left[\psi(x)\cdot \omega\left(\frac{x-\xi}{\mu}\right)-\psi(s)\cdot \omega\left(\frac{s-\xi}{\mu}\right)\right]\left[\omega\left(\frac{x-\xi}{\mu}\right)-\omega\left(\frac{s-\xi}{\mu}\right)\right]}{|x-s|^2}ds\cdot \omega\left(\frac{x-\xi}{\mu}\right)\right)\omega\left(\frac{x-\xi}{\mu}\right)\right\}[\bar{\lambda}]\right|\\
& = \frac{\bar{\lambda}}{\mu}\left|\left\{-\frac{1}{\pi}\int_{\mathbb{R}}\frac{\left[\psi(x)\cdot \left(\nabla\omega\left(\frac{x-\xi}{\mu}\right)\cdot\left(\frac{x-\xi}{\mu}\right)\right) -\psi(s)\cdot \left(\nabla\omega\left(\frac{s-\xi}{\mu}\right)\cdot\left(\frac{s-\xi}{\mu}\right)\right)\right]\left[\omega\left(\frac{x-\xi}{\mu}\right)
-\omega\left(\frac{s-\xi}{\mu}\right)\right]}{|x-s|^2}ds\right\}\right| \\
& \quad +
\frac{\bar{\lambda}}{\mu}\left|\left\{-\frac{1}{\pi}\int_{\mathbb{R}}\frac{\left[\psi(x)\cdot \omega\left(\frac{x-\xi}{\mu}\right)-\psi(s)\cdot \omega\left(\frac{s-\xi}{\mu}\right)\right]\left[\left(\nabla\omega\left(\frac{x-\xi}{\mu}\right)
\cdot\left(\frac{x-\xi}{\mu}\right)\right)-\left(\nabla\omega\left(\frac{x-\xi}{\mu}\right)\cdot\left(\frac{x-\xi}{\mu}\right)\right)\right]}
{|x-s|^2}ds\right\}\right|
\\
\end{aligned}
\end{equation*}
\begin{equation*}
\begin{aligned}
& +
\frac{\bar{\lambda}}{\mu}\left|\left\{\left(\frac{1}{\pi}\int_{\mathbb{R}}\frac{\left[\psi(x)\cdot \omega\left(\frac{x-\xi}{\mu}\right)-\psi(s)\cdot \omega\left(\frac{s-\xi}{\mu}\right)\right]\left[\omega\left(\frac{x-\xi}{\mu}\right)-\omega\left(\frac{s-\xi}{\mu}\right)\right]}{|x-s|^2}ds\cdot \left(\nabla\omega\left(\frac{x-\xi}{\mu}\right)\cdot\left(\frac{x-\xi}{\mu}\right)\right)\right)\omega\left(\frac{x-\xi}{\mu}\right)\right\}\right|\\
\\ & +
\frac{\bar{\lambda}}{\mu}\left|\left\{\left(\frac{1}{\pi}\int_{\mathbb{R}}\frac{\left[\psi(x)\cdot \omega\left(\frac{x-\xi}{\mu}\right)-\psi(s)\cdot \omega\left(\frac{s-\xi}{\mu}\right)\right]\left[\omega\left(\frac{x-\xi}{\mu}\right)-\omega\left(\frac{s-\xi}{\mu}\right)\right]}{|x-s|^2}ds\cdot \omega\left(\frac{x-\xi}{\mu}\right)\right)\left(\nabla\omega\left(\frac{x-\xi}{\mu}\right)\cdot\left(\frac{x-\xi}{\mu}\right)\right)\right\}\right|\\
\\ & + \left|\left(\frac{1}{\pi}\int_{\mathbb{R}}\frac{\left[Z(x)\cdot \omega\left(\frac{x-\xi}{\mu}\right)-Z(s)\cdot \omega\left(\frac{s-\xi}{\mu}\right)\right]\left[\omega\left(\frac{x-\xi}{\mu}\right)-\omega\left(\frac{s-\xi}{\mu}\right)\right]}{|x-s|^2}ds\cdot \omega\left(\frac{x-\xi}{\mu}\right)\right)\omega\left(\frac{x-\xi}{\mu}\right)\right|\\
&\lesssim \left|\frac{\bar{\lambda}}{\mu}\right|\frac{\|\psi\|_{**}\mu_0^{\sigma}(t)}{1+|y|^2} + \frac{\|Z\|_{*}\mu_0^{\sigma}(t)}{1+|y|^2}\|\bar{\lambda}\|_{1+\sigma}\lesssim e^{-\varepsilon t_0}(1+\|Z\|_*)\frac{\mu_0^{\sigma-1}(t)}{1+|y|^{1+\alpha}}\|\bar{\lambda}\|_{1+\sigma}.
\end{aligned}
\end{equation*}

{\it Proof of (\ref{e4:125})}: We have
\begin{equation*}
\begin{aligned}
&\left|\frac{\partial}{\partial\bar{\lambda}}\left\{N_U(\Pi_{U^\perp}[Z^*+\varphi])\right\}\right|\\
&= \left|\frac{\partial}{\partial\bar{\lambda}}\left\{\left(\frac{1}{\pi}\int_{\mathbb{R}}\frac{(a(x)U(x)-a(s)U(s))\cdot(U(x)+\Pi_{U^\perp}\varphi(x)-U(s) -\Pi_{U^\perp}\varphi(s))}{|x-s|^2}ds\right.\right.\right.\\ &\quad\quad\quad\quad\quad\quad\quad\left.+\frac{1}{\pi}\int_{\mathbb{R}}\frac{(U(x)-U(s))\cdot(\Pi_{U^\perp}\varphi(x) -\Pi_{U^\perp}\varphi(s))}{|x-s|^2}ds\right.\\ &\quad\quad\quad\quad\quad\quad\quad\left. +\frac{1}{2\pi}\int_{\mathbb{R}}\frac{(\Pi_{U^\perp}\varphi(x)-\Pi_{U^\perp}\varphi(s))\cdot(\Pi_{U^\perp}\varphi(x) -\Pi_{U^\perp}\varphi(s))}{|x-s|^2}ds\right.\\
&\quad\quad\quad\quad\quad\quad\quad +\left.\frac{1}{2\pi}\int_{\mathbb{R}}\frac{(a(x)U(x)-a(s)U(s))\cdot(a(x)U(x)-a(s)U(s))}{|x-s|^2}ds\right)\Pi_{U^\perp}\varphi\\
&\quad\quad\quad\quad\quad\quad\quad \left.\left.-aU_t - \frac{1}{\pi}\int_{\mathbb{R}}\frac{(a(x)-a(s))\cdot(U(x) - U(s))}{|x-s|^2}ds\right\}\right|.
\end{aligned}
\end{equation*}
A typical term is
\begin{equation*}
\begin{aligned}
&\left|\frac{\partial}{\partial\bar{\lambda}}\left\{\left(a(x)\frac{1}{\pi}\int_{\mathbb{R}}\frac{(U(x)-U(s))\cdot(U(x)-U(s))}
{|x-s|^2}ds\right)\Pi_{U^\perp}\varphi\right\}\right|\\
& = \left|\frac{\partial}{\partial\bar{\lambda}}\left\{\left(a(x)\frac{\mu^{-1}}{1+|y|^2}\right)\Pi_{U^\perp}\varphi\right\}\right|\\
& \lesssim \left|\left(\frac{\partial}{\partial\bar{\lambda}}a(x)\right)\frac{\mu^{-1}}{1+|y|^2}\Pi_{U^\perp}\varphi\right|\\
\end{aligned}
\end{equation*}
\begin{equation*}
\begin{aligned}
& \quad + \left|a(x)\left(\frac{\partial}{\partial\bar{\lambda}}\frac{\mu^{-1}}{1+|y|^2}\right)\Pi_{U^\perp}\varphi\right| + \left|a(x)\frac{\mu^{-1}}{1+|y|^2}\left(\frac{\partial}{\partial\bar{\lambda}}\Pi_{U^\perp}\varphi\right)\right|\\
&\lesssim e^{-\varepsilon t_0}\frac{\mu_0^{\sigma-1}(t)}{1+|y|^{1+\alpha}}\|\bar{\lambda}\|_{1+\sigma} + \|Z\|_{*}e^{-\varepsilon t_0}\frac{\mu_0^{\sigma-1}(t)}{1+|y|^{1+\alpha}}\|\bar{\lambda}\|_{1+\sigma}.
\end{aligned}
\end{equation*}
The estimates of the other integrals are similar.

{\it Proof of (\ref{e4:126})}:
\begin{equation*}
\begin{aligned}
&\left|\frac{\partial}{\partial\bar{\lambda}}\left\{\frac{1}{\pi}\int_{\mathbb{R}}\frac{(\eta(x)-\eta(s))\left[\phi\left(\frac{s-\xi(t)}{\mu_0(t)}, t\right)\cdot \omega\left(\frac{s-\xi}{\mu}\right)\right]\left[\omega\left(\frac{x-\xi}{\mu}\right)-\omega\left(\frac{s-\xi}{\mu}\right)\right]}{|x-s|^2}ds\right\}
\right|\\
&\lesssim \left|\frac{\bar{\lambda}}{\mu}\right|\left|\int_{\mathbb{R}}\frac{(\eta(x)-\eta(s))\left[\phi\left(\frac{s-\xi(t)}{\mu_0(t)}, t\right)\cdot \left(\nabla\omega\left(\frac{s-\xi}{\mu}\right)\cdot\left(\frac{s-\xi}{\mu}\right)\right)\right]
\left[\omega\left(\frac{x-\xi}{\mu}\right)-\omega\left(\frac{s-\xi}{\mu}\right)\right]}{|x-s|^2}ds\right|\\
&\quad + \left|\frac{\bar{\lambda}}{\mu}\right|\left|\int_{\mathbb{R}}\frac{(\eta(x)-\eta(s))\left[\phi\left(\frac{s-\xi(t)}{\mu_0(t)}, t\right)\cdot \omega\left(\frac{s-\xi}{\mu}\right)\right]
\left[\left(\nabla\omega\left(\frac{x-\xi}{\mu}\right)\cdot\left(\frac{x-\xi}{\mu}\right)\right)-
\left(\nabla\omega\left(\frac{s-\xi}{\mu}\right)\cdot\left(\frac{s-\xi}{\mu}\right)\right)\right]}{|x-s|^2}ds\right|\\
& \lesssim e^{-\varepsilon t_0}\|\phi\|_{\alpha,\sigma}\|\bar{\lambda}\|_{1+\sigma}\frac{\mu_0(t)^{\sigma-1}}{1+|y|^2}\lesssim e^{-\varepsilon t_0}\|\bar{\lambda}\|_{1+\sigma}\frac{\mu_0(t)^{\sigma-1}}{1+|y|^{1+\alpha}}.
\end{aligned}
\end{equation*}

The proof of (\ref{e4:127}) is similar.

Now we consider the fixed point problem corresponding to (\ref{e:equation4}). By similar arguments as Proposition \ref{p4:4.1}, one can show that the operator $\mathcal{A}_1$
\begin{equation*}
\mathcal{A}_1(Z):=T(g(Z))
\end{equation*}
has a fixed point in the set of functions satisfying
$$|Z(x, t)|\leq M e^{-\varepsilon t_0}\frac{\mu_0^{\sigma}(t)\log(1+|y|)}{1+|y|^{\alpha}}\|\bar{\lambda}\|_{1+\sigma}$$
when $M$ is a fixed large constant. Therefore, the estimate (\ref{e4:64}) for $\partial_{\lambda_1}\Psi[\bar{\lambda}_1]$ holds.
The proof of (\ref{e4:74}) is similar.

{\bf Step 2. Proof of (\ref{e4:75}) and (\ref{e4:76})}:
Differentiating (\ref{e:outerproblem}) with respect to $\dot{\xi}$, we obtain a nonlinear equation, from the Implicit Function Theorem, the solutions are given by $V:=\partial_{\dot{\xi}}\Psi[\dot{\bar{\xi}}](x, t)$ and it is a solution of
\begin{equation*}
\begin{aligned}
\partial_t \psi &= (-\Delta)^{\frac{1}{2}}\psi + g(x, t)\quad \text{ in }\mathbb{R}\times [t_0, +\infty)
\end{aligned}
\end{equation*}
with
\begin{equation*}
\begin{aligned}
&g(x, t)= \frac{\partial}{\partial\dot{\xi}}\left\{(1-\eta)\left[\frac{2\frac{1}{\mu}}{1+\left|\frac{x-\xi}{\mu}\right|^2}\psi -\frac{1}{\pi}\int_{\mathbb{R}}\frac{\left[\psi(x)\cdot \omega\left(\frac{x-\xi}{\mu}\right)-\psi(s)\cdot \omega\left(\frac{s-\xi}{\mu}\right)\right]\left[\omega\left(\frac{x-\xi}{\mu}\right)-\omega\left(\frac{s-\xi}{\mu}\right)\right]}{|x-s|^2}ds\right.
\right. \\
& \left.\left.+ \left(\frac{1}{\pi}\int_{\mathbb{R}}\frac{\left[\psi(x)\cdot \omega\left(\frac{x-\xi}{\mu}\right)-\psi(s)\cdot \omega\left(\frac{s-\xi}{\mu}\right)\right]\left[\omega\left(\frac{x-\xi}{\mu}\right)-\omega\left(\frac{s-\xi}{\mu}\right)\right]}{|x-s|^2}ds\cdot \omega\left(\frac{x-\xi}{\mu}\right)\right)\omega\left(\frac{x-\xi}{\mu}\right)\right]\right\}[\dot{\bar{\xi}}]\\
&+\eta\frac{\dot{\bar{\xi}}}{\mu_0}\cdot\nabla_y\phi .
\end{aligned}
\end{equation*}
Define $\|V\|_{***}$ as the least $M > 0$ such that
\begin{equation*}
|V(x, t)|\lesssim M\|\dot{\bar{\xi}}\|_{1+\sigma}\frac{\mu_0^{\sigma}(t)\log(1+|y|)}{1+|y|^{\alpha}}.
\end{equation*}
Since
\begin{equation*}
\begin{aligned}
\left|\frac{\partial}{\partial\dot{\xi}}\left((1-\eta)\frac{2\frac{1}{\mu}}{1+\left|\frac{x-\xi}{\mu}\right|^2}\psi\right)[\dot{\bar{\xi}}]\right|
&\lesssim \|V\|_{***}\|\dot{\bar{\xi}}\|_{1+\sigma}\frac{\mu_0^{\sigma-1}(t)\log(1+|y|)}{1+|y|^{\alpha+2}}\\
&\lesssim e^{-\varepsilon t_0}\|V\|_{***}\|\dot{\bar{\xi}}\|_{1+\sigma}\frac{\mu_0^{\sigma-1}(t)}{1+|y|^{1+\alpha}},
\end{aligned}
\end{equation*}
\begin{equation*}
\begin{aligned}
&\left|\frac{\partial}{\partial\dot{\xi}}\left((1-\eta)\frac{1}{\pi}\int_{\mathbb{R}}\frac{\left[\psi(x)\cdot \omega\left(\frac{x-\xi}{\mu}\right)-\psi(s)\cdot \omega\left(\frac{s-\xi}{\mu}\right)\right]\left[\omega\left(\frac{x-\xi}{\mu}\right)-\omega\left(\frac{s-\xi}{\mu}\right)\right]}{|x-s|^2}ds\right)
[\dot{\bar{\xi}}]\right|
\\ &\quad\quad\quad\quad\quad\quad\quad\quad\quad\lesssim \|V\|_{***}\|\dot{\bar{\xi}}\|_{1+\sigma}\mu_0^{\sigma-1}(t)\left(\frac{\log(1+|y|)}{1+|y|^{\alpha+2}}+\frac{1}{1+|y|^2}\right)\\
&\quad\quad\quad\quad\quad\quad\quad\quad\quad\lesssim e^{-\varepsilon t_0}\|V\|_{***}\|\dot{\bar{\xi}}\|_{1+\sigma}\frac{\mu_0^{\sigma-1}(t)}{1+|y|^{1+\alpha}},
\end{aligned}
\end{equation*}
\begin{equation*}
\begin{aligned}
&\left|\frac{\partial}{\partial\dot{\xi}}\left((1-\eta)\left(\frac{1}{\pi}\int_{\mathbb{R}}\frac{\left[\psi(x)\cdot \omega\left(\frac{x-\xi}{\mu}\right)-\psi(s)\cdot \omega\left(\frac{s-\xi}{\mu}\right)\right]\left[\omega\left(\frac{x-\xi}{\mu}\right)-\omega\left(\frac{s-\xi}{\mu}\right)\right]}{|x-s|^2}ds\cdot \omega\left(\frac{x-\xi}{\mu}\right)\right)\omega\left(\frac{x-\xi}{\mu}\right)\right)
[\dot{\bar{\xi}}]\right|
\\ &\quad\quad\quad\quad\quad\quad\quad\quad\quad\lesssim \|V\|_{***}\|\dot{\bar{\xi}}\|_{1+\sigma}\mu_0^{\sigma-1}(t)\left(\frac{\log(1+|y|)}{1+|y|^{\alpha+2}}+\frac{1}{1+|y|^2}\right)\\
&\quad\quad\quad\quad\quad\quad\quad\quad\quad\lesssim e^{-\varepsilon t_0}\|V\|_{***}\|\dot{\bar{\xi}}\|_{1+\sigma}\frac{\mu_0^{\sigma-1}(t)}{1+|y|^{1+\alpha}},
\end{aligned}
\end{equation*}
\begin{equation*}
\begin{aligned}
&\left|\eta\frac{\dot{\bar{\xi}}}{\mu_0}\cdot\nabla_y\phi\right|\lesssim e^{-\varepsilon t_0}\|\dot{\bar{\xi}}\|_{1+\sigma}\frac{\|\phi\|_{\alpha,\sigma}\mu^{\sigma}_0(t)}{1+|y|^{1+\alpha}}\lesssim e^{-\varepsilon t_0}\|\dot{\bar{\xi}}\|_{1+\sigma}\frac{\mu^{\sigma-1}_0(t)}{1+|y|^{1+\alpha}},
\end{aligned}
\end{equation*}
the same reason as Step 1 gives us the desired estimate. The proof of (\ref{e4:76}) is similar.

{\bf Step 3. Proof of (\ref{e4:84})}: Differentiating (\ref{e:outerproblem}) with respect to $\phi$, we know $W:=\partial_{\phi}\Psi[\bar{\phi}](x, t)$ satisfies
\begin{equation*}
\begin{aligned}
\partial_t \psi &= (-\Delta)^{\frac{1}{2}}\psi + g(x, t)\quad \text{ in }\mathbb{R}\times [t_0, +\infty)
\end{aligned}
\end{equation*}
with
\begin{equation*}
\begin{aligned}
&g(x, t)= \frac{\partial}{\partial\phi}\left\{(1-\eta)\left[\frac{2\frac{1}{\mu}}{1+\left|\frac{x-\xi}{\mu}\right|^2}\psi -\frac{1}{\pi}\int_{\mathbb{R}}\frac{\left[\psi(x)\cdot \omega\left(\frac{x-\xi}{\mu}\right)-\psi(s)\cdot \omega\left(\frac{s-\xi}{\mu}\right)\right]\left[\omega\left(\frac{x-\xi}{\mu}\right)-\omega\left(\frac{s-\xi}{\mu}\right)\right]}{|x-s|^2}ds\right.
\right. \\
& \left.+ \left(\frac{1}{\pi}\int_{\mathbb{R}}\frac{\left[\psi(x)\cdot \omega\left(\frac{x-\xi}{\mu}\right)-\psi(s)\cdot \omega\left(\frac{s-\xi}{\mu}\right)\right]\left[\omega\left(\frac{x-\xi}{\mu}\right)-\omega\left(\frac{s-\xi}{\mu}\right)\right]}{|x-s|^2}ds\cdot \omega\left(\frac{x-\xi}{\mu}\right)\right)\omega\left(\frac{x-\xi}{\mu}\right)\right]\\
& + \frac{\dot{\mu}_0}{\mu_0}\eta y\cdot \nabla_y\phi + \eta\frac{\dot{\xi}}{\mu_0}\cdot\nabla_y\phi + N_U(\Pi_{U^\perp}[\Phi^0 + Z^*+\eta\phi+\psi])\\
&  -\left[(-\Delta)^{\frac{1}{2}}\eta\right] \phi - \partial_t\eta\phi\left(\frac{x-\xi(t)}{\mu_0(t)}\right)+
\frac{1}{\pi}\int_{\mathbb{R}}\frac{(\eta(x)-\eta(s))\left(\phi\left(\frac{x-\xi(t)}{\mu_0(t)}, t\right)-\phi\left(\frac{s-\xi(t)}{\mu_0(t)}, t\right)\right)}{|x-s|^2}ds\\
&  -\frac{1}{\pi}\int_{\mathbb{R}}\frac{(\eta(x)-\eta(s))\left[\phi\left(\frac{s-\xi(t)}{\mu_0(t)}, t\right)\cdot \omega\left(\frac{s-\xi}{\mu}\right)\right]\left[\omega\left(\frac{x-\xi}{\mu}\right)-\omega\left(\frac{s-\xi}{\mu}\right)\right]}{|x-s|^2}ds+\\
\end{aligned}
\end{equation*}
\begin{equation*}
\begin{aligned}
&  \left.\left(\frac{1}{\pi}\int_{\mathbb{R}}\frac{(\eta(x)-\eta(s))\left[\phi\left(\frac{s-\xi(t)}{\mu_0(t)}, t\right)\cdot \omega \left(\frac{s-\xi}{\mu}\right)\right]\left[\omega\left(\frac{x-\xi}{\mu}\right)-\omega\left(\frac{s-\xi}{\mu}\right)\right]}{|x-s|^2}ds\cdot \omega\left(\frac{x-\xi}{\mu}\right)\right)\omega\left(\frac{x-\xi}{\mu}\right)\right\}[\bar{\phi}].
\end{aligned}
\end{equation*}
If we define $\|W\|_{****}$ as the least $M > 0$ such that
\begin{equation*}
|W(x, t)|\lesssim M\|\bar{\phi}(t)\|_{\alpha, \sigma}\frac{\mu_0^{\sigma}(t)\log(1+|y|)}{1+|y|^{\alpha}},
\end{equation*}
then we have
\begin{equation*}
\begin{aligned}
\left|\frac{\partial}{\partial\phi}\left((1-\eta)\frac{2\frac{1}{\mu}}{1+\left|\frac{x-\xi}{\mu}\right|^2}\psi\right)[\bar{\phi}]\right|&\lesssim \|W\|_{****}\|\bar{\phi}(t)\|_{\alpha, \sigma}\frac{\mu_0^{\sigma-1}(t)\log(1+|y|)}{1+|y|^{\alpha+2}}\\
&\lesssim e^{-\varepsilon t_0}\|W\|_{****} \|\bar{\phi}\|_{\alpha,\sigma}\frac{\mu_0^{\sigma-1}(t)}{1+|y|^{1+\alpha}},
\end{aligned}
\end{equation*}
\begin{equation*}
\begin{aligned}
&\left|\frac{\partial}{\partial\phi}\left(\frac{1-\eta}{\pi}\int_{\mathbb{R}}\frac{\left[\psi(x)\cdot \omega\left(\frac{x-\xi}{\mu}\right)-\psi(s)\cdot \omega\left(\frac{s-\xi}{\mu}\right)\right]\left[\omega\left(\frac{x-\xi}{\mu}\right)-\omega\left(\frac{s-\xi}{\mu}\right)\right]}{|x-s|^2}ds\right)
[\bar{\phi}]\right|\\ &\quad\quad\quad\quad\quad\quad\quad\quad\quad\lesssim \|W\|_{****}\|\bar{\phi}(t)\|_{\alpha, \sigma}\mu_0^{\sigma-1}(t)\left(\frac{\log(1+|y|)}{1+|y|^{\alpha+2}}+\frac{1}{1+|y|^2}\right)\\
&\quad\quad\quad\quad\quad\quad\quad\quad\quad\lesssim e^{-\varepsilon t_0}\|W\|_{****} \|\bar{\phi}\|_{\alpha,\sigma}\frac{\mu_0^{\sigma-1}(t)}{1+|y|^{1+\alpha}},
\end{aligned}
\end{equation*}
\begin{equation*}
\begin{aligned}
&\left|\frac{\partial}{\partial\phi}\left((1-\eta)\left(\frac{1}{\pi}\int_{\mathbb{R}}\frac{\left[\psi(x)\cdot \omega\left(\frac{x-\xi}{\mu}\right)-\psi(s)\cdot \omega\left(\frac{s-\xi}{\mu}\right)\right]\left[\omega\left(\frac{x-\xi}{\mu}\right)-\omega\left(\frac{s-\xi}{\mu}\right)\right]}{|x-s|^2}ds\cdot \omega\left(\frac{x-\xi}{\mu}\right)\right)\omega\left(\frac{x-\xi}{\mu}\right)\right)
[\bar{\phi}]\right|\\ &\quad\quad\quad\quad\quad\quad\quad\quad\quad\lesssim \|W\|_{****}\|\bar{\phi}(t)\|_{\alpha, \sigma}\mu_0^{\sigma-1}(t)\left(\frac{\log(1+|y|)}{1+|y|^{\alpha+2}}+\frac{1}{1+|y|^2}\right)\\
&\quad\quad\quad\quad\quad\quad\quad\quad\quad\lesssim e^{-\varepsilon t_0}\|W\|_{****} \|\bar{\phi}\|_{\alpha,\sigma}\frac{\mu_0^{\sigma-1}(t)}{1+|y|^{1+\alpha}},
\end{aligned}
\end{equation*}
\begin{equation*}
\begin{aligned}
&\left|\frac{\partial}{\partial\phi}\left(\frac{\dot{\mu}_0}{\mu_0}\eta y\cdot \nabla_y\phi\right)[\bar{\phi}]\right|\lesssim \|\bar{\phi}\|_{\alpha,\sigma}\frac{\mu_0^{\sigma}(t)}{1+|y|^{\alpha}}\lesssim e^{-\varepsilon t_0} \|\bar{\phi}\|_{\alpha,\sigma}\frac{\mu_0^{\sigma-1}(t)}{1+|y|^{1+\alpha}},
\end{aligned}
\end{equation*}
\begin{equation*}
\begin{aligned}
&\left|\frac{\partial}{\partial\phi}\left(\eta\frac{\dot{\xi}}{\mu_0}\cdot\nabla_y\phi\right)[\bar{\phi}]\right|\lesssim e^{-\varepsilon t_0} \|\bar{\phi}\|_{\alpha,\sigma}\frac{\mu_0^{\sigma-1}(t)}{1+|y|^{1+\alpha}},
\end{aligned}
\end{equation*}
\begin{equation*}
\begin{aligned}
&\left|\frac{\partial}{\partial\phi}\left(\left[(-\Delta)^{\frac{1}{2}}\eta\right] \phi\right)[\bar{\phi}]\right|\lesssim \|\bar{\phi}\|_{\alpha,\sigma}\frac{\mu_0^{\sigma}(t)}{1+|y|^{2+\alpha}}\lesssim e^{-\varepsilon t_0} \|\bar{\phi}\|_{\alpha,\sigma}\frac{\mu_0^{\sigma-1}(t)}{1+|y|^{1+\alpha}},
\end{aligned}
\end{equation*}
\begin{equation*}
\begin{aligned}
&\left|\frac{\partial}{\partial\phi}\left(\partial_t\eta\phi\left(\frac{x-\xi(t)}{\mu_0(t)}\right)\right)[\bar{\phi}]\right|\lesssim \|\bar{\phi}\|_{\alpha,\sigma}\frac{\mu_0^{\sigma}(t)}{1+|y|^{\alpha}}\lesssim e^{-\varepsilon t_0} \|\bar{\phi}\|_{\alpha,\sigma}\frac{\mu_0^{\sigma-1}(t)}{1+|y|^{1+\alpha}},
\end{aligned}
\end{equation*}
\begin{equation*}
\begin{aligned}
&\left|\frac{\partial}{\partial\phi}\left(\frac{1}{\pi}\int_{\mathbb{R}}\frac{(\eta(x)-\eta(s))\left(\phi\left(\frac{x-\xi(t)}{\mu_0(t)}, t\right)-\phi\left(\frac{s-\xi(t)}{\mu_0(t)}, t\right)\right)}{|x-s|^2}ds\right)[\bar{\phi}]\right|\lesssim \|\bar{\phi}\|_{\alpha,\sigma}\frac{\mu_0^{\sigma}(t)}{1+|y|^{2}}\\
\end{aligned}
\end{equation*}
\begin{equation*}
\begin{aligned}
&\quad\quad\quad\quad\quad\quad\quad\quad\quad\quad\quad\quad\quad\quad\quad\quad\quad\quad\quad\quad\quad\quad\quad\quad\quad\quad\quad
\lesssim e^{-\varepsilon t_0} \|\bar{\phi}\|_{\alpha,\sigma}\frac{\mu_0^{\sigma-1}(t)}{1+|y|^{1+\alpha}},
\end{aligned}
\end{equation*}
\begin{equation*}
\begin{aligned}
&\left|\frac{\partial}{\partial\phi}\left(\frac{1}{\pi}\int_{\mathbb{R}}\frac{(\eta(x)-\eta(s))\left[\phi\left(\frac{s-\xi(t)}{\mu_0(t)}, t\right)\cdot \omega\left(\frac{s-\xi}{\mu}\right)\right]\left[\omega\left(\frac{x-\xi}{\mu}\right)-\omega\left(\frac{s-\xi}{\mu}\right)\right]}{|x-s|^2}ds\right)
[\bar{\phi}]\right|\\ &\quad\quad\quad\quad\quad\quad\quad\quad\quad\quad\quad\quad\quad\quad\quad\quad\quad\quad\quad\quad\quad\quad \lesssim \|\bar{\phi}\|_{\alpha,\sigma}\frac{\mu_0^{\sigma}(t)}{1+|y|^{2}}
\lesssim e^{-\varepsilon t_0} \|\bar{\phi}\|_{\alpha,\sigma}\frac{\mu_0^{\sigma-1}(t)}{1+|y|^{1+\alpha}},
\end{aligned}
\end{equation*}
\begin{equation*}
\begin{aligned}
&\left|\frac{\partial}{\partial\phi}\left(\left(\frac{1}{\pi}\int_{\mathbb{R}}\frac{(\eta(x)-\eta(s))\left[\phi\left(\frac{s-\xi(t)}{\mu_0(t)}, t\right)\cdot \omega \left(\frac{s-\xi}{\mu}\right)\right]\left[\omega\left(\frac{x-\xi}{\mu}\right)-\omega\left(\frac{s-\xi}{\mu}\right)\right]}{|x-s|^2}ds\cdot \omega\left(\frac{x-\xi}{\mu}\right)\right)\omega\left(\frac{x-\xi}{\mu}\right)\right)
[\bar{\phi}]\right|\\ &\quad\quad\quad\quad\quad\quad\quad\quad\quad\quad\quad\quad\quad\quad\quad\quad\quad\quad\quad\quad\quad\quad \lesssim \|\bar{\phi}\|_{\alpha,\sigma}\frac{\mu_0^{\sigma}(t)}{1+|y|^{2}}
\lesssim e^{-\varepsilon t_0} \|\bar{\phi}\|_{\alpha,\sigma}\frac{\mu_0^{\sigma-1}(t)}{1+|y|^{1+\alpha}}
\end{aligned}
\end{equation*}
and
\begin{equation*}
\begin{aligned}
\left|\frac{\partial}{\partial\phi}\left(N_U(\Pi_{U^\perp}[\Phi^0 + Z^*+\eta\phi+\psi])\right)[\bar{\phi}]\right|& \lesssim \mu_0(t_0)^{\varepsilon}\frac{\mu_0^{-1}}{1+|y|^2}\left(|\bar{\phi}(y)|+|W|\right)\\ &\lesssim e^{-\varepsilon t_0}(1+\|W\|_{****}) \|\bar{\phi}\|_{\alpha,\sigma}\frac{\mu_0^{\sigma-1}(t)}{1+|y|^{1+\alpha}}.
\end{aligned}
\end{equation*}
By the same reason as Step 1 we get (\ref{e4:84}).
\end{proof}

\section{The inner problem}
Let $\psi = \Psi[\lambda,\xi,\dot{\lambda},\dot{\xi},\phi]$ be the solution of (\ref{e:outerproblem}) given by Proposition \ref{p4:4.1} and insert this term into (\ref{e:innerproblem}), then our main problem is reduced to the solvability of
\begin{equation}\label{e5:5.1}
\left\{
\begin{aligned}
&\partial_\tau \phi = L_\omega[\phi](y) + H[\lambda,\xi,\dot{\lambda},\dot{\xi},\phi](y,t(\tau))\text{ in }B_{2R}(0)\times [\tau_0, \infty),\\
&\phi = 0\quad\text{in }B_{2R}(0)\times \{\tau_0\},
\end{aligned}
\right.
\end{equation}
with
\begin{equation*}
\begin{aligned}
&H[\lambda,\xi,\dot{\lambda},\dot{\xi},\phi](y,t) =
\mu_0\Pi_{U^\perp}\mathcal{E}^*(\xi+\mu_0y, t) +\frac{2\frac{\mu_0}{\mu}}{1+\left|\frac{\mu_0}{\mu}y\right|^2}\Pi_{\omega^\perp}\psi \\ & \quad\quad\quad \quad\quad\quad\quad\quad\quad -\frac{\mu_0}{\pi}\int_{\mathbb{R}}\frac{\left[\psi(x)\cdot \omega\left(\frac{x-\xi}{\mu}\right)-\psi(s)\cdot \omega\left(\frac{s-\xi}{\mu}\right)\right]\left[\omega\left(\frac{x-\xi}{\mu}\right)-\omega\left(\frac{s-\xi}{\mu}\right)\right]}{|x-s|^2}ds \\
& + \left(\frac{\mu_0}{\pi}\int_{\mathbb{R}}\frac{\left[\psi(x)\cdot \omega\left(\frac{x-\xi}{\mu}\right)-\psi(s)\cdot \omega\left(\frac{s-\xi}{\mu}\right)\right]\left[\omega\left(\frac{x-\xi}{\mu}\right)-\omega\left(\frac{s-\xi}{\mu}\right)\right]}{|x-s|^2}ds\cdot \omega\left(\frac{x-\xi}{\mu}\right)\right)\omega\left(\frac{x-\xi}{\mu}\right)\\
& + B^1[\phi] + B^2[\phi] + B^3[\phi].
\end{aligned}
\end{equation*}

We will prove in this section that (\ref{e5:5.1}) is solvable for functions $\phi$ satisfying condition (\ref{e4:assumptiononphi}) when $\xi$ and $\lambda$ are chosen so that  $H[\lambda,\xi,\dot{\lambda},\dot{\xi},\phi](y,t(\tau))$ satisfies the following $L^2$-orthogonality conditions
\begin{equation*}
\int_{B_{2R}}H[\lambda,\xi,\dot{\lambda},\dot{\xi},\phi](y,t(\tau))Z_l(y)dy = 0
\end{equation*}
for all $\tau\geq \tau_0$ and $l = 2,3$. Our main tool is the Contraction Mapping Theorem, to do this, we first need a linear theory which is the context of next section. 
\subsection{The linear theory}
In this subsection, we consider the nonlocal initial value problem
\begin{equation}\label{e5:5.3}
\left\{
\begin{aligned}
&\partial_\tau \phi = L_\omega[\phi](y) + h(y,\tau)\text{ in }B_{2R}(0)\times [\tau_0, \infty),\\
&\phi(\cdot, \tau_0) = 0\quad\text{in }B_{2R}(0),\\
&\phi(y,\tau)\cdot \omega(y) = 0 \quad\text{in }B_{2R}(0)\times [\tau_0, \infty)
\end{aligned}
\right.
\end{equation}
for a large number $R > 0$.
Note that we do not add boundary conditions to (\ref{e5:5.3}), so we only need to find a solution $\phi$ defined on $\mathbb{R}\times [\tau_0,+\infty)$ such that (\ref{e5:5.3}) is satisfied, thus $-(-\Delta)^{\frac{1}{2}}\phi$ and other integrals on $\mathbb{R}$ is well defined.

For a fixed constant $\eta\in (\frac{1}{2}, 1)$, define
\begin{equation*}
\|h\|_{a, \nu, \eta}: = \sup_{\tau > \tau_0}\sup_{y\in B_{2R}}\tau^\nu(1 + |y|^a)(|h(y,\tau)|+(1+|y|^\eta)\chi_{\{|y|\leq 2R\}}[h(\cdot, \tau)]_{\eta, B_{2R}(0)}).
\end{equation*}
In the following, we assume $h = h(y, \tau)$ is a function defined in $\mathbb{R}$ which is zero outside of $B_{2R}(0)$ for any $\tau > \tau_0$.
The main result of this subsection is
\begin{prop}\label{p5:5.11}
Suppose $a\in (0, 1)$, $\nu > 0$, $\|h\|_{1+a,\nu, \eta} < +\infty$ and
\begin{equation*}
\int_{B_{2R}}h(y,\tau)\cdot Z_j(y)dy = 0\quad\text{for all}\quad\tau\in [\tau_0,+\infty),\quad j = 2, 3.
\end{equation*}
Then for sufficiently large $R$, there exist $\phi = \phi[h](y, \tau)$ defined on $\mathbb{R}\times [\tau_0,+\infty)$ satisfying (\ref{e5:5.3}) and
\begin{equation*}\label{e5:100}
(1+|y|)|\nabla_y \phi|\chi_{\{|y|\leq R\}}+ |\phi|\lesssim \tau^{-\nu}(1+|y|)^{-a}\|h\|_{1+a,\nu, \eta}, \tau\in [\tau_0,+\infty),y\in \mathbb{R}.
\end{equation*}
\end{prop}

Problem (\ref{e5:5.3}) is a system of two equations, but under the condition $\phi(y,\tau)\cdot \omega(y) = 0$, this system can be reduced to a single equation for a scalar function. Since $\phi(y,\tau)\cdot \omega(y) = 0$, we can write $\phi$ as $\phi(y, \tau) = v(y, \tau)Z_1(y)$ for a scalar function $v(y, \tau)$. Then by direct computation, (\ref{e5:5.3}) becomes
\begin{equation}\label{e:afterprojection}
\partial_\tau v(y, \tau) = -(-\Delta)^{\frac{1}{2}}v(y,\tau) + \frac{2}{1+y^2}v(y,\tau) - \frac{2}{\pi(1+y^2)}\int_{\mathbb{R}}\frac{v(s, \tau)}{1+s^2}ds + g(y, \tau)
\end{equation}
for a function $g(y,\tau)$ satisfies $h(y,\tau) = g(y,\tau) Z_1(y)$. Indeed, we have
\begin{equation*}
\begin{aligned}
(-\Delta)^{1/2}\phi(y,\tau) &= \frac{1}{\pi}\int_{\mathbb{R}}\frac{v(y, \tau)Z_1(y) - v(s, \tau)Z_1(s)}{|y-s|^2}ds\\
& = \frac{1}{\pi}\int_{\mathbb{R}}\frac{v(y, \tau)Z_1(y) -v(s, \tau)Z_1(y) + v(s, \tau)Z_1(y) - v(s, \tau)Z_1(s)}{|y-s|^2}ds\\
& = [(-\Delta)^{1/2}v]Z_1(y) + \frac{1}{\pi}\int_{\mathbb{R}}\frac{v(s,\tau)(Z_1(y)-Z_1(s))}{|y-s|^2}ds.
\end{aligned}
\end{equation*}
Since
\begin{equation*}
\begin{aligned}
\frac{1}{\pi}\int_{\mathbb{R}}\frac{v(s,\tau)(Z_1(y)-Z_1(s))}{|y-s|^2}ds\cdot \omega(y) &= \frac{1}{\pi}\int_{\mathbb{R}}\frac{v(s,\tau)(Z_1(y)\cdot \omega(y)-Z_1(s)\cdot \omega(y))}{|y-s|^2}ds\\
& = \frac{1}{\pi}\int_{\mathbb{R}}\frac{v(s,\tau)(Z_1(s)\cdot \omega(s)-Z_1(s)\cdot \omega(y))}{|y-s|^2}ds\\
& = \frac{1}{\pi}\int_{\mathbb{R}}\frac{v(s,\tau)Z_1(s)\cdot (\omega(s)-\omega(y))}{|y-s|^2}ds\\
& = \frac{1}{\pi}\int_{\mathbb{R}}\frac{-\phi(s,\tau)\cdot \omega(y)}{|y-s|^2}ds\\
& = \frac{1}{\pi}\int_{\mathbb{R}}\frac{\phi(y,\tau)-\phi(s,\tau) }{|y-s|^2}ds\cdot\omega(y)\\
&= [(-\Delta)^{1/2}\phi(y,\tau)]\cdot \omega(y)
\end{aligned}
\end{equation*}
and
\begin{equation*}
\begin{aligned}
\frac{1}{\pi}\int_{\mathbb{R}}\frac{v(s,\tau)(Z_1(y)-Z_1(s))}{|y-s|^2}ds\cdot Z_1(y) &=\frac{1}{\pi}\int_{\mathbb{R}}\frac{v(s,\tau)(Z_1(y)\cdot Z_1(y)-Z_1(s)\cdot Z_1(y))}{|y-s|^2}ds\\
& = \frac{1}{\pi}\int_{\mathbb{R}}\frac{v(s,\tau)(Z_1(s)\cdot Z_1(s)-Z_1(s)\cdot Z_1(y))}{|y-s|^2}ds\\
& = \frac{1}{\pi}\int_{\mathbb{R}}\frac{v(s,\tau)Z_1(s)\cdot(Z_1(s)-Z_1(y))}{|y-s|^2}ds\\
& = \frac{1}{\pi}\int_{\mathbb{R}}\frac{\phi(s,\tau)\cdot(Z_1(s)-Z_1(y))}{|y-s|^2}ds,
\end{aligned}
\end{equation*}
we have
\begin{equation*}
\begin{aligned}
&(-\Delta)^{1/2}\phi(y,\tau) =  [(-\Delta)^{1/2}v]Z_1(y) + [[(-\Delta)^{1/2}\phi(y,\tau)]\cdot \omega(y)]\omega(y)\\ &\quad\quad\quad\quad\quad\quad\quad + \left[\frac{1}{\pi}\int_{\mathbb{R}}\frac{\phi(s,\tau)\cdot(Z_1(s)-Z_1(y))}{|y-s|^2}ds\right]Z_1(y).
\end{aligned}
\end{equation*}
So the left hand side of (\ref{e5:5.3}) becomes
\begin{equation*}
\begin{aligned}
&-\partial_\tau \phi + L_\omega[\phi](y) = -\partial_\tau v(y,\tau) Z_1(y)\\
&\quad - [(-\Delta)^{1/2}v]Z_1(y) - [[(-\Delta)^{1/2}\phi(y,\tau)]\cdot \omega(y)]\omega(y) - \left[\frac{1}{\pi}\int_{\mathbb{R}}\frac{\phi(s,\tau)\cdot(Z_1(s)-Z_1(y))}{|y-s|^2}ds\right]Z_1(y)\\
&\quad + \left(\frac{1}{2\pi}\int_{\mathbb{R}}\frac{|\omega(y)-\omega(s)|^2}{|y-s|^2}ds\right)\phi(y,\tau)
+ \left(\frac{1}{\pi}\int_{\mathbb{R}}\frac{(\omega(y)-\omega(s))\cdot(\phi(y,\tau) -\phi(s,\tau))}{|y-s|^2}ds\right)\omega(y)\\
& = -\partial_\tau v(y,\tau) Z_1(y) - [(-\Delta)^{1/2}v]Z_1(y) - \left[\frac{1}{\pi}\int_{\mathbb{R}}\frac{\phi(s,\tau)\cdot(Z_1(s)-Z_1(y))}{|y-s|^2}ds\right]Z_1(y)\\
&\quad + \left(\frac{1}{2\pi}\int_{\mathbb{R}}\frac{|\omega(y)-\omega(s)|^2}{|y-s|^2}ds\right)\phi(y,\tau)\\
& = -\partial_\tau v(y,\tau) Z_1(y) - [(-\Delta)^{1/2}v]Z_1(y) - \left[\frac{1}{\pi}\int_{\mathbb{R}}\frac{\phi(s,\tau)\cdot(Z_1(s)-Z_1(y))}{|y-s|^2}ds\right]Z_1(y)\\
&\quad + \frac{2}{1+|y|^2}\phi(y,\tau).
\end{aligned}
\end{equation*}
Moreover, we have
\begin{equation*}
\begin{aligned}
\frac{1}{\pi}\int_{\mathbb{R}}\frac{\phi(s,\tau)\cdot(Z_1(s)-Z_1(y))}{|y-s|^2}ds & = \frac{1}{\pi}\int_{\mathbb{R}}\frac{v(s,\tau)(1-Z_1(s)\cdot Z_1(y))}{|y-s|^2}ds\\
& = \frac{2}{\pi(1+y^2)}\int_{\mathbb{R}}\frac{v(s,\tau)}{1+s^2}ds.
\end{aligned}
\end{equation*}
Therefore, (\ref{e5:5.3}) becomes
\begin{eqnarray*}
-\partial_\tau v(y,\tau) - [(-\Delta)^{1/2}v] - \frac{2}{\pi(1+y^2)}\int_{\mathbb{R}}\frac{v(s,\tau)}{1+s^2}ds + \frac{2}{1+|y|^2}v(y,\tau) + g(y,\tau) = 0,
\end{eqnarray*}
which is (\ref{e:afterprojection}).

Hence Proposition \ref{p5:5.11} is equivalent to
\begin{prop}\label{p5:5.12}
Suppose $a\in (0, 1)$, $\nu > 0$, $\|g\|_{1+a,\nu, \eta} < +\infty$ and
\begin{equation*}
\int_{B_{2R}}g(y,\tau)\frac{2}{y^2+1}dy = \int_{B_{2R}}g(y,\tau)\frac{2y}{y^2+1}dy = 0\quad\text{for all}\quad\tau\in [\tau_0,+\infty).
\end{equation*}
Then for sufficiently large $R$, there exists $v = v[h](y, \tau)$ defined on $\mathbb{R}\times [\tau_0,+\infty)$ satisfying (\ref{e:afterprojection}) and
\begin{equation*}
(1+|y|)|\nabla_y v|\chi_{\{|y|\leq R\}}+ |v|\lesssim \tau^{-\nu}(1+|y|)^{-a}\|g\|_{1+a,\nu, \eta}, \tau\in [\tau_0,+\infty),y\in \mathbb{R}.
\end{equation*}
\end{prop}

To prove this proposition, we first consider the following equation defined in the whole space
\begin{equation}\label{e5:32}
\left\{
\begin{array}{lll}
\partial_\tau v = -(-\Delta)^{\frac{1}{2}}v + \frac{2}{1+y^2}v - \frac{2}{\pi(1+y^2)}\int_{\mathbb{R}}\frac{v(s,\tau)}{1+s^2}ds + g\text{ in }\mathbb{R}\times [\tau_0, +\infty),\\
v(y,\tau_0) = 0\text{ in }\mathbb{R}.
\end{array}
\right.
\end{equation}
The following result is crucial in the proof.
\begin{lemma}\label{l5:3}
Suppose $a\in (0, 1)$, $\nu > 0$, $\|g\|_{1+a,\nu, \eta} < +\infty$ and
\begin{equation*}
\int_{\mathbb{R}}g(y,\tau)\frac{2}{y^2+1}dy = \int_{\mathbb{R}}g(y,\tau)\frac{2y}{y^2+1}dy = 0\quad\text{for all }\tau\in [\tau_0,+\infty).
\end{equation*}
Then for any sufficiently large $\tau_1 > \tau_0$, the solution of (\ref{e5:32}) satisfies
\begin{equation}\label{e5:35}
\|v(y,\tau)\|_{a,\tau_1}\lesssim \|g\|_{1+a,\tau_1}.
\end{equation}
Here, $\|g\|_{b, \tau_1}:=\sup_{\tau\in (\tau_0,\tau_1)}\tau^\nu\|(1+|y|^b)g\|_{L^\infty(\mathbb{R})}$.
\end{lemma}
\begin{proof}
First, we claim that given $\tau_1 > \tau_0$ we have $\|v\|_{a, \tau_1} < +\infty$. Indeed, standard linear theory indicates that $v(y, \tau)$ is locally bounded in time and space, i.e., given $R > 0$ there is a $K = K(R,\tau_1)$ such that
\begin{equation*}
|v(y,\tau)|\leq K \quad\text{in }B_R(0)\times (\tau_0, \tau_1].
\end{equation*}
If we fix a large $R$ and take a sufficiently large $K_1$, then $K_1\rho^{-a}$ is a supersolution for (\ref{e5:32}). Therefore $|v|\leq 2K_1\rho^{-a}$ and $\|v\|_{a,\tau_1} < +\infty$ for any $\tau_1 > 0$. Next, we claim that
\begin{equation}\label{e5:33}
\int_{\mathbb{R}}v(y,\tau)\frac{2}{y^2+1}dy = \int_{\mathbb{R}}v(y,\tau)\frac{2y}{y^2+1}dy = 0\text{ for all }\tau\in (\tau_0,\tau_1).
\end{equation}
Denote $w^{1} = \frac{2}{y^2+1}$ and $w^{2} = \frac{2y}{y^2+1}$,
for $j = 1, 2$, test equation (\ref{e5:32}) against
\begin{equation*}
w^{j}\eta,\quad \eta(y) = \eta_0(|y|/R)
\end{equation*}
where $\eta_0$ is a smooth cut-off function with $\eta_0(r) = 1$ for $r < 1$, $\eta_0(r) = 0$ for $r > 2$ and $R$ is an arbitrary large constant, then we have
\begin{equation*}
\int_{\mathbb{R}}v(\cdot, \tau) w^j\eta = \int_{0}^\tau ds\int_{\mathbb{R}}v(\cdot, s)\cdot(L[\eta w^j] + g\cdot w^j\eta).
\end{equation*}
On the other hand,
\begin{equation*}
\begin{aligned}
&\int_{\mathbb{R}}v\cdot(L[\eta w^j] + g\cdot w^j\eta)\\
&\quad\quad\quad\quad = \int_{\mathbb{R}}v\cdot (w^j(-(-\Delta)^{\frac{1}{2}})\eta + [-(-\Delta)^{1/2})\eta, -(-\Delta)^{s/2})w^j]) - g\cdot w^j(1-\eta)\\
&\quad\quad\quad\quad\quad - \int_{\mathbb{R}}v\cdot \frac{2}{\pi(1+|y|^2)}\int_{\mathbb{R}}\frac{(\eta-1) w^j(s)}{s^2+1}ds\\
&\quad\quad\quad\quad = O(R^{-\epsilon})
\end{aligned}
\end{equation*}
uniformly on $\tau\in (\tau_0,\tau_1)$ and $\epsilon > 0$ is a small constant. Here
$$
L[v]:= -(-\Delta)^{\frac{1}{2}}v + \frac{2}{1+y^2}v - \frac{2}{\pi(1+y^2)}\int_{\mathbb{R}}\frac{v(s,\tau)}{1+s^2}ds
$$
and
$$[-(-\Delta)^{1/2})\eta, -(-\Delta)^{s/2})w^j]:= \frac{1}{2\pi}\int_{\mathbb{R}}\frac{(\eta(x)-\eta(y))(w^j(x)-w^j(y))}{|x-y|^2}dy.$$ Letting $R\to +\infty$ we obtain (\ref{e5:33}).

Now we claim that for all $\tau_1 > 0$ sufficiently large, any solution $v$ of (\ref{e5:32}) satisfying $\|v\|_{a,\tau_1} < +\infty$ and (\ref{e5:33}), the following estimate
\begin{equation}\label{e5:34}
\|v\|_{a,\tau_1}\lesssim \|g\|_{1+a,\tau_1}
\end{equation}
holds. This implies (\ref{e5:35}).

To prove (\ref{e5:34}), we use the contradiction argument. Suppose there exist sequence $\tau_1^k\to +\infty$ and $\phi_k$, $g_k$ satisfying
\begin{equation*}
\begin{aligned}
\partial_\tau v_k = -(-\Delta)^{\frac{1}{2}}v_k + \frac{2}{1+y^2}v_k - \frac{2}{\pi(1+y^2)}\int_{\mathbb{R}}\frac{v_k(s,\tau)}{1+s^2}ds + g_k,\quad y\in \mathbb{R},\quad \tau\geq \tau_0\\
\int_{\mathbb{R}}v_k(y, \tau)\cdot w^j(y)dy = 0\text{ for all }\tau\in (\tau_0,\tau_1^k), j = 1,2\\
v_k(y,\tau_0) = 0,\quad y\in \mathbb{R},
\end{aligned}
\end{equation*}
and
\begin{equation}\label{e5:38}
\|v_k\|_{a,\tau_1^k}\to 0,\quad \|g_k\|_{1+a,\tau_1^k}\to 0
\end{equation}
We claim first that
\begin{equation}\label{e5:37}
\sup_{\tau_0 < \tau < \tau_1^k}\tau^\nu|v_k(y,\tau)|\to 0
\end{equation}
uniformly on compact subsets of $\mathbb{R}$. If not, for some $|y_k|\leq M$ and $\tau_0 < \tau_2^k < \tau_1^k$, we have
\begin{equation*}
(\tau_2^k)^\nu(1+|y_k|^a)|v_k(y_k,\tau_2^k)|\geq \frac{1}{2}.
\end{equation*}
Clearly $\tau_2^k\to +\infty$. Define
\begin{equation*}
\tilde{v}_k(y,\tau) = (\tau_2^k)^\nu v_k(y,\tau_2^k + \tau),
\end{equation*}
then
\begin{equation*}
\partial_\tau\tilde{v}_k = L[\tilde{v}_k] + \tilde{g}_k\quad\text{in }\mathbb{R}\times (\tau_0-\tau_2^k,0],
\end{equation*}
where $\tilde{g}_k\to 0$ uniformly on compact subsets of $\mathbb{R}\times (-\infty, 0]$ and
\begin{equation*}
|\tilde{v}_k(y,\tau)|\leq \frac{1}{1+|y|^a}\quad\text{in }\mathbb{R}\times (\tau_0-\tau_2^k,0].
\end{equation*}
From parabolic estimates \cite{silvestre2014regularity} and passing to a subsequence, $\tilde{v}_k\to\tilde{v}$ uniformly on compact subsets of $\mathbb{R}\times (-\infty, 0]$ where $\tilde{v}\neq 0$ and
\begin{equation*}
\begin{aligned}
\partial_\tau\tilde{v} = -(-\Delta)^{\frac{1}{2}}\tilde{v} + \frac{2}{1+y^2}\tilde{v} - \frac{2}{\pi(1+y^2)}\int_{\mathbb{R}}\frac{\tilde{v}(s,\tau)}{1+s^2}ds,\quad\text{in }\mathbb{R}\times (-\infty, 0],\\
\int_{\mathbb{R}}\tilde{v}(y, \tau)\cdot w^j(y)dy = 0\text{ for all }\tau\in (-\infty, 0], \quad j= 1,2,\\
|\tilde{v}(y,\tau)|\leq \frac{1}{1+|y|^a}\quad\text{in }\mathbb{R}\times (-\infty, 0].
\end{aligned}
\end{equation*}
We claim that necessarily $\tilde{v} = 0$ which is a contradiction. By parabolic regularity \cite{silvestre2014regularity}, $\tilde{v}(y,\tau)$ is a smooth function. Testing the above equation with $\tilde{v}$, we have
\begin{equation*}
\frac{1}{2}\partial_\tau\int_{\mathbb{R}}|\tilde{v}_\tau|^2 + B(\tilde{v}_\tau, \tilde{v}_\tau) = 0
\end{equation*}
where
\begin{equation*}
B(\tilde{v}, \tilde{v}) = \int_{\mathbb{R}}\left[((-\Delta)^{\frac{1}{2}}\tilde{v})\tilde{v} - \frac{2}{1+y^2}\tilde{v}^2 + \frac{2\tilde{v}}{\pi(1+y^2)}\int_{\mathbb{R}}\frac{\tilde{v}(s,\tau)}{1+s^2}ds\right]dy.
\end{equation*}
Clearly, $B(\tilde{v}, \tilde{v})\geq 0$ and
\begin{equation*}
\int_{\mathbb{R}}|\tilde{v}_\tau|^2 = -\frac{1}{2}\partial_\tau B(\tilde{v}_\tau, \tilde{v}_\tau) = 0.
\end{equation*}
Therefore
\begin{equation*}
\partial_\tau\int_{\mathbb{R}}|\tilde{v}_\tau|^2 \leq 0,\quad \int_{-\infty}^0d\tau\int_{\mathbb{R}}|\tilde{v}|^2 < +\infty
\end{equation*}
and hence $\tilde{v}_\tau = 0$. Thus $\tilde{v}$ is independent of $\tau$ and $L[\tilde{v}] = 0$. Since $\tilde{v}$ decays at infinity, the nondegeneracy result \cite{sire2017nondegeneracy} implies that $\tilde{v}$ is a linear combination of $w^j$, $j = 1,2$. Since $\int_{\mathbb{R}}\tilde{v}\cdot w^j = 0$, $j = 1,2$, $\tilde{v} = 0$, a contradiction. So (\ref{e5:37}) holds. From (\ref{e5:38}), for a certain $y_n$ with $|y_n|\to +\infty$ we have
\begin{equation*}
(\tau_2^k)^\nu|y_k|^a|v_k(y_k, \tau_2^k)|\geq \frac{1}{2}.
\end{equation*}
Now let
\begin{equation*}
\tilde{v}_k(z, \tau):=(\tau_2^k)^\nu|y_k|^av_k(y_k+|y_k|z,|y_k|\tau + \tau_2^k)
\end{equation*}
so that
\begin{equation*}
\partial_\tau \tilde{v}_k = -(-\Delta)^{\frac{1}{2}}\tilde{v}_k + \frac{2}{1+|y_k+|y_k|z|^2}\tilde{v}_k - \frac{2}{\pi(1+|y_k+|y_k|z|^2)}\int_{\mathbb{R}}\frac{\tilde{v}_k(\frac{s-y_k}{|y_k|},|y_k|\tau+\tau_2^k)}{1+s^2}ds + \tilde{g}_k(z,\tau)
\end{equation*}
where
\begin{equation*}
\tilde{g}_k(z,\tau) = (\tau_2^k)^\nu|y_k|^{1+a}g_k(y_k+|y_k|z,|y_k|\tau + \tau_2^k).
\end{equation*}
By assumption on $g_k$ we have
\begin{equation*}
|\tilde{g}_k(z,\tau)| \lesssim o(1)|\hat{y}_k+z|^{-1-a}((\tau_2^k)^{-1}|y_k|\tau + 1)^{-\nu}
\end{equation*}
with
\begin{equation*}
\hat{y}_k = \frac{y_k}{|y_k|}\to -\hat{e}
\end{equation*}
and $|\hat{e}|= 1$. Hence $\tilde{g}_k(z,\tau)\to 0$ uniformly on compact subsets of $\mathbb{R}\setminus\{\hat{e}\}\times (-\infty, 0]$ and the same property holds for $\frac{2}{1+|y_k+|y_k|z|^2}\tilde{v}_k$ and $\frac{2}{\pi(1+|y_k+|y_k|z|^2)}\int_{\mathbb{R}}\frac{\tilde{v}_k(\frac{s-y_k}{|y_k|},|y_k|\tau+\tau_2^k)}{1+s^2}ds$. On the other hand, we have $|\tilde{v}_k|\geq \frac{1}{2}$ and
\begin{equation*}
|\tilde{v}_k(z,\tau)| \lesssim |\hat{y}_k+z|^{-a}((\tau_2^k)^{-1}|y_k|\tau + 1)^{-\nu}.
\end{equation*}
Therefore, we can assume that $\tilde{v}_k\to \tilde{v}\neq 0$ uniformly on compact subsets of $\mathbb{R}\setminus\{\hat{e}\}\times (-\infty,0]$ and $\tilde{v}$ satisfies
\begin{equation}\label{e5:39}
\tilde{v}_\tau = -(-\Delta)^{\frac{1}{2}}\tilde{v}\quad\text{in }\mathbb{R}\setminus\{\hat{e}\}\times (-\infty,0]
\end{equation}
and
\begin{equation}\label{e5:40}
|\tilde{v}(z,\tau)|\leq |z-\hat{e}|^{-a}\quad\text{in }\mathbb{R}\setminus\{\hat{e}\}\times (-\infty,0].
\end{equation}
By Lemma \ref{l4.3}, functions $\tilde{v}$ satisfying (\ref{e5:39}) and (\ref{e5:40}) must be zero, which is a contradiction.
This completes the proof.
\end{proof}
\begin{lemma}\label{l4.3}
Suppose $u = u(x, t):\mathbb{R}\times (-\infty, 0] \to \mathbb{R}$ satisfies $\sup_{(x, t)\in \mathbb{R}\times (-\infty, 0]}|x|^a|u(x, t)|\leq C$
for some $C > 0$ and $0 < a < 1$.
If $u$ is an ancient solution to the equation
\begin{equation*}
\partial_t u = -\Lambda^\alpha u,
\end{equation*}
where $\Lambda =  (-\Delta)^{\frac{1}{2}}$ and $0 < \alpha \leq 2$, then $u \equiv 0$.
\end{lemma}
\begin{proof}
Note that by assumption $u(\cdot, t)$ is a tempered distribution. Let $P_j$ be the usual Littlewood-Paley projection operator adapted to frequency $|\xi|\sim 2^j$, $j\in\mathbb{Z}$, see \cite{Grafakos2004}. Then for each fixed $j$, it is not difficult to check that
\begin{equation*}
|(P_ju)(x, t)|\leq C_j(1+|x|^2)^{-\frac{a}{2}},
\end{equation*}
where the constant $C_j$ is independent of $t$. In particular it follows that for any $\frac{1}{a} < q < \infty$,
\begin{equation*}
\sup_{t}\|P_ju(\cdot, t)\|_q < \infty.
\end{equation*}

Consider the equation for $P_ju$, we have
\begin{equation*}
\partial_tP_ju = -\Lambda^\alpha P_ju.
\end{equation*}
Fix $q$ with $\frac{1}{a} < q < \infty$ and let $-\infty < t_1 < t_0 < 0$ where $t_1$ will tend to $-\infty$. Multiplying both sides by $|P_ju|^{q-2}P_ju$ and integrating, we obtain
\begin{equation*}
\begin{aligned}
&\partial_t\|P_ju\|_q^q = -\int_{\mathbb{R}}\Lambda^\alpha P_ju|P_ju|^{q-2}P_judx\\
&\quad\quad\quad\leq -c_1\|P_ju\|_q^q,
\end{aligned}
\end{equation*}
where $c_1 > 0$ is a constant independent of $t$ and in the second line we have used the improved Bernstein inequality, see \cite{LiDongMRL2013}. Integrating in time and using the uniform in time $L^q$ bound then gives
\begin{equation*}
\|P_ju(\cdot, t_0)\|_q^q\leq e^{-c_1(t_0 - t_1)}\|P_ju(\cdot, t_1)\|_q^q\to 0,\text{ as }t_1\to -\infty.
\end{equation*}

It follows easily that $P_ju(\cdot, t_0)\equiv 0$ for any $j\in \mathbb{Z}$. Thus the Fourier transformation $\hat{u}$ can only be supported at $\xi = 0$ which implies that $u(\cdot, t_0)$ is polynomial. By the decay condition we conclude that $u(\cdot, t_0)\equiv 0$. Since $t_0$ is arbitrary, $u \equiv 0$.
This completes the proof.
\end{proof}

{\it Proof of Proposition \ref{p5:5.12}}.
Extend the function $g$ in the statement of the lemma as zero outside $B_{2R}(0)$ and still denote it as $g$. Let $v$ be the unique bounded solution of the initial value problem (\ref{e5:32}) which defines a linear operator of $g$. From Lemma \ref{l5:3}, for any $\tau_1 > 0$, we have
\begin{equation*}
|v(y,\tau)|\leq C\tau^{-\nu}(1+|y|)^{-a}\|g\|_{1+a, \tau_1} \quad\text{ for all }\tau\in (\tau_0, \tau_1), y\in \mathbb{R}.
\end{equation*}
By assumption $\|g\|_{1+a,\nu, \eta} < +\infty$, hence
\begin{equation*}
|g(y,\tau)|\leq C\tau^{-\nu}(1+|y|)^{-(1+a)}\|g\|_{1+a,\nu, \eta}
\end{equation*}
and
\begin{equation*}
\|g\|_{1+a, \tau_1}\leq \|g\|_{1+a,\nu, \eta}
\end{equation*}
for an arbitrary $\tau_1$. This implies that
\begin{equation*}
|v(y,\tau)|\leq C\tau^{-\nu}(1+|y|)^{-a}\|g\|_{1+a,\nu, \eta}\quad\text{ for all }\tau\in (\tau_0, \tau_1), y\in \mathbb{R}.
\end{equation*}

Now we consider the regularity part. Since $v$ satisfies
\begin{equation*}
\partial_\tau v = -(-\Delta)^{\frac{1}{2}}v + f(y,\tau)
\end{equation*}
with $f(y,\tau):=\frac{2}{1+y^2}v - \frac{2}{\pi(1+y^2)}\int_{\mathbb{R}}\frac{v(s)}{1+s^2}ds + g(y, \tau)$ on $B_{2R}(0)$. For a fixed constant $\eta\in (\frac{1}{2}, 1)$ and $\tau_1 \geq \tau_0+1$, from the regularity estimates for parabolic integro-differential (see \cite{silvestreium2012differentiability} and \cite{silvestre2014regularity}), we obtain
\begin{equation*}
\begin{aligned}
~[v(\cdot,\tau_1)]_{\eta, B_{1}(0)} &\lesssim \|v\|_{L^\infty}+\|f\|_{L^\infty}\\
&\lesssim (\tau_1-1)^{-\nu}\|g\|_{1+a,\nu, \eta}\\
&\lesssim \tau_1^{-\nu}\|g\|_{1+a,\nu, \eta}.
\end{aligned}
\end{equation*}
By Theorem 1.1 of \cite{silvestreium2012differentiability}, under the assumption that $g\in L^\infty([\tau_0,\infty), C^\eta(B_{2R}(0)))$, for $\tau_2 > \tau_0 + 2$ and $y\in B_{1/4}(0)$,
\begin{equation*}
\begin{aligned}
|\nabla_y v(y,\tau_2)| &\lesssim \|v\|_{L^\infty}+\|f\|_{L^\infty([\tau_2-1,\tau_2], C^\eta(B_{1/2}(0)))}\\
&\lesssim (\tau_2-1)^{-\nu}\|g\|_{1+a,\nu, \eta}\\
&\lesssim \tau_2^{-\nu}\|g\|_{1+a,\nu, \eta}.
\end{aligned}
\end{equation*}
Therefore, we have
\begin{equation}\label{e4:20}
4R|\nabla_y v(y,\tau_2)|\lesssim \tau_2^{-\nu}\|g\|_{1+a,\nu, \eta}.
\end{equation}
Estimate (\ref{e4:20}) also holds for $\tau_0\leq \tau \leq \tau_0 + 2$ by a similar parabolic regularity estimate.
This completes the proof.
\qed

\subsection{The solvability condition: choice of the parameters $\lambda$ and $\xi$}
In this subsection, we choose $\lambda$ and $\xi$ such that the orthogonality conditions
\begin{equation}\label{e5:7}
\int_{B_{2R}}H[\phi,\lambda,\xi,\dot{\lambda},\dot{\xi}](y,t(\tau))Z_{l}(y)dy = 0,\quad l = 2,3
\end{equation}
hold. First, we consider the case $l =3$.
\begin{lemma}\label{l5:1}
When $l = 3$, (\ref{e5:7}) is equivalent to
\begin{equation*}
\dot{\lambda}(t)+\frac{10}{3}\kappa_0\lambda(t) = \Pi_1[\lambda, \xi, \dot{\lambda}, \dot{\xi}, \phi](t).
\end{equation*}
The right hand side term can be expressed as
\begin{equation}\label{e5:11}
\Pi_1[\lambda, \xi, \dot{\lambda}, \dot{\xi}, \phi](t) = e^{-\varepsilon t_0}\mu_0(t)^\sigma f_1(t) + e^{-\varepsilon t_0}\Theta[\dot{\lambda},\dot{\xi}, \lambda, \mu_0(\xi-q), \mu_0^{\sigma}\phi](t)
\end{equation}
where $f(t)$ and $\Theta[\cdots](t)$ are smooth and bounded functions for $t\in [t_0,\infty)$ and the following estimates hold,
\begin{equation*}
\left|\Theta[\dot{\lambda}_1](t) - \Theta[\dot{\lambda}_2](t)\right|\lesssim e^{-\varepsilon t_0}|\dot{\lambda}_1(t) - \dot{\lambda}_2(t)|
\end{equation*}
\begin{equation*}
\left|\Theta[\dot{\xi}_1](t) - \Theta[\dot{\xi}_2](t)\right|\lesssim e^{-\varepsilon t_0}|\dot{\xi}_1(t) - \dot{\xi}_2(t)|,
\end{equation*}
\begin{equation*}
\left|\Theta[\lambda_1](t) - \Theta[\lambda_2](t)\right|\lesssim e^{-\varepsilon t_0}|\dot{\lambda}_1(t) - \dot{\lambda}_2(t)|
\end{equation*}
\begin{equation*}
\left|\Theta[\mu_0(\xi_1-q)](t) - \Theta[\mu_0(\xi_2-q)](t)\right|\lesssim e^{-\varepsilon t_0}|\xi_1(t) - \xi_2(t)|,
\end{equation*}
\begin{equation}\label{e5:104}
\left|\Theta[\mu_0^{\sigma}\phi_1](t) - \Theta[\mu_0^{\sigma}\phi_2](t)\right|\lesssim e^{-\varepsilon t_0}\|\phi_1(t) - \phi_2(t)\|_{a,\sigma}.
\end{equation}
\end{lemma}
\begin{proof}
We compute
\begin{eqnarray*}
\int_{B_{2R}}H[\phi,\lambda,\xi,\dot{\lambda},\dot{\xi}](y,t(\tau))Z_{3}(y)dy,
\end{eqnarray*}
where $H_j$ is given by (\ref{e:innerproblem}). Write
\begin{equation*}
\begin{aligned}
\mu_{0}\Pi_{U^\perp}\mathcal{E}^*(\xi + \mu_{0}y, t)&=\mu_0[S_1(z, t) + S_2(z, t) + S_3(z, t)]_{z = \xi + \mu y}\\
&\quad+\mu_0[S_1(\xi+\mu_{0}y, t)-S_1(\xi+\mu y, t)]\\
&\quad+\mu_0[S_2(\xi+\mu_{0}y, t)-S_2(\xi+\mu y, t)]\\
&\quad+\mu_0[S_3(\xi+\mu_{0}y, t)-S_3(\xi+\mu y, t)],
\end{aligned}
\end{equation*}
where
\begin{eqnarray*}
&&S_1(z,t) = \Bigg\{\frac{\dot{\mu}}{\mu}\begin{pmatrix}
\frac{-4y(y^2-1)}{(1+y^2)^3} \\
\frac{8y^2}{(1+y^2)^3}
\end{pmatrix} + \\
&&\quad\quad\quad\quad\quad \int_{t_0}^t\frac{p(\tilde{s})}{\mu^2(\tilde{s})}\frac{\left(\frac{t-\tilde{s}}{\mu}+1\right)^2}{\left(\frac{t-\tilde{s}}{\mu}+2\right)^2 \left(\left(\frac{t-\tilde{s}}{\mu}+1\right)^2+y^2\right)}d\tilde{s}Z_3(y)\\
\end{eqnarray*}
\begin{eqnarray*}
&&\quad\quad+\frac{1}{\mu}\left[
\frac{1}{\pi}\int_{\mathbb{R}}\frac{z_1^*(s,t)-z_1^*(x,t)}{s-x}\frac{2\left(\frac{s-\xi}{\mu}\right)^2}
{\left[\left(\frac{s-\xi}{\mu}\right)^2+1\right]^2}ds\right. \\
&&\quad\quad\left. \quad\quad\quad +
\frac{1}{\pi}\int_{\mathbb{R}}\frac{z_2^*(s,t)-z_2^*(x,t)}{s-x}\frac{\left(\frac{s-\xi}{\mu}\right)\left(\left(\frac{s-\xi}{\mu}\right)^2-1\right)}
{\left[\left(\frac{s-\xi}{\mu}\right)^2+1\right]^2}ds
\right]Z_3(y)\Bigg\}\Bigg |_{y = \frac{z-\xi(t)}{\mu(t)}},
\end{eqnarray*}
\begin{eqnarray*}
&& S_2(z,t) = \left(- \frac{\dot{\xi}}{\mu}Z_2(y) -\int_{t_0}^t\frac{p(\tilde{s})}{\mu^2(\tilde{s})}\frac{y}{\left(\frac{t-\tilde{s}}{\mu}+2\right)^2 \left(\left(\frac{t-\tilde{s}}{\mu}+1\right)^2+y^2\right)}d\tilde{s}Z_2(y)\right)\Bigg |_{y = \frac{z-\xi(t)}{\mu(t)}}\\
&&\quad\quad\quad\quad +\frac{1}{\mu}\Bigg\{\left[\frac{1}{\pi}\int_{\mathbb{R}}\frac{z_1^*(s,t)-z_1^*(x,t)}{s-x}\frac{2\left(\frac{s-\xi}{\mu}\right)}
{\left[\left(\frac{s-\xi}{\mu}\right)^2+1\right]^2}ds\right.\\
&&\quad\quad\quad\quad\quad\quad\quad\left. + \frac{1}{\pi}\int_{\mathbb{R}}\frac{z_2^*(s,t)-z_2^*(x,t) }{s-x}\frac{\left(\frac{s-\xi}{\mu}\right)^2-1}{\left[\left(\frac{s-\xi}{\mu}\right)^2+1\right]^2}ds\right]Z_2(y)\Bigg\}\Bigg|_{y = \frac{z-\xi(t)}{\mu(t)}}
\end{eqnarray*}
and
\begin{eqnarray*}
S_3(z,t) = \left((\Phi^0\cdot U)U_t + (\Phi^0\cdot U_t)U + (Z\cdot U)U_t + (Z\cdot U_t)U\right)\Big |_{y = \frac{z-\xi(t)}{\mu(t)}}.
\end{eqnarray*}
By direct computations, we have
\begin{eqnarray*}
&&\int_{B_{2R}}\mu_0S_1(\xi+\mu y)Z_{3}(y)dy\\
&&= -\pi\frac{\dot{\mu}}{\mu}\mu_0(1+O(1/R)) + 2\pi\frac{\mu_0}{\mu}\left[
\frac{1}{\pi}\int_{\mathbb{R}}\frac{z_1^*(s,t)-z_1^*(x,t)}{s-x}\frac{2\left(\frac{s-\xi}{\mu}\right)^2}
{\left[\left(\frac{s-\xi}{\mu}\right)^2+1\right]^2}ds\right. \\
&&\quad\quad\left. +
\frac{1}{\pi}\int_{\mathbb{R}}\frac{z_2^*(s,t)-z_2^*(x,t)}{s-x}\frac{\left(\frac{s-\xi}{\mu}\right)\left(\left(\frac{s-\xi}{\mu}\right)^2-1\right)}
{\left[\left(\frac{s-\xi}{\mu}\right)^2+1\right]^2}ds
\right]\Bigg|_{x=\xi(t)}(1+O(1/R))\\
&&\quad\quad+ 2\pi\mu_0\int_{t_0}^t\frac{p(\tilde{s})}{\mu^2(\tilde{s})}\frac{\left(\frac{t-\tilde{s}}{\mu}+1\right)^2}{\left(\frac{t-\tilde{s}}
{\mu}+2\right)^4}d\tilde{s}(1+O(1/R))\\
\end{eqnarray*}
\begin{eqnarray*}
&&= -\pi\frac{\dot{\mu}}{\mu}\mu_0(1+O(1/R))\\
&&\quad + 2\pi\mu_0\left[
\frac{1}{\pi}\int_{\mathbb{R}}\frac{z_2^*(s,t)-z_2^*(x,t)}{s-x}\frac{\left(\frac{s-\xi}{\mu}\right)\left(\left(\frac{s-\xi}{\mu}\right)^2-1\right)}
{\left[\left(\frac{s-\xi}{\mu}\right)^2+1\right]^2}ds
\right]\Bigg|_{x=\xi(t)}(1+O(1/R))\\
&&\quad+ 2\pi\mu_0\int_{t_0}^t\frac{p(\tilde{s})}{\mu^2(\tilde{s})}\frac{\left(\frac{t-\tilde{s}}{\mu}+1\right)^2}{\left(\frac{t-\tilde{s}}
{\mu}+2\right)^4}d\tilde{s}(1+O(1/R))
+O(\mu_0^2)\\
&&= -\pi\frac{\dot{\mu}}{\mu}\mu_0(1+O(1/R))\\
&&\quad + 2\pi\mu_0\left[
\frac{1}{\pi}\int_{\mathbb{R}}\frac{z_2^*(s+\xi(t),t)-z_2^*(x+\xi(t),t)}{s-x}\frac{s^3}
{\left(s^2+1\right)^2}ds
\right]\Bigg|_{x=0}(1+O(1/R))\\
&&\quad+ 2\pi\mu_0\int_{t_0}^t\frac{p(\tilde{s})}{\mu^2(\tilde{s})}\frac{\left(\frac{t-\tilde{s}}{\mu}+1\right)^2}{\left(\frac{t-\tilde{s}}
{\mu}+2\right)^4}d\tilde{s}(1+O(1/R))
+O(\mu_0^2)\\
&&= -\pi\dot{\lambda} -\pi \kappa_0\lambda+ 2\pi\mu_0\int_{t_0}^t-2\frac{p(\tilde{s})}{\mu^3_0(\tilde{s})}\lambda(\tilde{s})\frac{\left(\frac{t-\tilde{s}}{\mu_0}+1\right)^2}
{\left(\frac{t-\tilde{s}}{\mu_0}+2\right)^4}d\tilde{s}+O(\mu_0^{1+\sigma}).
\end{eqnarray*}
The same reason as Section 2.5, we easily know that
\begin{eqnarray*}
\mu_0(t)\int_{t_0}^t-2\frac{p(\tilde{s})}{\mu^3_0(\tilde{s})}\lambda(\tilde{s})\frac{\left(\frac{t-\tilde{s}}{\mu_0}+1\right)^2}
{\left(\frac{t-\tilde{s}}{\mu_0}+2\right)^4}d\tilde{s} = -\frac{7}{6}\kappa_0\lambda(t) + O(\mu_0^{1+\sigma}),
\end{eqnarray*}
hence
\begin{eqnarray*}
&&\int_{B_{2R}}\mu_0S_1(\xi+\mu y)Z_{3}(y)dy = -\pi\dot{\lambda}(t) -\pi \kappa_0\lambda(t)-\frac{7}{3}\pi\kappa_0\lambda(t) + O(\mu_0^{1+\sigma}).
\end{eqnarray*}
By symmetry,
\begin{equation*}
\int_{B_{2R}}\mu_0S_2(\xi+\mu y)Z_{3}(y)dy =  0.
\end{equation*}
Moreover, it holds that
\begin{equation*}
\int_{B_{2R}}\mu_0S_3(\xi+\mu y)Z_{3}(y)dy = \mu_0(t)O(1/R).
\end{equation*}
Since $\frac{\mu_{0}}{\mu} = (1 + \frac{\lambda}{\mu_{0}})^{-1}$, for any $l = 1, 2, 3$, we have
\begin{equation*}
\begin{aligned}
&\quad~\int_{B_{2R}}[S_l(\xi+\mu_{0}y, t)-S_l(\xi+\mu y, t)]Z_{3}(y)dy \\
&\quad= g(t,\frac{\lambda}{\mu_0})\dot{\lambda} + g(t, \frac{\lambda}{\mu_0})\dot{\xi} +g(t,\frac{\lambda}{\mu_0})\lambda + \mu_0^{1+\sigma}f(t),
\end{aligned}
\end{equation*}
where $f$, $g$ are smooth, bounded functions satisfying $g(\cdot, s)\sim s$ as $s\to 0$. Thus
\begin{equation*}
\begin{aligned}
& -\frac{1}{\pi}\int_{B_{2R}}\mu_{0}\Pi_{U^\perp}\mathcal{E}^*_{\mu, \xi}(\xi + \mu_{0}y, t)\cdot Z_3(y)dy\\
&= \dot{\lambda}(t)+\frac{10}{3}\kappa_0\lambda(t)+ e^{-\varepsilon t_0}g(t,\frac{\lambda}{\mu_0})(\dot{\lambda} + \dot{\xi}) + e^{-\varepsilon t_0}g(t, \frac{\lambda}{\mu_0})\lambda,
\end{aligned}
\end{equation*}
where $g$ are smooth, bounded functions and $g(\cdot, s)\sim s$ as $s\to 0$.

Next we compute $\int_{B_{2R}}\frac{2\frac{\mu_0}{\mu}}{1+\left|\frac{\mu_0}{\mu}y\right|^2}\Pi_{\omega^\perp}\psi(\xi + \mu_0 y) Z_{3}(y)dy$. The principal part is $I: = \int_{B_{2R}}\frac{2}{1+\left|y\right|^2}\Pi_{\omega^\perp}\psi(\xi + \mu_0 y) Z_{3}(y)dy$. Recall $\psi = \psi[\lambda,\xi,\dot{\lambda},\dot{\xi}, \phi](y, t)$, we have
\begin{equation*}
\begin{aligned}
I &=\psi[0,q,0,0,0](q, t)\int_{B_{2R}}\frac{2}{1+\left|y\right|^2} Z_{3}(y)dy\\
&\quad + \int_{B_{2R}}\frac{2}{1+\left|y\right|^2}Z_{3}(y)(\psi[0,q,0,0,0](\xi+\mu_{0}y, t)-\psi[0,q,0,0,0](q, t))dy\\
&\quad + \int_{B_{2R}}\frac{2}{1+\left|y\right|^2}Z_{3}(y)(\psi[\lambda,\xi,\dot{\lambda},\dot{\xi}, \phi] - \psi[0,q,0,0,0])(\xi+\mu_{0}y, t)dy\\
& = I_1 + I_2 + I_3.
\end{aligned}
\end{equation*}
By Proposition \ref{p4:4.1}, $I_1 = e^{-\varepsilon t_0}\mu_0(t)^\sigma f(t)$ with $f$ smooth and bounded, $I_2 = e^{-\varepsilon t_0}\mu_0(t)^\sigma g(t, \frac{\lambda}{\mu_0}, \xi -q)$ for a smooth, bounded function $g$ satisfying $g(\cdot, s, \cdot)\sim s$ and $g(\cdot,\cdot, s)\sim s$ as $s \to 0$. By the mean value theorem, we have
\begin{equation*}
\begin{aligned}
I_3 &= \int_{B_{2R}}\frac{2}{1+\left|y\right|^2}Z_{3}(y)\bigg[\partial_\lambda\psi[0,q,0,0,0][s\lambda](\xi + \mu_{0}y, t)\\
&\quad + \partial_\xi\psi[0,q,0,0,0][s(\xi-q)](\xi + \mu_{0}y, t)+ \partial_{\dot{\lambda}}\psi[0,q,0,0,0][s\dot{\lambda}](\xi + \mu_{0}y, t)\\
&\quad + \partial_{\dot{\xi}}\psi[0,q,0,0,0][s\dot{\xi}](\xi + \mu_{0}y, t) + \partial_{\phi}\psi[0,q,0,0,0][s\phi](\xi + \mu_{0}y, t)\bigg]dy
\end{aligned}
\end{equation*}
for some $s\in (0, 1)$.
Using Proposition \ref{p4:4.2},
\begin{equation*}
I_3 = e^{-\varepsilon t_0}\mu_0(t)^\sigma f(t)(\dot{\lambda} + \dot{\xi} + \lambda + \xi)F[\lambda, \xi, \dot{\lambda}, \dot{\xi}, \phi](t),
\end{equation*}
where $f$ is a smooth, bounded function and $F$ is a nonlocal operator with $F[0,q,0,0,0](t)$ bounded. Similarly, set
\begin{equation}\label{e5:105}
\begin{aligned}
&B(\psi[\lambda,\xi,\dot{\lambda},\dot{\xi}, \phi])(\xi+\mu_0 y, t):=-\frac{\mu_0}{\pi}\int_{\mathbb{R}}\frac{\left[\psi(x)\cdot \omega\left(\frac{x-\xi}{\mu}\right)-\psi(s)\cdot \omega\left(\frac{s-\xi}{\mu}\right)\right]\left[\omega\left(\frac{x-\xi}{\mu}\right)-\omega\left(\frac{s-\xi}{\mu}\right)\right]}{|x-s|^2}ds \\
& + \left(\frac{\mu_0}{\pi}\int_{\mathbb{R}}\frac{\left[\psi(x)\cdot \omega\left(\frac{x-\xi}{\mu}\right)-\psi(s)\cdot \omega\left(\frac{s-\xi}{\mu}\right)\right]\left[\omega\left(\frac{x-\xi}{\mu}\right)-\omega\left(\frac{s-\xi}{\mu}\right)\right]}{|x-s|^2}ds\cdot \omega\left(\frac{x-\xi}{\mu}\right)\right)\omega\left(\frac{x-\xi}{\mu}\right),
\end{aligned}
\end{equation}
we have
\begin{equation*}
\begin{aligned}
&\int_{B_{2R}}B(\psi[\lambda,\xi,\dot{\lambda},\dot{\xi}, \phi])(\xi+\mu_0 y, t)Z_{3}(y)dy\\
& =  e^{-\varepsilon t_0}\mu_0(t)^\sigma f(t)+e^{-\varepsilon t_0}\mu_0(t)^\sigma g(t, \frac{\lambda}{\mu_0}, \xi -q)+e^{-\varepsilon t_0}\mu_0(t)^\sigma f(t)(\dot{\lambda} + \dot{\xi} + \lambda + \xi)F[\lambda, \xi, \dot{\lambda}, \dot{\xi}, \phi](t).
\end{aligned}
\end{equation*}

Now, considering the terms $B^1[\phi]$, $B^2[\phi]$ and $B^3[\phi]$, we obtain that
\begin{equation*}
\int_{B_{2R}}B^i[\phi](y, t)Z_{3}(y)dy = e^{-\varepsilon t_0}\mu_0(t)^\sigma g\left(\frac{\lambda}{\mu_0}\right)\ell[\phi](t)
\end{equation*}
for a smooth function $g(s)$ satisfying $g(s)\sim s$ as $s\to 0$ and $\ell[\phi](t)$ is smooth and bounded in $t$.
Combining the above estimates, we conclude the result.
\end{proof}

Similarly, we have
\begin{lemma}\label{l5:2}
For $l = 2$, (\ref{e5:7}) is equivalent to
\begin{equation*}
\dot{\xi} = \Pi_{2}[\lambda, \xi, \dot{\lambda}, \dot{\xi}, \phi](t),
\end{equation*}
\begin{equation}\label{e5:18}
\begin{aligned}
\Pi_{2}[\lambda, \xi, \dot{\lambda}, \dot{\xi}, \phi](t)&= e^{-\varepsilon t_0}\mu_0(t)^\sigma f_2(t) + e^{-\varepsilon t_0}\Theta[\dot{\lambda},\dot{\xi}, \lambda, \mu_0(\xi-q), \mu_0^{\sigma}\phi](t),
\end{aligned}
\end{equation}
where $f(t)$ is a function which is smooth and bounded in $t\in [t_0, \infty)$. The function $\Theta$ has the same properties as Lemma \ref{l5:1}.
\end{lemma}

From Lemma \ref{l5:1} and \ref{l5:2}, we know that
\begin{equation*}
\int_{B_{2R}}H[\lambda,\xi, \dot{\lambda}, \dot{\xi},\phi](y, t(\tau))Z_l(y)dy,
\end{equation*}
for $l = 2, 3$, is equivalent to the following system of ODEs for $\lambda$ and $\xi$
\begin{equation}\label{e5:9}
\left\{
\begin{aligned}
&\dot{\lambda}(t)+\frac{10}{3}\kappa_0\lambda(t) = \Pi_1[\lambda, \xi, \dot{\lambda}, \dot{\xi}, \phi](t),\\
&\dot{\xi} = \Pi_{2}[\lambda, \xi, \dot{\lambda}, \dot{\xi}, \phi](t).
\end{aligned}
\right.
\end{equation}
System (\ref{e5:9}) is solvable for $\lambda$ and $\xi$ satisfying (\ref{e:assumptionsondotlambda})-(\ref{e:assumptionsonlambda}), indeed, we have
\begin{prop}\label{p5:5.1}
There exists a solution $\lambda = \lambda[\phi](t)$, $\xi = \xi[\phi](t)$ to (\ref{e5:9}) satisfying (\ref{e:assumptionsondotlambda}) and (\ref{e:assumptionsonlambda}). For $t\in (t_0, +\infty)$, it holds that
\begin{equation}\label{e5:12}
\mu_0^{-(1+\sigma)}(t)\big|\lambda[\phi_1](t) - \lambda[\phi_2](t)\big|\lesssim e^{-\varepsilon t_0}\|\phi_1 - \phi_2\|_{a, \sigma}
\end{equation}
and
\begin{equation}\label{e5:13}
\mu_0^{-1+\sigma}(t)\big|\xi[\phi_1](t) - \xi[\phi_2](t)\big|\lesssim e^{-\varepsilon t_0}\|\phi_1 - \phi_2\|_{a, \sigma}.
\end{equation}
\end{prop}
\begin{proof}
Let $h$ be a function with $\|h\|_{1+\sigma} < +\infty$. We consider the following nonhomogeneous linear ODE,
\begin{equation*}\label{e5:14}
\dot{\lambda}(t)+\frac{10}{3}\kappa_0\lambda(t) = h(t).
\end{equation*}
The solution can be expressed as
\begin{equation*}\label{e5:15}
\lambda(t) =  e^{-\frac{10}{3}\kappa_0 t}\left[d + \int_{t_0}^t e^{\frac{10}{3}\kappa_0 s} h(s) ds\right],
\end{equation*}
where $d$ is an arbitrary constant. Then we have
\begin{equation*}
\|e^{\kappa_0(1+\sigma)t}\lambda(t)\|_{L^\infty(t_0,\infty)}\lesssim e^{(\sigma-\frac{7}{3})\kappa_0 t}d + \|h\|_{1+\sigma}
\end{equation*}
and
\begin{equation*}
\|\dot{\lambda}(t)\|_{1+\sigma}\lesssim e^{(\sigma-\frac{7}{3})\kappa_0 t}d + \|h\|_{1+\sigma}.
\end{equation*}

Let $\Lambda(t) = \dot{\lambda}(t)$, then
\begin{equation*}\label{e5:17}
\Lambda + \int_{t}^\infty\Lambda(s)ds = h(t),
\end{equation*}
which defines a linear operator $\mathcal{L}_1: h\to \Lambda$ associating to any $h$ with $\|h\|_{1+\sigma}$ bounded the solution $\Lambda$.
$\mathcal{L}_1$ is continuous between the spaces $L^\infty(t_0, \infty)$ equipped the $\|\cdot\|_{1+\sigma}$-topology.

For any $h: [t_0,\infty)\to \mathbb{R}$ with $\|h\|_{1+\sigma}$ bounded, the solution of
\begin{equation*}\label{e5:19}
\dot{\xi} =  h(t)
\end{equation*}
is given by
\begin{equation}\label{e5:20}
\xi(t) = q + \int_{t}^\infty h(s)ds.
\end{equation}
Then we have
\begin{equation*}
|\xi(t) - q|\lesssim e^{-\kappa_0(1+\sigma)t}\|h\|_{1+\sigma}
\end{equation*}
and
\begin{equation*}
\|\dot{\xi}\|_{1+\sigma}\lesssim \|h\|_{1+\sigma}.
\end{equation*}
Let $\Xi(t) = \dot{\xi}(t)$, then (\ref{e5:20}) defines a linear operator $\mathcal{L}_2: h\to \Xi$ which is continuous in the $\|\cdot\|_{1+\sigma}$-topology.

Observe that the tuple $(\lambda, \xi)$ is a solution of (\ref{e5:9}) if $(\Lambda,\Xi) := (\dot{\lambda}, \dot{\xi})$ is a fixed point for problem
\begin{equation}\label{e5:23}
(\Lambda, \Xi) = \mathcal{A}(\Lambda, \Xi)
\end{equation}
where
\begin{equation*}
\mathcal{A}: = \left(\mathcal{L}_1(\hat{\Pi}_1[\Lambda, \Xi, \phi], \mathcal{L}_2(\hat{\Pi}_2[\Lambda, \Xi, \phi])\right) = \left(\bar{A}_1(\Lambda, \Xi), \bar{A}_2(\Lambda, \Xi)\right)
\end{equation*}
with
\begin{equation*}
\hat{\Pi}_1[\Lambda, \Xi, \phi] = \Pi_1[\int_{t}^\infty \Lambda, q + \int_{t}^\infty\Xi, \Lambda, \Xi, \phi], \hat{\Pi}_2[\Lambda, \Xi, \phi] = \Pi_2[\int_{t}^\infty \Lambda, q + \int_{t}^\infty\Xi, \Lambda, \Xi, \phi].
\end{equation*}
Let
\begin{equation*}
K = \max\{\|f_1\|_{1+\sigma}, \|f_2\|_{1+\sigma}\}
\end{equation*}
where $f_1$, $f_2$ are defined in (\ref{e5:11}) and (\ref{e5:18}). We show that problem (\ref{e5:23}) has a fixed point $(\Lambda, \Xi)$ in the set
\begin{equation*}
\mathcal{B} = \{(\Lambda, \Xi)\in L^\infty(t_0, \infty)\times L^\infty(t_0, \infty): \|\Lambda\|_{1+\sigma} + \|\Xi\|_{1+\sigma}\leq cK\}
\end{equation*}
for suitable $c > 0$.
Indeed, from (\ref{e5:9}) we have
\begin{equation*}
\left|e^{\kappa_0(1+\sigma)t}\bar{A}_1(\Lambda, \Xi)\right|\lesssim e^{(\sigma-\frac{7}{3})\kappa_0 t}d + \|\phi\|_{a, \sigma} + K + e^{-\varepsilon t_0}\|\Lambda\|_{1+\sigma} + e^{-\varepsilon t_0}\|\Xi\|_{1+\sigma}
\end{equation*}
and
\begin{equation*}
\left|e^{\kappa_0(1+\sigma)t}\bar{A}_2(\Lambda, \Xi)\right|\lesssim \|\phi\|_{a, \sigma} + K + e^{-\varepsilon t_0}\|\Lambda\|_{1+\sigma} + e^{-\varepsilon t_0}\|\Xi\|_{1+\sigma}.
\end{equation*}
Thus, for $d$ satisfying $e^{(\sigma-\frac{7}{3})\kappa_0 t}d < K$ and the constant $c$ chosen large enough, $\mathcal{A}(\mathcal{B})\subset \mathcal{B}$. For the Lipschitz property of $\mathcal{A}$, we have
\begin{equation*}
\begin{aligned}
e^{\kappa_0(1+\sigma)t}\left|\bar{A}_1(\Lambda_1, \Xi)-\bar{A}_1(\Lambda_2, \Xi)\right| &= e^{\kappa_0(1+\sigma)t}\left|\mathcal{L}_1(\hat{\Pi}_1[\Lambda_1, \Xi, \phi]-\hat{\Pi}_1[\Lambda_2, \Xi, \phi])\right|\\
&\leq e^{\kappa_0(1+\sigma)t}e^{-\varepsilon t_0}\left|\mathcal{L}_1(\Theta(\Lambda_1, \Xi)-\Theta_2(\Lambda_2, \Xi))\right|\\
& \leq e^{-\varepsilon t_0}\|\Lambda_1 - \Lambda_2\|_{1+\sigma}.
\end{aligned}
\end{equation*}
The same estimate holds for $\left|\bar{A}_1(\Lambda, \Xi_1)-\bar{A}_1(\Lambda, \Xi_2)\right|$, thus
\begin{equation*}
\|\mathcal{A}(\Lambda_1, \Xi_1) - \mathcal{A}(\Lambda_2, \Xi_2)\|_{1+\sigma}\leq e^{-\varepsilon t_0}\|\Lambda_1 - \Lambda_2\|_{1+\sigma}.
\end{equation*}
Since $e^{-\varepsilon t_0} < 1$ when $t_0$ is large enough, $\mathcal{A}$ is a contraction map. Hence, from the Contraction Mapping Theorem, there exists a solution $(\lambda,\xi)$ to system (\ref{e5:9}) satisfying (\ref{e:assumptionsondotlambda}) and (\ref{e:assumptionsonlambda}).

To prove (\ref{e5:12}) and (\ref{e5:13}), we observe that $\bar{\lambda} = \lambda[\phi_1] - \lambda[\phi_2]$, $\bar{\xi} = \xi[\phi_1] - \xi[\phi_2]$ satisfy
\begin{equation*}
\dot{\lambda}(t)+\frac{10}{3}\kappa_0\lambda(t) = \bar{\Pi}_1(t),~~\dot{\xi} = \bar{\Pi}_{2}(t),
\end{equation*}
where
\begin{equation*}
\begin{aligned}
\bar{\Pi}_1(t) &= \int_{B_{2R}}\frac{2\frac{\mu_0}{\mu}}{1+\left|\frac{\mu_0}{\mu}y\right|^2}\Pi_{\omega^\perp}\left[\psi[\phi_1]-\psi[\phi_2]\right](\xi + \mu_0 y) Z_{3}(y)dy\\
&\quad + \int_{B_{2R}}(B(\psi[\lambda,\xi,\dot{\lambda},\dot{\xi}, \phi_1])-B(\psi[\lambda,\xi,\dot{\lambda},\dot{\xi}, \phi_2]))(\xi+\mu_0 y, t)Z_{3}(y)dy\\
&\quad + \int_{B_{2R}}(B^i[\phi_1]-B^i[\phi_2])(y, t)Z_{3}(y)dy
\end{aligned}
\end{equation*}
and
\begin{equation*}
\begin{aligned}
\bar{\Pi}_2(t) &= \int_{B_{2R}}\frac{2\frac{\mu_0}{\mu}}{1+\left|\frac{\mu_0}{\mu}y\right|^2}\Pi_{\omega^\perp}\left[\psi[\phi_1]-\psi[\phi_2]\right](\xi + \mu_0 y) Z_{2}(y)dy\\
&\quad + \int_{B_{2R}}(B(\psi[\lambda,\xi,\dot{\lambda},\dot{\xi}, \phi_1])-B(\psi[\lambda,\xi,\dot{\lambda},\dot{\xi}, \phi_2]))(\xi+\mu_0 y, t)Z_{2}(y)dy\\
&\quad + \int_{B_{2R}}(B^i[\phi_1]-B^i[\phi_2])(y, t)Z_{2}(y)dy.
\end{aligned}
\end{equation*}
Here $B(\psi[\lambda,\xi,\dot{\lambda},\dot{\xi}, \phi])$ is defined in (\ref{e5:105}).
Then (\ref{e5:12}) and (\ref{e5:13}) follow from (\ref{e5:104}). This completes the proof.
\end{proof}

\subsection{Final argument}
After we have chosen parameters $\lambda = \lambda[\phi]$ and $\xi = \xi[\phi]$ such that (\ref{e5:7}) holds, we only need to solve problem (\ref{e5:5.1}) in the class of functions with $\|\phi\|_{a,\sigma}$ bounded. Proposition 5.1 states that there exists a linear operator $\mathcal{T}$ associating any function $h(y,\tau)$ with $\|h\|_{1+a, \sigma, \eta}$-bounded the solution of (\ref{e5:5.1}). Thus problem (\ref{e5:5.1}) is reduced to a fixed point problem
\begin{equation*}
\phi = \mathcal{A}(\phi): = \mathcal{T}(H[\lambda,\xi,\dot{\lambda},\dot{\xi},\phi])).
\end{equation*}
We claim the validity of the following estimates,
\begin{equation}\label{e6:1}
(1+|y|^\eta)\left[H[\lambda,\xi,\dot{\lambda},\dot{\xi},\phi](\cdot, t)\right]_{\eta, B_{3R}(0)}\chi_{\{|y|\leq 3R\}} + \left|H[\lambda,\xi,\dot{\lambda},\dot{\xi},\phi](y, t)\right|\lesssim e^{-\varepsilon t_0}\frac{\mu_0^{\sigma}(t)}{1+|y|^{1+a}}
\end{equation}
and
\begin{equation}\label{e6:2}
\begin{aligned}
(1+|y|^\eta)\left[H[\phi^{(1)}](\cdot, t)-H[\phi^{(2)}](\cdot, t)\right]_{\eta, B_{3R}(0)}\chi_{\{|y|\leq 3R\}} &+ \left|H[\phi^{(1)}]-H[\phi^{(2)}]\right|(y, t)\\ &\lesssim e^{-\varepsilon t_0}\frac{\mu_0^{\sigma}(t)}{1+|y|^{1+a}}\|\phi^{(1)} - \phi^{(2)}\|_{a,\sigma}.
\end{aligned}
\end{equation}
From (\ref{e6:1}) and (\ref{e6:2}), $\mathcal{A}$ has a fixed point $\phi$ within the set of functions $\|\phi\|_{a,\sigma}\leq ce^{-\varepsilon t_0}$ for some large positive constant $c$. This proves Theorem \ref{t:main}.

Estimate (\ref{e6:1}) follows from the definition of $H$, (\ref{e3:estimateoferror}), (\ref{e4:pointwiseestimate}) and (\ref{e:estimateholder}).
As for (\ref{e6:2}), from (\ref{e5:12}) and (\ref{e5:13}), we have
\begin{equation*}
\mu_{0}\left|\Pi_{U^\perp}\mathcal{E}^*[\lambda_1,\xi_1](\xi_{1}+\mu_0y, t) - \Pi_{U^\perp}\mathcal{E}^*[\lambda_2,\xi_2](\xi_{2}+\mu_0y, t)\right|\lesssim e^{-\varepsilon t_0}\frac{\mu_0^\sigma(t)}{1+|y|^{1+a}}\|\phi^{(1)} - \phi^{(2)}\|_{a,\sigma}
\end{equation*}
where
\begin{equation*}
\lambda_i = \lambda[\phi^{(i)}],\quad \xi_i = \xi[\phi^{(i)}],\quad i = 1, 2.
\end{equation*}
By Proposition \ref{p4:4.2}, it holds that
\begin{eqnarray*}
&&\left|\frac{2\frac{\mu_0}{\mu_1}}{1+\left|\frac{\mu_0}{\mu_1}y\right|^2}\Pi_{\omega^\perp}\psi[\phi^{(1)}](\xi_{1} + \mu_{0}y, t) - \frac{2\frac{\mu_0}{\mu_2}}{1+\left|\frac{\mu_0}{\mu_2}y\right|^2}\Pi_{\omega^\perp}\psi[\phi^{(2)}](\xi_{2} + \mu_{0}y, t)\right|\\
&&\quad\quad\quad\lesssim e^{-\varepsilon t_0}\frac{\mu_0^{\sigma}(t)}{1+|y|^{1+a}}\|\phi^{(1)} - \phi^{(2)}\|_{a,\sigma}
\end{eqnarray*}
where
\begin{equation*}
\mu_{i} = \mu_0+\lambda[\phi^{(i)}],\quad \psi[\phi^{(i)}] = \Psi[\lambda_i, \xi_i, \dot{\lambda}_i, \dot{\xi}_i, \phi^{(i)}],\quad i = 1, 2.
\end{equation*}
From the definitions in Section 3, for $j = 1, 2, 3$, we have
\begin{equation*}
\left|B^j[\phi^{(1)}]-B^j[\phi^{(2)}]\right|\lesssim e^{-\varepsilon t_0}\frac{\mu_0^{\sigma}(t)}{1+|y|^{1+a}}\|\phi^{(1)} - \phi^{(2)}\|_{a,\sigma}
\end{equation*}
Finally, for $B(\psi[\lambda,\xi,\dot{\lambda},\dot{\xi}, \phi])$ defined in (\ref{e5:105}), the estimate
\begin{eqnarray*}
\left|B(\psi[\phi^{(1)}])(\xi_1+\mu_0 y, t) - B(\psi[\phi^{(2)}])(\xi_2+\mu_0 y, t)\right|\lesssim e^{-\varepsilon t_0}\frac{\mu_0^{\sigma}(t)}{1+|y|^{1+a}}\|\phi^{(1)} - \phi^{(2)}\|_{a,\sigma}
\end{eqnarray*}
hold. The H\"{o}lder part is similar. This proves estimate (\ref{e6:2}).\qed

\end{document}